\documentclass[12pt, reqno]{amsart}

\usepackage{amsmath,amstext,amssymb,amsopn,amsthm}
\usepackage{url,verbatim}
\usepackage{mathtools}
\usepackage{enumerate}

\usepackage{color,graphicx}

\usepackage[margin=30mm]{geometry}
\usepackage{eucal,mathrsfs,dsfont}

\newtheorem{theorem}{Theorem}[section]
\newtheorem{corollary}[theorem]{Corollary}
\newtheorem{lemma}[theorem]{Lemma}
\newtheorem{proposition}[theorem]{Proposition}

\theoremstyle{definition}

\newtheorem{definition}[theorem]{Definition}

\newtheorem{example}[theorem]{Example}

\theoremstyle{remark}
\newtheorem{remark}[theorem]{Remark}

\numberwithin{equation}{section}

\newcommand{\eps}{\varepsilon}

\newcommand{\calL}{\mathcal{L}}
\newcommand{\calF}{\mathcal{F}}

\newcommand{\calA}{\mathcal{A}}

\newcommand{\calP}{\mathcal{P}}
\newcommand{\calR}{\mathcal{R}}
\newcommand{\calT}{\mathcal{T}}

\newcommand{\calC}{\mathcal{C}}
\newcommand{\calB}{\mathcal{B}}

\newcommand{\calE}{\mathcal{E}}
\newcommand{\calH}{\mathcal{H}}
\newcommand{\calK}{\mathcal{K}}

\newcommand{\E}{\operatorname{\mathds{E}}} 
\renewcommand{\P}{\operatorname{\mathds{P}}} 
\newcommand{\R}{\mathds{R}}
\newcommand{\C}{\mathds{C}}
\newcommand{\HH}{\mathds{H}}
 \newcommand{\Q}{\mathds{Q}}

\newcommand{\bfA}{{\mathbf{A}}}

\newcommand{\prt}{\partial}
\newcommand{\lra}{\leftrightarrow}

\newcommand{\ol}{\overline}
\newcommand{\exc}{\mathrm{e}}
\newcommand{\wh}{\widehat}
\newcommand{\wt}{\widetilde}
\newcommand{\cadlag}{c\`adl\`ag }

\DeclareMathOperator{\barg}{{\bf arg}}

\DeclareMathOperator{\dist}{dist}

\DeclareMathOperator{\interior}{Int}

\DeclareMathOperator{\Var}{Var}
\DeclareMathOperator{\Cov}{Cov}

\DeclareMathOperator{\ntlim}{NT-lim}

\def\Re{{\rm Re\,}}
\def\Im{{\rm Im\,}}

\def\bone{{\bf 1}}
\def\n{{\bf n}}
\def\bt{{\bf t}}
\def\bv{{\bf v}}

\def\ball{{\calB}}
\def\bDelta{{\bf \Delta}}
\def\Tkb{T^k_b}

\title[Obliquely reflected Brownian motion]{Obliquely reflected Brownian motion in non-smooth planar domains}

\author{Krzysztof Burdzy,    Zhen-Qing Chen,  Donald Marshall  and Kavita Ramanan}

\address{KB, ZC, DM: Department of Mathematics, Box 354350, University of Washington, Seattle, WA 98195}
\email{burdzy@math.washington.edu}
\email{zqchen@uw.edu}
\email{marshall@math.washington.edu}

\address{KR: Division of Applied Mathematics,
Brown University,
182 George Street,
Providence, RI 02912}
\email{Kavita\_Ramanan@brown.edu}

\thanks{Research supported in part by NSF Grants
DMS-1206276,  DMS-0900814 and DMS-1407504}

\begin{document}

\begin{abstract}  
 We construct obliquely reflected Brownian motions in all bounded
 simply connected planar domains, including non-smooth domains, with general reflection
vector fields on the boundary.  Conformal mappings and excursion theory
are our main technical tools. 
A key intermediate step, which may be of independent interest,
is an alternative characterization of reflected Brownian
motions in smooth bounded planar domains with a given field of angles of oblique
reflection on the boundary in terms of a pair of quantities, namely an
integrable positive harmonic function, which represents
the stationary distribution of the process, and a real number that
represents, in a suitable sense, the
asymptotic rate of rotation of the process around a reference point in the domain.
Furthermore, we also show that any obliquely reflected Brownian motion in a
simply connected Jordan domain can be obtained as a suitable limit of
obliquely reflected Brownian motions in smooth domains. 
\end{abstract}

\maketitle

\section{Introduction}\label{intro}

Obliquely reflected Brownian motion (ORBM) arises naturally in some
applied probabilistic models, for example, in queuing theory; see
\cite{Ram06,Will} and the references therein. 
This part of the theory of ORBMs is mostly concerned with processes confined to the positive quadrant of the Euclidean space
 with constant reflection direction on each face.
ORBMs in non-smooth (fractal) domains serve as a toy model for some biological phenomena (see \cite{reduction}). In this paper, we will construct and investigate  ORBMs in bounded simply connected planar domains, including non-smooth domains, with variable and possibly non-smooth reflection directions. Conformal mappings will be our main technical tool. The construction of ORBM in non-smooth domains is difficult because the process (if it exists) is non-symmetric and, therefore, the (symmetric) Dirichlet form approach (see \cite{Fuku, Chen} and the references therein), very successful in the case of normally reflected Brownian motion, is not applicable to ORBM with general non-smooth reflection directions.

A conceptual problem with obliquely reflected Brownian motion is that the oblique
reflection represents, in heuristic terms, a slight push away from the boundary
accompanied by a proportional push along the boundary. In fractal domains, the
concepts of ``normal'' direction at a boundary point and moving ``along'' the
boundary do not have a meaning according to classical definitions. Hence
describing and classifying ORBMs in non-smooth domains requires a new approach.
The key to our study is the observation that ORBMs in smooth domains can be fully and uniquely classified using
two ``parameters''---an integrable positive harmonic function $h$ and a real
number $\mu_0$.  The
harmonic function $h$ represents the density of the stationary distribution of the
process and the real number $\mu_0$ represents, in an appropriate sense, the asymptotic rate of rotation around a reference point in the domain.
This alternative characterization of ORBM will allow us to construct and investigate ORBM in non-smooth planar domains with general reflection on the boundary.
More specifically,
we will first show in Theorems \ref{f3.2} and  \ref{j15.5} that $h$ and $\mu_0$ provide a parametrization of ORBMs in the unit disc alternative to the reflection vector field on the boundary. Then we will show  in Theorems \ref{j18.3}-\ref{j18.8}  how ORBMs in non-smooth domains can be constructed and classified.

Yet another ``parametrization'' of ORBM's in simply connected domains is given by ``rotation rates'' $\mu(z)$ of the process around points $z$ in the domain. Every function $\mu(z)$ representing rotation rates is harmonic but not every harmonic function $\mu(z)$ represents rotation rates for an ORBM.

We will also discuss some ORBMs with degenerate (``tangential'') ``reflection" along the boundary. The infinitely strong tangential push generates jumps along the boundary, a feature not normally associated with models labeled ``Brownian.''
We will show that ORBMs with ``degenerate'' boundary behavior are processes that recently appeared in the probabilistic literature in a different context.

The present paper can be viewed as a first step in a much more
ambitious project to define ORBMs in $d$-dimensional non-smooth
domains with $d\geq 2$. In the two-dimensional case, especially in
simply connected domains, one can give a meaning to the ``angle of
reflection'' even in domains with fractal boundary by approximating
the boundary with continuous curves, defining the angle of reflection
on these curves, then defining the corresponding ORBMs and finally
passing to the limit (see Theorem \ref{j18.8} below). The same program
is questionable in higher dimensional domains. It is not clear how to
define the direction of reflection on a fractal boundary or how to
define the direction of reflection on a sequence of approximating
smooth surfaces in a ``consistent'' way. We believe that our approach
via the stationary density  (see \cite{KanRam} for a
characterization of stationary distributions of ORBMs in
$d$-dimenesional piecewise smooth domains) and appropriate ``rotations about
$(d-2)$-dimensional sets'' may be the right approach to the
high-dimensional version of the problem but we leave it for a future project.

There are two classes of domains to which some of our results should extend in a fairly straightforward way: unbounded simply connected planar domains and finitely connected bounded planar domains. These generalizations are also left for a future article.

Some results for ORBM in multidimensional domains were obtained in
\cite{DupIsh, DupIshCorr, Ram06, Will} under rather restrictive assumptions about smoothness of the boundary of the domain and/or the direction of reflection.
The theory of non-symmetric Dirichlet forms was used to construct families of ORBMs in \cite{KKY,duarte} under fairly strong assumptions.
A fairly explicit formula for the stationary distribution for ORBM in
a smooth planar domain was derived in \cite{HLS}. 
Some results on convergence of ORBMs have been recently obtained in \cite{Saran2, Saran1} but the setting of those papers is considerably different from ours.

The article is organized as follows. Section \ref{sec:prelim} contains a review of some basic probabilistic and analytic facts used in the article. It also contains a theorem relating reflection vector fields on the boundary of a domain and harmonic functions inside
the domain; this theorem is the fundamental analytic ingredient of our arguments. Our main results are stated in Section \ref{sec:main}. Their proofs are given in Section \ref{sec:proofs}.
Our proofs are based in part on ideas developed in \cite{BurMar}.

\section{Preliminaries}\label{sec:prelim}

\subsection{Reflected Brownian motion}

We will identify $\C$ and $\R^2$.
Let $\ball(x,r) = \{z\in \R^2: |x-z| < r\}$ and $D_* = \ball(0,1)$.
Suppose that $D\subset \C$ is a bounded open set with smooth boundary and $\theta: \prt D \to (-\pi/2, \pi/2)$ is a Borel measurable function satisfying $\sup_{x\in \prt D} |\theta(x)| < \pi/2$.
Let $\n(x)$ denote the unit
inward normal vector at $x\in\prt D$ and let $\bt(x) = e^{-i \pi/2} \n(x)$ be the unit vector tangent to $\prt D$ at $x$.

Let $\bv_\theta(x) = \n(x) + \tan \theta(x) \bt(x)$, let $B$ be standard
two-dimensional Brownian motion and consider the following
Skorokhod equation,
\begin{align}
 X_t = x_0 + B_t + \int_0^t  \bv_\theta (X_s) dL_s,
 \qquad \hbox{for } t\geq 0. \label{j13.1}
\end{align}
Here $x_0\in \ol D$ and $L$ is the local time of $X$ on $\prt
D$. In other words, $L$ is a non-decreasing continuous 
process that does not increase when $X$ is in $D$, i.e.,
$\int_0^\infty \bone_{D}(X_t) dL_t = 0$, almost surely. If $\theta$ is $C^2$ then equation
\eqref{j13.1} has a unique pathwise solution $(X,L)$ such
that $X_t \in \ol D$ for all $t\geq 0$, by \cite[Cor. 5.2]{DupIsh}
(see also \cite{DupIshCorr}). The process $X$ is a continuous strong
Markov process on $\ol D_*$, and is called obliquely reflected 
Brownian motion in $D$ with reflecting vector field  $\bv_\theta$. 
When $\theta \equiv 0$, that is, when $\bv_\theta = \n$, $X$ is called normally reflected Brownian motion in $D$.
The goal of this paper is to construct and characterize obliquely reflected Brownian motions when $\theta$ is non-smooth and can possibly take values in $[-\pi/2, \pi/2]$, and when $\prt D$ is also possibly non-smooth.

Consider the case when $D=D_*$ and recall that we are assuming that $\theta$ is measurable and
$\Vert \theta \Vert_{\infty} < \pi/2.$
Then one can show that
\eqref{j13.1} has a unique pathwise solution using the decomposition of the process in $D_*$
into the radial and angular parts, and an argument similar to that in
\cite[Remark 4.2 (ii)]{LS}.
In both cases discussed above, the ORBM $X$ is a strong Markov process.
Since $X$ does not visit the origin as it behaves like a Brownian motion
inside the disk $D_*$, applying
It\^o's formula to  $Y_t=f(X_t)$ with $f(x)=|x|$,  we obtain
\begin{equation}\label{e:2.31}
d Y_t = dW_t + \frac{1}{Y_t} dt - dL_t ,
\end{equation}
where  $W_t=\int_0^t \frac{X_s}{|X_s|}\cdot dB_s$ is a one-dimensional Brownian motion. Note that $L_t$ increases only when $Y_t=1$. Thus $Y_t$ is a 2-dimensional Bessel process in $(0, 1]$ reflected at $1$. It is known (see \cite{BC}) that the one-dimensional SDE \eqref{e:2.31} has a unique strong solution and all its weak solutions have the same distribution.
It follows that  the distribution of $(|X|, L)$ is independent of the
reflection angle $\theta$. Theorem \ref{j15.5} proved below implies
that this  property continues to hold for ORBMs in $D_*$ with non-smooth reflection angles 
$\theta$ including those that could be tangential in some subset of the boundary
$\partial D_*$. 

It is known that (see Theorem \ref{j25.1}(ii) below) the submartingale problem formulation of ORBM is equivalent to the one given above.
Let $\calC$ be the family of all real functions $f \in C^2(\ol D)$ such that
\begin{align*}
\frac \prt{\prt \n}f(x) + \tan \theta(x) \frac \prt {\prt \bt} f(x) \geq 0,
\qquad x\in \prt D.
\end{align*}
We will say that $\{\P_z: z\in \ol D\}$ is a solution of the submartingale problem
defining an ORBM with the angle of reflection $\theta$ if $\P_z(X_0 = z) = 1$ for every $z\in \ol D$, and
\begin{align}\label{f3.1}
f(X_t) - \frac 12 \int_0^t \Delta f (X_s) ds , \qquad t\geq 0,
\end{align}
is a submartingale under $\P_z$ for every $z\in \ol D$ and $f\in \calC$.

\subsection{Review of excursion theory}\label{sec:exc}

We will use excursion theory of Brownian motion in our characterization of obliquely reflected Brownian motion.
This section contains a brief review of the excursion theory needed
in this paper. See, for example, \cite{Mais} for the foundations of the
theory in the abstract setting and \cite{BurBook} for the special
case of excursions of Brownian motion. Although \cite{BurBook} does
not discuss reflected Brownian motion, all of the results we will use from
that book readily apply in the present context.

Let $\P_{x}$ denote the distribution of the process $X$
with $X_0=x$, defined by \eqref{j13.1} or \eqref{f3.1}, and let $\E_{x}$ be the
corresponding expectation. Let $\P_x^D$ denote the distribution
of Brownian motion starting from $x\in D$ and killed upon
exiting $D$.

An ``exit system'' for excursions of an ORBM $X$ from $\prt D$ is a pair $(L^*_t, H^x)$ consisting of
a positive continuous additive functional $L^*_t$ of $X$
and a family
of ``excursion laws'' $\{H^x\}_{x\in\prt D}$.
 Let $\bDelta$ denote the ``cemetery''
point outside $\ol D$ and let $\calC$ be the space of all
functions $f:[0,\infty) \to \ol D\cup\{\bDelta\}$ that  are
continuous and take values in $\ol D$ on some interval
$[0,\zeta)$, and are equal to $\bDelta$ on $[\zeta,\infty)$.
For $x\in \prt D$, the excursion law $H^x$ is a $\sigma$-finite
(positive) measure on $\calC$, such that the canonical process is
strong Markov on $(t_0,\infty)$, for every $t_0>0$, with 
transition probabilities
$\P_\centerdot^D$.
 Moreover, $H^x$ gives zero
mass to paths that do not start from $x$. We will be concerned
only with the ``standard'' excursion laws; see Definition 3.2
of \cite{BurBook}. For every $x\in \prt D$ there exists a unique
standard excursion law $H^x$ in $D$, up to a multiplicative
constant.

Excursions of $X$ from $\prt D$ will be denoted $\exc$ or $\exc_s$,
i.e., if $s< u$, $X_s,X_u\in\prt D$, and $X_t \notin \prt D$
for $t\in(s,u)$ then $\exc_s = \{\exc_s(t) = X_{t+s} ,\,
t\in[0,u-s)\}$, $\zeta(\exc_s) = u -s$ and $\exc_s(t)
= \bDelta$ for $t\geq \zeta$. By convention, $\exc_t \equiv \bDelta$ if
$\inf\{s> t: X_s \in \prt D\} = t$.

Let $\sigma_t = \inf\{s\geq 0: L^*_s  >  t\}$ and
$\calE_u = \{\exc_s: s < \sigma_u\}$.
Let $I$ be
the set of left endpoints of all connected components of $(0,
\infty)\setminus \{t\geq 0: X_t\in \partial D\}$. The following
is a special case of the exit system formula of \cite{Mais}. For
every $x\in \ol D$, every bounded predictable process $V_t$ and
every universally measurable function $f:\, \calC\to[0,\infty)$
that vanishes on
excursions $\exc_t$ identically equal to $\bDelta$, we have
\begin{align}\label{exitsyst}
\E_x \left[ \sum_{t\in I} V_t \cdot f ( \exc_t) \right]
= \E_x \left[\int_0^\infty V_{\sigma_s} H^{X(\sigma_s)}(f) ds\right]
= \E_x \left[\int_0^\infty V_t H^{X_t}(f) dL^*_t\right].  
\end{align}
Here and
elsewhere $H^x(f) = \int_\calC f dH^x$. Informally speaking,
\eqref{exitsyst} says that the right continuous version
$\calE_{t+}$ of the process of excursions is a Poisson
point process on the local time scale with variable intensity $H^{{}^\centerdot}(f)$.

The normalization of the exit system is somewhat arbitrary, for
example, if $(L^*_t, H^x)$ is an exit system and
$c\in(0,\infty)$ is a constant then $(cL^*_t, (1/c)H^x)$ is
also an exit system. One can even make $c$ dependent on
$x\in\prt D$. Theorem 7.2 of \cite{BurBook} shows how to choose a
``canonical'' exit system; that theorem is stated for the usual
planar Brownian motion but it is easy to check that both the
statement and the proof apply to  normally reflected Brownian motion (i.e., ORBM with $\theta \equiv 0$).
According to that result, if $D$ is Lipschitz then we can take $L^*_t$ to be the
continuous additive functional $L^X$ whose Revuz measure is a
constant multiple of the surface area measure $dx$
on $\prt D$ and
$H^x$'s to be standard excursion laws normalized so that
\begin{align}\label{eq:M5.2}
H^x (A) = \lim_{\delta\downarrow 0}
\frac 1  \delta \P^D_{x + \delta\n(x)} (A),
\end{align}
for any event $A$ in a $\sigma$-field generated by the process
on an interval $[t_0,\infty)$, for any $t_0>0$. The Revuz
measure of $L^X$ is the measure $dx/(2|D|)$ on $\prt D$, i.e.,
if the initial distribution of $X$ is the uniform probability
measure $\mu$ on $D$,
then
\begin{align}\label{feb19.2}
\E_\mu \left[\int_0^1 \bone_A (X_s) dL^X_s\right]
= \int_A \frac{dx}{2|D|},
\end{align}
for any Borel set $A\subset \prt D$.
It has
been shown in \cite{BCJ} that $L^*_t=L^X_t$.

\medskip

Let $K_x(\,\cdot\,)$ denote the Poisson kernel for $D_*$, that is, $K_x(\,\cdot\,)$ vanishes continuously on $\prt D_*\setminus \{x\}$ and is harmonic and strictly positive in $D_*$. We normalize $K_x$ so that $K_x(0) =1$ for all $x$. It is easy to see that the following
equality holds up to a multiplicative constant,
\begin{align}\label{jan16.1}
\int_{A} K_x(y) dy
= \lim_{\delta\downarrow 0}
\frac 1  \delta
\E^{D_*}_{x + \delta\n(x)} \left[\int_0^\infty \bone_A (X_s) ds\right], \qquad A\subset D_*.
\end{align}
In view of \eqref{eq:M5.2}, this means that $K_x(\,\cdot\,)$ is (a constant multiple of) the density of the expected occupation measure for the excursion law $H^x$, i.e.,
\begin{align}\label{jan17.1}
\int_{A} K_x(y) dy
=
H^x\left( \int_0^\infty \bone_A (X_s) ds\right), \qquad A\subset D_*.
\end{align}
We omitted the multiplicative constant in \eqref{jan16.1} and \eqref{jan17.1} because it is equal to 1; see the proof of Theorem \ref{j18.1} (ii).

\subsection{Analytic preliminaries}

Recall that $\ball(x,r) = \{z\in \R^2: |x-z| < r\}$ and $D_* := \ball(0,1)$.
Let $\theta: \prt D_* \to [-\pi/2, \pi/2]$ be a Borel measurable function.
Typically, $|dx|$ will
refer to the arc length measure on $\prt D_*$ and $dz$ will refer to the
two-dimensional Lebesgue measure on $ D_*$.
The notation $|A|$ will represent either the arc length measure of $A\subset\prt D_*$ or the two-dimensional Lebesgue measure of $A\subset D_*$; the meaning should be clear from the context.
Let $\|\,\cdot\,\|_{L^1(D)}$ denote the $L^1$ norm for real functions on an open bounded set $D$ with
respect to two-dimensional Lebesgue measure $dz$ on $D$ and let $L^1(D)$ be the family of real functions in $D$ with finite $L^1$ norm. We will abbreviate $\|\,\cdot\,\|_{L^1(D_*)}$ as $\|\,\cdot\,\|_1$.
Similar conventions will apply to $L^\infty=L^\infty(\partial D_*)$ with respect to
the measure $|dx|$ on $\prt D_*$.
As usual, we identify functions that are equal to
each other a.e. $|dx|$ on $\prt D_*$.

For a function $f$ and constant $c$, the notation $f\not\equiv c$ will mean that $f$ is not identically equal to $c$. If $f$ is harmonic and non-negative in $D_*$ then
\begin{align*}
\|f\|_1 = \int_0^1 \int_0^{2\pi} f(r e^{it}) dt~r dr = \pi f(0).
\end{align*}
If the non-tangential limit of $f(z)$ at $x\in \prt D_*$ exists, we
denote it by $\ntlim_{z\to x} f(z)$. If $f\in L^1(\prt D_*)$ then the
harmonic extension of $f$ to $D_*$, given by the Poisson integral, has
nontangential limits equal to $f$ a.e.. We will follow the usual
convention of using the same letter $f$ to denote the harmonic
extension. If $f$ is harmonic in $D_*$, let $\wt f$ denote the harmonic conjugate of $f$
that vanishes at $0$.

\medskip

Define
\begin{align*}
\calT&=\{\theta\in L^{\infty}(\prt D_*): \|\theta\|_{\infty} \le
\pi/2,~ \theta \not\equiv \pi/2, \text{ and } ~\theta\not\equiv
-\pi/2\},\\\\
\calB &=\{\theta: \theta \text{   is harmonic in } D_*  \text{ and }
|\theta(z)| < \pi/2 \text{  for all } z\in D_*\}, \\\\
\calH &= \{ (h,\mu_0): h \text{ is harmonic in } D_*, ~h(z) > 0
\text{ for all }
z\in D_*, ~\|h\|_1 = \pi h(0)= 1 \text{ and }  \mu_0 \in
\R\},\\
&\text{and }\\
\calR&=\{\mu: \mu \text{ is harmonic in } D_* \text{ and its harmonic conjugate } \wt \mu(z)
> -1 \text{ for all } z\in D_*\}.\\
\end{align*}

The following theorem relates these spaces. See \eqref{ju21.21},
\eqref{aug28.2}, and Corollary \ref{aug28.3} for additional formulae.

\begin{theorem}\label{aug27.0} There are one-to-one correspondences
\begin{align*}
\calT \leftrightarrow \calB, \quad\quad&\theta(x) \leftrightarrow \theta(z);\\
\calH \leftrightarrow \calR, \quad\quad&(h(z),\mu_0) \leftrightarrow \mu(z);\\
\calB \leftrightarrow \calH, \quad\quad&\theta(z) \leftrightarrow (h(z),\mu_0);
\end{align*}
given by
\begin{align}
\theta(z)&=\Re \int_{\prt D_*}\frac{x+z}{x-z} \theta(x)
\frac{|dx|}{2\pi},\label{aug27.1}\\
\theta(x)&= \ntlim_{z\to x} \theta(z) ~~a.e.~|dx|,\label{aug27.2}\\
\mu(z)&= \mu_0-\pi \wt{h}(z),\label{aug27.3}\\
h(z)&=(\wt{\mu}(z) +1)/\pi \text{ and }\mu_0=\mu(0),\label{aug27.4}\\
h(z)&=  \frac{e^{\wt \theta(z)}\cos \theta(z)}{\pi \cos \theta(0)}
\ \text{ and } \  \mu_0=\tan \theta(0),\label{aug27.5}\text{ and }\\
\theta(z)&=-\arg(h(z)+i \wt{h}(z) -i \mu_0/\pi). \label{aug27.6}
\end{align}
\text{ Moreover }
\begin{align}
\mu(z)&=\pi h(z)\tan \theta(z)
=\frac{1}{2}
\lim_{r\uparrow 1} \int_{|x|=r} \Re\biggl(\frac{x+z}{x-z}\biggl) h(x)
\tan \theta(x) |dx|\label{aug27.7}
\end{align}
\text{and}
\begin{align}
\theta(z)=-\arg(h(z)-i\mu(z)/\pi).\label{sept13.1}
\end{align}
\end{theorem}

\begin{proof}

The subject of analytic and harmonic functions on
the disk and their boundary values
has a long history. An eminently readable reference  for
background material on this subject
is  given in
the first three introductory chapters of \cite{Hoff}.

Non-tangential limits give the correspondence between $\calT$ and
$\calB$.
If $\theta\in \calB$, then $\theta$ has a non-tangential limit at
almost every $x\in \partial D_*$, which we will call $\theta(x)$.
The limit function $\theta(x) \in L^{\infty}(\partial D_*)$, and $\|\theta\|_{\infty}
\le \pi/2$.
Moreover, since $\frac{1}{2\pi}\Re \frac{x+z}{x-z}$
is the Poisson kernel on $\prt D_*$ for $z\in D_*$,
we have that
\begin{align}
\theta(z)+i \wt \theta(z) = \int_{\partial D_*} \frac{x+z}{x-z} \theta(x)
\frac{|dx|}{2\pi}.\label{ju21.12}
\end{align}
In fact if $\theta$ is any function in $L^{\infty}$ bounded by
$\pi/2$ then the right-hand
side \eqref{ju21.12}
defines an analytic function on $D_*$ whose real part is harmonic on
$D_*$, bounded by $\pi/2$ and has non-tangential limit function
$\theta(x)$, a.e.
Since $\int_{\prt D_*} \theta(x) \frac{|dx|}{2\pi}=\theta(0)$,
we have $\theta(x)\not\equiv \pi/2$ and $\theta(x)\not\equiv-\pi/2$
a.e. if and only if $|\theta(0)|<\pi/2$ and by the maximum principle,
this occurs if and only if $|\theta(z)|< \pi/2$ for all $z\in D_*$.

If $(h,\mu_0)\in \calH$ then $\mu$ defined by \eqref{aug27.3}
is harmonic on $D_*$, with $\mu(0)=\mu_0$, and
$h(z)=(\wt \mu(z) +1)/\pi$, since $\pi h(0)=1$ and $\wt{\wt
h}=h(0)-h$. Since $h > 0$, we conclude that $\wt \mu > -1$ and $\mu \in \calR$.
If $\mu\in \calR$, and if $h$ is given by \eqref{aug27.4} then it is
easy to verify that $(h,\mu_0)\in \calH$.
This proves the one-to-one correspondence between
functions in $\calH$ and $\calR$.

\medskip
The proof for the correspondence between $\calB$ and $\calH$,
\eqref{aug27.7}-\eqref{sept13.1},
 as well as  useful formulae for the corresponding harmonic conjugates
 are presented in the next two lemmas.

\begin{lemma}\label{ju20.1}
There is a one-to-one correspondence between $\calB$ and $\calH$, $\theta
\leftrightarrow (h,\mu_0)$, given by
\begin{align}
\theta + i \wt \theta &= i \log(h+i \wt h -i\mu_0/\pi) -i
\log\left(\left(\sqrt{1+\mu_0^2}\right)/\pi\right) \ \ \text{  and   }\label{ju20.3} \\
\cr
h+i \wt h &= \frac{e^{-i(\theta + i \wt \theta)}}{\pi \cos \theta(0)}
+i \frac{\tan \theta(0)}{\pi}, \ \text{ and } \  \mu_0=\tan\theta(0).
\label{ju20.4}
\end{align}
\end{lemma}

\medskip
\begin{proof}
If $(h,\mu_0)\in \calH$ then the right-hand side of \eqref{ju20.3}
defines an analytic function  $S(h,\mu_0)(z)$  on $D_*$ with
\begin{align*}
\Re S(h,\mu_0)(z)=-\arg(h+i \wt h -i\mu_0/\pi) \in (-\pi/2,\pi/2)
\end{align*}
and $S(h,\mu_0)(0)=-\arg(1-i\mu_0)$, which is purely real. Thus
$S(h,\mu_0)=\theta+i \wt \theta$ for some $\theta\in \calB$. Likewise,
if $\theta\in \calB$ then the right-hand side of the first equation in
\eqref{ju20.4} defines an analytic function, $T(\theta)(z)$, on $D_*$ with
$\Re T(\theta)(z)=e^{\wt{\theta} (z)}\cos{\theta(z)}/(\pi
\cos{\theta(0)})
> 0$ and $\Re T(\theta)(0)=1/\pi$. Setting $\mu_0=\tan\theta(0)$ we
conclude that
if $h\equiv\Re T(\theta)$ then
$(h,\mu_0)\in \calH$. Moreover it is straightforward to verify
that, given $(h,\mu_0)\in \calH$, if $\theta$ is defined by \eqref{ju20.3}
then
\begin{align*}
h=\Re T(\theta) \quad \hbox{and} \quad \mu_0=\tan \theta(0).
\end{align*}
Alternatively, given $\theta\in \calB$, if $(h,\mu_0)$ is defined by
\eqref{ju20.4} then
\begin{align*}
\theta=\Re S(h,\mu_0).
\end{align*}
This proves the one-to-one correspondence in Lemma \ref{ju20.1}.
\end{proof}

\medskip

The equality in \eqref{sept13.1} of Theorem \ref{aug27.0} follows
immediately from \eqref{aug27.6}
and \eqref{aug27.3}.
The first equality in \eqref{aug27.7} of Theorem \ref{aug27.0} follows
by taking real and imaginary parts in \eqref{ju20.4}, then applying
\eqref{aug27.3}.
The second equality  in \eqref{aug27.7} follows from the Poisson
integral formula on the circle of radius $r < 1$
because $\mu$ is harmonic by \eqref{aug27.3}.

This completes the proof of Theorem \ref{aug27.0}.
\end{proof}

The next lemma
relates $\mu\in \calR$ to both $h$ and $\theta$ via a Mobius transformation. It will be used
in the proof of Theorem \ref{j15.7}.

\begin{lemma}\label{ju21.1}
Suppose $(h,\mu_0)\in \calH$, $\theta\in \calB$, and $\mu \in \calR$ with
$(h,\mu_0) \leftrightarrow \theta \leftrightarrow \mu$.
If $\phi$ is a one-to-one analytic map of $D_*$ onto $D_*$ then
\begin{align}
\theta\circ \phi \in \calB \leftrightarrow \biggl(\frac{h\circ\phi}{\|h\circ \phi\|_1},
\frac{\mu(\phi(0))}{\|h\circ \phi\|_1}\biggl)\in \calH. \label{ju21.2}
\end{align}
\end{lemma}

\begin{proof}
First observe that if $f$ is harmonic then
$(f+i\wt f)\circ \phi -i \wt f(\phi(0))$ is analytic with imaginary
part vanishing at $0$, so that
\begin{align}
\wt{f\circ\phi} = \wt f \circ\phi -\wt f (\phi(0)).\label{ju21.3}
\end{align}
Evaluating the real part of  \eqref{ju20.4} at $z=\phi(0)$ we obtain
\begin{align}
\|h\circ \phi\|_1=\pi h(\phi(0))=\frac{e^{\wt \theta(\phi(0))}\cos \theta(\phi(0))}{
\cos \theta(0)}.\label{ju21.4}
\end{align}
Set $h_1={h\circ\phi}/{\|h\circ \phi\|_1}={h\circ\phi}/{\pi h(\phi(0))}$.
Then composing \eqref{ju20.4} with $\phi$ and using \eqref{ju21.3} and
\eqref{aug27.3},
\begin{align*}
h_1+i \wt h_1 &=
\frac{h\circ \phi+i\wt h\circ \phi -i \wt h (\phi(0))}{\| h\circ \phi\|_1} \\
&=
\frac{\exp ({-i(\theta+i\wt\theta)\circ\phi})}{\|h\circ \phi\|_1\pi
\cos{\theta(0)}} +\frac{i}{\pi}\biggl(\frac{\tan\theta(0)-\pi \wt h
(\phi(0))}{\|h\circ \phi\|_1}\biggl)\\
&=\frac{\exp ({-i(\theta\circ\phi +i \wt{\theta\circ\phi})})}{\pi \cos\theta(\phi(0))}+
\frac{i\mu(\phi(0))}{\pi \|h\circ\phi\|_1}.
\end{align*}
\end{proof}

\medskip
By \eqref{ju20.4} and \eqref{aug27.1} the correspondence between
$(h,\mu_0)\in \calH$, $\mu\in \calR$,  and $\theta\in \calT$
can also be written as
\begin{align}\label{ju21.21}
h(z)&=\Re\left(\frac{\exp\left(-i\int_{\partial
D_*}\frac{x+z}{x-z}\theta(x)\frac{|dx|}{2\pi}\right)}
{\pi \cos
\left(\int_{\partial D_*} \theta(x) \frac{|dx|}{2\pi}\right)}\right) \ \ \hbox{ and
} \ \ \mu_0=\tan\biggl(\int_{\partial D_*} \theta(x)
\frac{|dx|}{2\pi}\biggl),\\
\label{aug28.2}\mu(z)&=-\pi \Im \left(\frac{ \exp\left(-i\int_{\partial
D_*}\frac{x+z}{x-z}\theta(x)\frac{|dx|}{2\pi}\right)}{\pi \cos
\left(\int_{\partial D_*} \theta(x) \frac{|dx|}{2\pi}\right)}\right).
\end{align}

\medskip

We would like to have a similar formula for $\mu$ and $\theta$ in terms of $h$, but
the situation is a little more complicated for boundary values of positive harmonic
functions. A function $h$ is positive and harmonic on $D_*$ if and
only if
\begin{align}
h(z)=\int_{\partial D_*} \Re\biggl(\frac{x+z}{x-z}\biggl)
\sigma (dx) ,\label{ju21.13}
\end{align}
for some positive finite (regular Borel) measure $\sigma$ on $\partial D_*$.
The measures $h(rx)|dx|$ converge
weakly to $\sigma (dx)$ as $r\uparrow 1$. The function $h$ has
a non-tangential limit at almost every $x\in \partial D_*$, which we
will call $h(x)$, but $h(z)$ is not necessarily the Poisson integral
of $h(x)$. In fact $h\to +\infty$ radially $\sigma_s$-a.e.,  where $\sigma_s$ is
the singular component of the Radon-Nikodym decomposition of $\sigma$
with respect to the length measure $|dx|$ on $\partial D_*$.
It is true, however, that
a harmonic function $f$
has non-tangential limits $f(x)$ a.e.  and satisfies
\begin{align}
f(z)+i \wt f (z)=\int_{\partial D_*} \frac{x+z}{x-z} f(x) \frac{|dx|}{2\pi}
\label{aug28.5}
\end{align}
if and only if
\begin{align}
\lim_{r\uparrow 1} \int_{\partial D_*}
|f(rx)-f(x)|\, |dx|=0.\label{ju21.15}
\end{align}
Given a function $f$ defined on $\prt D_*$ which is integrable $|dx|$,
if we define $f(z)$ for $z\in D_*$ via \eqref{aug28.5} then $f$ satisfies
\eqref{ju21.15}. See \cite[pages 32 and 33]{Hoff}.

If  for some $p > 1$,
\begin{align}
\sup_{r<1} \int_{\prt D_*} |f(rx)|^p |dx| < \infty, \label{aug28.4}
\end{align}
or if
\begin{align*}
\sup_{r<1} \int_{\prt D_*} |(f+ i \wt{f})(rx)| |dx| < \infty
\end{align*}
then \eqref{ju21.15} holds.
See \cite[pages 33 and 51]{Hoff}.

\begin{example}\label{aug28.1}
A good example to keep in mind is
\begin{align}\label{jul12.1}
h(z)=\frac{1}{\pi}\Re\biggl(\frac{1+z}{1-z}\biggl).
\end{align}
Then $h(x)=0$ for $x\in \prt D_*\setminus\{1\}$. So $h$ cannot be the Poisson
integral of its boundary values. Nevertheless, if
$\theta\leftrightarrow (h,0)$ then since $\theta$ is bounded, it
satisfies \eqref{aug28.4} and
hence satisfies \eqref{ju21.15}.
In fact,
$\theta(x)= -\pi/2$ for $x\in \prt D_*$ with $\Im x >0$ and
$\theta(x)= \pi/2$ for $x\in \prt D_*$ with $\Im x <0$,
so that
\begin{align*}
\theta(z)+i \wt \theta (z) ={i}\log\frac{1+z}{1-z} =\int_{\prt D_*} \frac{x+z}{x-z}\theta(x)|dx|/(2\pi).
\end{align*}
\end{example}

If $h$ satisfies \eqref{ju21.15}, where
$(h,\mu_0)\in \calH \leftrightarrow \theta \in \calB$, then we can
recover $\theta$
directly from the boundary values of $h$ and $\mu_0$.
A similar result holds for $\mu$. The following corollary will be used later to
interpret $\mu(z)$ as a ``rotation rate'' about the point $z\in D_*$.

\medskip

\begin{corollary}\label{aug28.3}
Suppose $(h,\mu_0)\in \calH \leftrightarrow \theta(z) \in
\calB\leftrightarrow \theta(x)\in \calT\leftrightarrow \mu\in \calR$.
\begin{enumerate}[\rm (i)]
\item If $h$ satisfies \eqref{ju21.15} then for $z\in D_*$
\begin{align}
\theta(z)=-\arg\biggl(\int_{\partial D_*} \frac{x+z}{x-z}h(x)
\frac{|dx|}{2\pi}-i\mu_0/\pi\biggl).\label{ju21.16}
\end{align}
\item If $h(z)\tan\theta(z)$ or $\wt h(z)$ satisfy \eqref{ju21.15}, then
\begin{align}
\mu_0&=\mu(0)=\frac{1}{2}\int_{\prt D_*} h(x) \tan\theta(x) |dx|,
~~\text{and}
\label{ju21.22}\\
\mu(z)&=
\frac{1}{2}\int_{\prt D_*} \Re\biggl(\frac{x+z}{x-z}\biggl)
h(x)\tan\theta(x)|dx|\label{m11.1}\\
&=\frac{1}{2}\int_{\prt D_*}
h\Bigl(\frac{x+z}{1+\overline{z}x}\Bigl)\tan\theta\Bigl(\frac{x+z}{1+\overline{z}x}\Bigl)
|dx|.
\label{ju21.17}
\end{align}
\end{enumerate}
\end{corollary}

\begin{proof}
 $(i)$
 follows from the discussion above and \eqref{ju20.3}.

$(ii)$
Note that since $\mu=\mu_0-\pi \wt h (z)=\pi h(z)\tan\theta(z)$, for
$z\in D_*$, it follows that $h(z)\tan\theta(z)$ satisfies
\eqref{ju21.15} if and only if $\wt h (z)$ satisfies \eqref{ju21.15}.
Equations \eqref{ju21.22} and \eqref{m11.1} follow from \eqref{aug27.3}, \eqref{aug27.7},
and \eqref{aug28.5}.
Finally, equation \eqref{ju21.17} follows from
\eqref{m11.1} and a change of variables.
\end{proof}

\medskip

\begin{remark}\label{ju20.5}
\begin{enumerate}[\rm (i)]
\item
The maps $(h,\mu_0) \to \theta$ and $\theta \to (h,\mu_0)$ are continuous
under the topologies of uniform convergence on compact subsets of
$D_*$ and $(D_*,\R)$.
\item\label{ju21.20}
For functions in $\calB$, uniform convergence on compact subsets of
$D_*$ is equivalent to pointwise bounded convergence in $D_*$ and is
also equivalent to weak-* convergence (of the corresponding boundary value functions) in
$L^{\infty}(\partial D_*)$, as elements of the dual space of
$L^1(\partial D_*)$. But this convergence is not equivalent to
pointwise bounded a.e.\ convergence on $\partial D_*$.
For example, if
$\theta_k(z)=-\arg (1+ z^k/2)$, then $\theta_k \leftrightarrow
(h_k,0)$, with $h_k=\Re (1+z^k/2)$.  The functions $\theta_k$ converge to $0$,
uniformly on compact subsets
of $D_*$, pointwise boundedly on $D_*$, and  weak-* on $\partial D_*$. However, $\theta_k$
does not contain a subsequence converging pointwise on any subarc
in $\partial D_*$.
\item
The function $\theta$ is a constant function if and only if $h\equiv
1/\pi$ and $\mu_0=\tan \theta$. It is tempting to extend the definition
of $\calT$ to include $\theta\equiv \pi/2$ by saying $\theta\equiv
\pi/2$ corresponds to $h\equiv 1/\pi$ and $\mu_0 = +\infty$.  However,
we would lose the continuity of the correspondence. Indeed if
$(h,\mu_n), (g,\mu_n) \in \calH$ with $\mu_n\to +\infty$ and $g \ne h$,
let
$\theta_{2n} \leftrightarrow (h,\mu_{2n})$ and
$\theta_{2n+1} \leftrightarrow (g,\mu_{2n+1})$. Then $\theta_n$
converges to $\pi/2$ uniformly on compact subsets of $D_*$, but
the corresponding elements of $\calH$ do not converge.
\item If the pair $(h,\mu_0)$ corresponds to $\theta$ then $(h(\bar z),-\mu_0)$
corresponds to $-\theta(\bar z)$.  This follows from
Lemma \ref{ju20.1} since $f$ is analytic if and only if
$\overline{f(\bar z)}$ is analytic.  But $(h,-\mu_0)$ does not
correspond to $-\theta$, unless $h\equiv 1/\pi$. Indeed, if $(h,-\mu_0)$
does correspond to $-\theta$ then
\begin{align*}
-(\theta+i \wt \theta)=i\log(h+i\wt h
-i(-\mu_0)/\pi)-i\log\sqrt{1+\mu_0^2}/\pi.
\end{align*}
Adding this equation to \eqref{ju20.3} we obtain
\begin{align*}
0=i\log((h+i\wt h)^2+\mu_0^2/\pi^2)-2i\log\sqrt{1+\mu_0^2}/\pi,
\end{align*}
and thus $h+i\wt h$ is constant. Since $(h,\mu_0)\in \calH$, we have
$h\equiv h(0)=1/\pi$.

\item Equation \eqref{ju21.16} fails for the example $\theta \leftrightarrow(h,0)\in\calH$
where $h$ is given by \eqref{jul12.1}.
\end{enumerate}
\end{remark}

 \medskip

\begin{example}
Let $F=\phi+i \wt \phi= \sqrt{\log(1-z^2)}.$ We claim we can choose the branch of the square
root so that $F$ is analytic on $D_*$, with $\phi$ continuous on $\ol D_*$ and
$\wt \phi$ not bounded above or below. By Theorem \ref{aug27.0} and the definition of $\calR$ there is no
$(h,\mu_0)\in \calH$ so that
$\phi=\mu$, where $\mu \leftrightarrow (h,\mu_0)$. In fact
there do not exist any $a, b \in \R$, $b\ne 0$, and $(h,\mu_0)\in \calH$
such that $a+b\phi=\mu$.
 To see the claim, we set $g(z)=(\log(1-z))/z$. Then
$g$ is analytic on a simply connected neighborhood  of $\ol D_* \setminus\{1\}$
and non-vanishing, and hence has an analytic square root
$k$. Then $F(z)\equiv z k(z^2)$ is analytic on a neighborhood of $\ol
D_*\setminus\{\pm 1\}$ and satisfies $F(z)^2=\log(1-z^2).$
Thus $\phi$ and $\wt \phi$ are continuous and smooth on $\ol D_*\setminus\{\pm 1\}$.
Since $\phi^2-\wt \phi^2 =\log|1-z^2| \to -\infty$ as $z\to \pm 1$,
we conclude $\wt \phi^2 \to \infty$ as $z \to
\pm 1$. But $2\phi \wt \phi = \arg(1-z^2)$ is bounded, so we must have $\phi \to 0$
as $z\to \pm 1$. Thus $\phi$ is continuous on $\ol D_*$, and $\wt \phi$ is unbounded.
Since $F$ is odd, $\wt \phi$ is neither bounded above nor below.
\end{example}

\medskip

\begin{example}
Consider the harmonic function $\phi(z)=\Re z$ in $D_*$ with
boundary values $\phi(e^{it}) = \cos t$, $0 \leq t < 2\pi$.
If $a,b\in \R$, with $b\ne 0$, set $\mu=a+b\phi=a+b\Re z$. Then
$\wt \mu=b\Im z > -1$ for all $z\in D_*$ if and only if $|b| \le 1$.
By the equivalence of $\calR$ and $\calH$ given in Theorem \ref{aug27.0},
 $\mu=a+b\phi$ corresponds to some $(h,\mu_0)\in \calH$ if and only if
$|b| \le 1$.
\end{example}

If $\phi$ is harmonic on $D_*$ and if $\wt \phi$ is bounded, then
for $a,b\in \R$ with $b\ne 0$, the function $\mu=a+b\phi$ has harmonic
conjugate $b\wt \phi$. So for sufficiently small $b$, we have ~$\wt\mu>-1$
which implies
$\mu \in \calR$ and $a+b\phi \leftrightarrow
(h,\mu_0)\in \calH$ for some $(h,\mu_0)$.
Since $\wt\mu(0)=0$,  we have that $\inf \wt \phi < 0 < \sup \wt \phi$ so that
for $|b|$ sufficiently large
$\mu=a+b\phi$ fails to be in
$\calR$. So in some sense,
membership in $\calR$ depends on the ``oscillation'' of the harmonic
function on $D_*$, but not its mean.
The next proposition gives a more precise version. Its proof is elementary, but it
will be useful for understanding our (later) description of rotation rates and
stationary distributions for ORBMs.

\begin{proposition} \label{ju21.9}
Suppose $\phi$ is (real-valued and) harmonic in $D_*$. Set
\begin{align*}
K_-=\inf_{z\in D_*} \wt \phi(z)
\ \ \  \text{  and  } \ \ \ K_+= \sup_{z\in D_*} \wt \phi(z)
\end{align*}
If $a, b \in \R$ with
$-1/|K_+| \le b \le 1/|K_-|$, then there is a
unique $(h,\mu_0)\in \calH$ such that
\begin{align}
a + b \phi(z)= \mu(z), \label{ju21.10}
\end{align}
where $\mu$ and $(h,\mu_0)$ are related as in Theorem \ref{aug27.0}.
Conversely, if $b < -1/|K_+|$ or $b>1/|K_-|$
then there do not exist any $a\in \R$ and $(h,\mu_0)\in
\calH$ such that \eqref{ju21.10} holds.
\end{proposition}

In the statement of Proposition \ref{ju21.9} we allow the possibility that $K_+$
is infinite, in which case we interpret $1/|K_+|$ as equal to zero. A similar
statement holds for $|K_-|$.

\begin{proof}
Note that $K_-\le 0 \le K_+$ since $\wt \phi(0)=0$. If $b\in \R$ and
if
$-1/|K_+| \le b \le 1/|K_-|$,
set
$\mu=a+b \phi$. Then $\wt \mu(z)=b \wt \phi(z) \ge -1$. Since $\wt \mu(0)=0$, the
maximum principle implies that $\wt \mu(z) > -1$ for all $z\in D_*$,
so that $\mu\in \calR$.
The corresponding $(h,\mu_0)\in \calH$ is given by \eqref{aug27.4} of Theorem
\ref{aug27.0}.

Conversely if $(h,\mu_0)\in \calH$ corresponds to $\mu=a+b\phi\in
\calR$ as in Theorem \ref{aug27.0}, then $\wt \mu(z)=b \wt \phi(z) > -1$. But this
implies $b \ge -1/\sup \wt \phi(z)$ and $ b \le 1/|\inf \wt \phi(z)|$.
\end{proof}

\medskip

If a real-valued function is slightly better than continuous, then its harmonic
conjugate is continuous and hence bounded.
For a function $f: \partial D_* \to \R$, we define the modulus of continuity of $f$ by
$\omega_f(a) = \sup_{|s-t|< a} |f(e^{is}) - f(e^{it})|$. We say that $f$ is Dini continuous
if $\int_0^b (\omega_f(a)/a) da < \infty$ for some $b>0$. If $f$ is Dini
continuous then $\wt f$ is continuous and therefore bounded. See
\cite[Thm III.1.3]{Gar}.

\begin{theorem} \label{j27.4}
Suppose that $\theta\in \calT$, $(h,\mu_0)\in \calH$, and $\mu\in \calR$
correspond to each other as in Theorem \ref{aug27.0}. See also
\eqref{ju21.21} and \eqref{aug28.2}.
\begin{enumerate}[\rm (i)]
\item
If $\theta$ is Dini continuous on $\prt D_*$, then $h$
and $\mu$
extend to be continuous on $\ol D_*$.
If $\mu$ is Dini continuous on $\prt D_*$, then $h$ is continuous on $\ol D_*$
and $\theta$ is continuous on $\ol D_*\setminus Z$, where $Z=\{x\in \prt
D_*: h(x)=\mu(x)=0\}.$ Similarly, if $h$ is Dini continuous on $\prt D_*$,
then $\mu$ is continuous on $\ol D_*$, and $\theta$ is continuous on $\ol
D_*\setminus Z$. In each of these cases, $h$ and $\wt h$ satisfy
\eqref{ju21.15}, so that the conclusions of Corollary \ref{aug28.3} hold.
\item
Suppose that $\omega$ is an increasing continuous concave function on
$[0,\pi/2]$ such that $\omega(0)=0$, $\omega(\pi/2)=\pi/4$, and  $\int_0^{\pi/2} \frac{\omega(a)}{a} da = \infty$.
Then there exists $\theta\in \calT$ such that its modulus of continuity $\omega_\theta(a)= \omega(a)$ for $a \in [0,  \pi/2]$ and
both $h$ and $\mu$ are unbounded.
\end{enumerate}
\end{theorem}

\begin{proof}

(i) By \cite[Thm. III.1.3]{Gar}, if $\theta$ is Dini continuous then
the harmonic conjugate $\wt \theta$ is continuous on $\ol D_*$.
Hence, $F(z)=\exp(\wt \theta(z)-i\theta(z))$ is continuous and so is $h+i\wt
h$ by \eqref{ju20.4}. Hence $h$ and $\mu=\mu_0-\pi \wt h$ are continuous.
The remaining statements in (i) follow from \eqref{aug27.3},
\eqref{aug27.4}, and \eqref{ju20.3} and \cite[Cor. III.1.4]{Gar}.
In each of the cases in (i), $h$ and $\wt h$ are continuous on $\ol D_*$ and
hence satisfy \eqref{ju21.15}.

(ii) We give here an example based on \cite[page 101]{Gar}.
Suppose that $\omega$ is increasing and concave on $[0,\pi/2]$ with
$\omega(0)=0$, $\omega(\pi/2)=\pi/4$,
 and
\begin{align}
\int_0^{\pi/2}(\omega(t)/t) dt =
\infty.\label{j29.1}
\end{align}
Set
\begin{align*}
\alpha(t) =
\begin{cases}
\omega(t)  & \text{if  } 0 \leq t \leq \pi/2, \\
\omega(\pi-t)  & \text{if  } \pi/2 \leq t \leq \pi, \\
0 & \text{if  } -\pi < t < 0.
\end{cases}
\end{align*}
For $0\le x<y\le \pi$, write $x=ty$, ~$0 < t < 1$, and so
$y-x=(1-t)y$. 
Since $\omega$ is concave and $\alpha(0)=\omega(0)=0$,
$$
t\alpha(y)  \le \alpha(x)\qquad
\text{and}\qquad
(1-t)\alpha(y)  \le \alpha(y-x).
$$
Adding these inequalities we obtain $\alpha(y)-\alpha(x)\le \alpha(y-x)$.
Since $\alpha(\pi)=0$, replacing $\alpha(t)$ by $\alpha(\pi-t)$ in the
above argument, we also have that
$\alpha(x) -\alpha(y) \le \alpha(y-x)$.  If $ x < 0 < y < \pi$ with
$|x-y|< \pi/2$, then
$$
\alpha(y)-\alpha(x)=\alpha(y)\le \alpha(y+|x|)=\alpha(y-x).
$$
Set $\theta(e^{it})=-\alpha(t)$.
Then $\theta \in \calT$, because $|\alpha|\le \pi/4$, and $\omega_{\theta}(a)=\omega_{\alpha}(a)=\omega(a)$ for $0 \le a \le \pi/2$.

Let $b(r) = \cos^{-1}(\frac{1+r}{2})$.  Then for $r\in (0, 1)$,
\begin{align*}
\wt \theta(r)&= -\frac{1}{2\pi}\int_{0}^{\pi} \Im
\biggl(\frac{e^{it}+r}{e^{it}-r}\biggl)\alpha(t)dt\\
&\ge \frac{1}{2\pi}\int_{b(r)}^{\pi}\frac{2r\sin t}{|e^{it}-r|^2}
\alpha(t)
dt.
\end{align*}
Since $|e^{it}-1|\ge|e^{it}-r|$ when $\cos t \le (1+r)/2$, we have that 
$$
\wt \theta(r) \ge -\frac{r}{2\pi}\int_{b(r)}^{\pi} \Im
\biggl(\frac{e^{it}+1}{e^{it}-1}\biggl)\alpha(t)dt = \frac{r}{2\pi}\int_{b(r)}^{\pi}\frac{\alpha(t)}{\tan{t/2}} dt,
$$
which increases to $+\infty$ as $r\to 1$.
So $\wt \theta(r)$
is not bounded above. Because $\theta$ is continuous on $\partial D_*$
with $\theta(1)=0$, ~$\theta(z)$ extends to be continuous on
$\ol D_*$ and
$\cos\theta(r) \to 1$ as $r\to 1$, so by
\eqref{aug27.5} $h$ is also unbounded.
\end{proof}

Theorem \ref{j27.4} (ii) implies that if $\theta \in \calT$
is not Dini continuous on $\partial D_*$, then $h$ and $\mu$ may not be
extended continuously to $\ol D_*$.
The next proposition examines the situation when $\theta$ is as large as
possible on an interval of $\prt D_*$.

\begin{proposition}\label{ju21.18}
Suppose $I$ is an open arc in $\partial D_*$, and suppose $\theta\in \calT
\leftrightarrow (h,\mu_0)\in\calH$.
\begin{enumerate}[\rm (i)]
\item
If $\theta(x)=\pi/2$ a.e. on $I$, then
$f=h+i\wt h-i\mu_0/\pi$
extends to be analytic in a neighborhood of $D_*\cup I$ with $h=0$ on
$I$. The same conclusion holds if $\theta(x)=-\pi/2$ a.e. on $I$.
\item
If $h$ extends to be continuous on $D_* \cup I$ with $h=0$ on $I$,
then $f=h+i \wt h -i\mu_0/\pi$ extends to be analytic in a neighborhood
of $D_* \cup I$ with at most one zero $e^{it_0}\in I$. If $f\ne 0$ on
$I$ then $\theta \equiv \pi/2$ or $\theta \equiv -\pi/2$ on $I$. If
$f(e^{it_0})=0$ for some $e^{it_0}\in I$, then $\theta(e^{it})=-\pi/2$
for $e^{it}\in I$ with $t< t_0$ and $\theta(e^{it})=\pi/2$ for
$e^{it}\in I$ with $t> t_0$.
\end{enumerate}
\end{proposition}

\begin{proof}  (i)  Suppose $\theta(x)=\pi/2$ a.e. on $I$.
For $z\in D_*$ set $F(z)=\theta(z)-\pi/2 +i \wt \theta (z)$.
Then by \eqref{ju21.12}
\begin{align}
F(z)=\int_{\partial D_*} \frac{x+z}{x-z}
(\theta(x)-\pi/2)\frac{|dx|}{2\pi}=
\int_{\partial D_*\setminus I}
\frac{x+z}{x-z}(\theta(x)-\pi/2)\frac{|dx|}{2\pi}.\label{ju21.19}
\end{align}
The right-hand side of \eqref{ju21.19} defines an analytic function on
$\C \setminus (\partial D_*\setminus I)$. By \eqref{ju20.4}, $f\equiv h+i
\wt h -i \mu_0/\pi$ extends to be analytic in a neighborhood of $D_*\cup
I$. Also by \eqref{ju21.19}
\begin{align*}
\Re F(z)=\theta - \pi/2 =\int_{\partial D_*\setminus I}
\frac{1-|z|^2}{|x-z|^2}(\theta(x)-\pi/2)\frac{|dx|}{2\pi}.
\end{align*}
If $y\in I$, then $\frac{1-|z|^2}{|x-z|^2} \to 0$ uniformly in $x\in
\partial D_*\setminus I$ as $z\to y$. Thus $\Re F(z)=\theta(z)-\pi/2
\to 0$ as $z\to y\in I$.  Taking real part of \eqref{ju20.4},
\begin{align*}
h(z)=\frac{e^{\wt \theta(z)} \cos \theta(z)}{\pi \cos \theta(0)},
\end{align*}
so by the continuity of $\theta$ and $\wt \theta$ on $D_* \cup I$,
we have $h\to 0$ as $z\to  y\in I$.

To prove  (ii), suppose that $h$ extends to be continuous on
$D_*\cup I$  with $h=0$ on $I$. By the Schwarz reflection principle $f=h+i \wt h
-i\mu_0/\pi$ extends
analytically across $I$. By the Cauchy-Riemann equations,
$$\frac{\partial}{\partial t} \Im f(e^{it})= \frac{\partial}{\partial r}
\Re f(r e^{it}) \vert_{r=1} = \frac{\partial h}{\partial r} \le 0$$
on $I$ since $h=0$ on $I$ and $h>0$ on $D_*$. Since $\Re f = 0$ on $I$,
$\Im f$ cannot be constant on any subarc of $I$ and thus
$f$ is a one-to-one map
of the arc $I$ onto a subarc of the imaginary axis, and $(ii)$ follows
from \eqref{ju21.16}.
\end{proof}

\section{Main results}\label{sec:main}

This section contains only statements of the main results of this paper.
The proofs will be given in  Section \ref{sec:proofs}.  
First,  in Section \ref{subs-main1}, we establish  results when the
domain $D$ is smooth and the angle of reflection $\theta$ is
$C^2$  and non-tangential everywhere, that is, $\theta$ lies in a closed subinterval 
of $(-\pi/2,\pi/2)$.     Theorem \ref{j25.1} summarizes results on
existence and uniqueness of ORBMs, and Theorem \ref{f3.2} considers 
ORBMs on the disk $D_*$ and establishes
the probabilistic interpretation of the quantity $(h(z), \mu_0)$
 corresponding to $\theta \in \calT$, as specified in Theorem \ref{aug27.0}. 
ORBMs in $D_*$  with general reflection angles 
$\theta \in \calT$ are constructed in 
Section \ref{subs-main2}.  The  focus of Section \ref{subs-main3} (in
particular, see Theorem \ref{j18.1})  is
the case when the reflection vector field is
tangential at every point, which leads to a process referred to as 
excursion reflected Brownian motion (ERBM). 
  Lastly,  in Section \ref{subs-main4} (specifically, Theorems
\ref{j15.7}--\ref{j17.3} therein) we construct ORBMs in simply connected domains using conformal
mappings and then show, in the case of simply connected bounded
Jordan domains, that they can also be obtained as suitable limits of
ORBMs in $C^2$ domains. 

\subsection{Smooth $D$ and $C^2$-smooth non-tangential $\theta$}
\label{subs-main1}

We start with a theorem on existence and uniqueness of ORBM in the simplest case, when the domain is smooth and the angle of reflection is smooth and takes values in a closed  subinterval of $(-\pi/2, \pi/2)$. The result is essentially known.

\begin{theorem}\label{j25.1}
Assume that $D\subset \C$ is a bounded open set with $C^2$ boundary, and a function $\theta: \prt D \to (-\pi/2, \pi/2)$ is $ C^2$.
\begin{enumerate}[\rm (i)]
\item {\rm (\cite[Thm.~2.6]{HLS})}
The submartingale problem \eqref{f3.1} has a unique solution which defines a strong Markov process.

\item
The strong Markov process defined by the Skorokhod equation 
\eqref{j13.1}  is continuous and has the same distribution as the process defined by the submartingale problem \eqref{f3.1}.

\item {\rm (\cite{KKY}) } The ORBM obtained in (i) and (ii) can also
  be constructed by using the 
non-symmetric Dirichlet form approach.
\end{enumerate}
\end{theorem}

It follows from the results in \cite{HLS} that if $\theta$ is $C^1$
then the ORBM
$X$ in the unit disc $D_*$ has a unique stationary distribution with the
density $h$ given by \eqref{ju21.21}. The stationary distribution was
characterized in \cite{HLS} in terms of a partial differential
equation in $D_*$ with appropriate boundary conditions.
In Theorem \ref{f3.2} (ii), we will show a partial converse, namely, that the stationary distribution characterizes an ORBM up to a real number that represents the ``rotation rate'' of $X$ about 0.

Under the assumptions of Theorem \ref{j25.1}, the ORBM
$X$ is continuous, a.s.. Consider a fixed $z\in D_*$.
Since $X_t \ne z$ for all $t>0$, a.s. (even if $X_0 =z$), we can uniquely define the function $t\to\arg(X_t-z)$ by choosing its continuous version and making an arbitrary convention that $\arg(X_1-z) \in [0, 2\pi)$.

Since $h$ is the density of the stationary measure of $X$ and $\theta$ is the
reflection angle, \eqref{ju21.22}
suggests that $\mu_0$ represents one half of the speed of rotation of $X$ about $0$. Hence, one might hope that
$\lim_{t\to \infty} \arg X_t /t $ is equal to a constant multiple of $
\mu_0$, a.s. Unfortunately, this simple interpretation of $\mu_0$ is false because $\arg X_t$ behaves like a Cauchy process (see \cite{Spi,BerWer1}) and, therefore, the law of large numbers does not hold for $\arg X_t$.
We will identify $\mu_0$ with the speed of rotation using two other representations in Theorem \ref{f3.2} (ii)-(iii).
We need the following definitions to state the representations. First of all, recall that a random variable has the Cauchy distribution if its density is $1/(\pi(1+x^2))$ for $x\in \R$. Next we will define a new measure of winding speed which does not include large windings if they occur during a single excursion from the boundary.
Recall definitions related to excursions from Section \ref{sec:exc}.
We will say that $\exc_s$ belongs to the family $\calE_t^L $ of
excursions with ``large winding number'' if
$s + \zeta(\exc_s) \leq t$ and
 $|\arg X_s - \arg X_{s +\zeta(\exc_s)-}| > 2\pi$, where $X_{u-}$ denotes the
left-hand limit.
For $z\in D_*$, let
\begin{align}\label{n11.1}
\arg^* X_t &= \arg X_t - \sum_{s: \, \exc_s \in \calE^L_t}
\left(\arg X_{s + \zeta(\exc_s)-} - \arg X_s \right),\\
\arg^* (X_t-z) &= \arg (X_t-z) - \sum_{s: \, \exc_s \in \calE^L_t}
\left(\arg (X_{s + \zeta(\exc_s)-}-z) - \arg (X_s-z) \right).
\label{n11.2}
\end{align}

\begin{theorem}\label{f3.2}
In parts (i)-(iii), we
assume that a $C^2$ function $\theta: \prt D_* \to (-\pi/2, \pi/2)$ is given.

\begin{enumerate}[\rm (i)]
\item {\rm (\cite[Thm.~2.18]{HLS})}
The density of the stationary measure for $X$ defined in \eqref{j13.1} is a
positive harmonic function $h$ in $D_*$ given by \eqref{ju21.21} (see also
\eqref{ju20.4}).
\item
With probability 1, $X$ is continuous and, therefore, $\arg X_t$ is well
defined for $t>0$.
Let  $\mu_0 \in \R$ be given by \eqref{ju21.21}.
For every $z\in \ol D_*$,
the distributions of $\frac 1t \arg X_t - \mu_0$ under $\P_z$ converge to the Cauchy distribution when $t\to \infty$.
\item For every $y\in \ol D_*$,
\begin{align}\label{j15.2}
\lim_{t\to \infty}  \frac1{t}  {\arg^* X_t}  &= \mu_0,
\quad \P_y \hbox{-a.s.}
\end{align}
The formula holds more generally. For   any $y,z\in D_*$,
\begin{align} \label{j16.1}
\lim_{t\to \infty}\frac{1}{t}   \arg^* (X_t-z)   &= \mu(z),  \quad \P_y \hbox{-a.s.,}
\end{align}
where   $\mu(z)$ is given by \eqref{aug28.2}.

\item
Conversely, suppose we are given any $\mu_0 \in \R$ and a harmonic function $h$ in $D_*$ that is
 $C^2$ in $\ol D_*$, positive on $\ol D_*$,
and satisfies $h(0)=1/\pi$. Let $\theta \lra (h,\mu_0)$.
Then for every $x_0\in \ol D_*$,
there exists a unique in distribution process $X$ satisfying \eqref{j13.1} with this $\theta$. Its stationary distribution has density $h$ and \eqref{j15.2} holds.
\end{enumerate}
\end{theorem}

\begin{remark}
(i) We could have defined the family $\calE_t^L $ of excursions $\exc_s$ with ``large winding number'' as those satisfying
$s + \zeta(\exc_s) \leq t$ and  $|\arg X_s - \arg X_{s + \zeta(\exc_s)-}| > a$, where $a>0$ is not necessarily $2\pi$. It turns out that \eqref{j15.2} holds for any $a>0$. The limit in \eqref{j15.2} holds for any value of $a$ because the only thing that matters in \eqref{n11.1} is that the large jumps of the Cauchy-like process $\arg X$ are removed. The ``remaining part'' of this process satisfies the law of large numbers and has  mean $\mu_0 t$, no matter how large  the threshold for the ``large jumps'' is.
We have chosen $a=2\pi$ because this value has a natural geometric interpretation and is invariant, in a sense, under conformal mappings.

(ii)
We will prove \eqref{j16.1} using \eqref{f6.2} and a purely analytic argument.
Formula \eqref{j16.1} has the same heuristic meaning as
\eqref{ju21.22} as a rotation rate, except that it represents the sum (integral) of infinitesimally small increments of the angle around $z$, not 0.

(iii) In view of Theorem \ref{aug27.0}, if the rotation rate $\mu(z)$
is known for all $z\in D_*$, it completely determines $\theta$ and $h$.
Moreover, due to the harmonic character of $\mu(z)$, if this function
is known in an arbitrarily small non-empty open subset of $D_*$,
this also determines $\theta$ and $h$.

(iv) Theorem \ref{aug27.0} and the definition of the function space
$\calR$ show which harmonic
functions $\mu(z)$ represent rotation rates for an ORBM. See also
Proposition \ref{ju21.9}. Roughly speaking,
$\mu(z)$ represents rotation rates for an ORBM if its oscillation over
$\ol D_*$ is not too large. There is no restriction, however, on the
average value of $\mu(z)$. If $\mu(z)$ and $\mu_1(z)$ represent the
rotation rates for
two ORBM's, and $\mu(z) = c + \mu_1(z)$ for some constant $c$ and all
$z$ then  $\wt \mu=\wt \mu_1$. By \eqref{aug27.4} of Theorem
\ref{aug27.0}, the corresponding stationary densities are the same for both
ORBM's.

(v) Parts (ii) and (iii) of Theorem \ref{f3.2} are similar in spirit to \cite[Thm.~7.1]{legallyor} although that paper is concerned with Brownian motion with drift, not reflection.
\end{remark}

\medskip

\subsection{ORBMs on $D_*$ with general reflection angles $\theta$}
\label{subs-main2}

Suppose $\theta\in \calT$.  
Then $\theta\not\equiv \pi/2$ and $\theta\not\equiv-\pi/2,$ although 
$\theta$ could be tangential on a strict subset of the boundary
$\partial D_*$. 
In Theorem \ref{j15.5} we  show that ORBMs on the disk $D_*$ 
associated with $\theta$ can be obtained as limits of ORBMs 
on  $D_*$ with $C^2$ angles of reflection, which are well defined by
Theorem \ref{j25.1}.   Then in Theorem \ref{T:3.8} we establish a
conformal invariance property for such ORBMs. 
If there do exist  points on the boundary
at which $\theta$ is tangential,  the associated ORBM
will not in general be continuous, and thus one has to 
carefully define the topology in which the above limit procedure 
can be carried out. 

 We start by introducing some relevant notation to
define this topology.  Let
\begin{align}\label{n26.11}
N^+_\theta &= \{x\in \prt D_*: \theta(x) = \pi/2\},\qquad
N^-_\theta = \{x\in \prt D_*: \theta(x) = -\pi/2\}.
\end{align}
Since we identify functions in $\calT$ that  are equal to each other
a.e.,
\begin{align}\label{n26.10}
|N^+_\theta| < 2\pi
\qquad \text {  and  } \qquad
 |N^-_\theta| < 2\pi.
\end{align}
We will say that $x\in \interior N^+_\theta$ if
$\theta\equiv \pi/2$ a.e. in some neighborhood of $x$.
The definition of  $\interior N^-_\theta$ is analogous.
For $x= e^{i\alpha}\in \interior N^+_\theta$,
let $ \alpha^+$ be the largest real number such that $\{ e^{it}: t \in [\alpha, \alpha^+)\} \subset \interior N^+_\theta$, and
let $\beta^+(x) = e^{i\alpha^+}$. Similarly,
for $x= e^{i\alpha}\in \interior N^-_\theta$,
let $ \alpha^-$ be the smallest real number such that $\{ e^{it}: t \in (\alpha^-, \alpha]\} \subset \interior N^-_\theta$, and
let $\beta^-(x) = e^{i\alpha^-}$.

We recall below the definition of the $M_1$ topology introduced by
Skorokhod in \cite{SKO56}. We will use the $M_1$ topology rather than
the more popular $J_1$ topology because we will be concerned with
convergence of continuous processes to (possibly) discontinuous
processes. In the $J_1$ topology, a
sequence of continuous processes cannot converge to a discontinuous process.
We will also define an $M_1^\calT$ topology, appropriate for our setting.

\begin{definition}\label{n25.5}
(i) Suppose that $0<T< \infty$ and $x: [0,T] \to \R^n$ is a c\`adl\`ag function. The
graph $\Gamma_x$ is the set consisting of all pairs $(a, t)$ such that
$0\leq t\leq T$ and $a\in [x(t-), x(t)]$ (here $[x(t-), x(t)]$ is the line segment between the left-hand limit $x(t-)$ and $ x(t)$ in $\R^n$). A pair of functions $\{(y(s), t(s)), s \in [0,1]\}$ is a parametric representation of $\Gamma_x$ if $y$ is continuous, $t$ is continuous and non-decreasing, and $(v,u) \in \Gamma_x$ if and only if $(v,u) = (y(s), t(s))$ for some $ s \in [0,1]$.
We say that {\it $x_n$ converge to $x$ in $M_1$ topology} if there exist parametric representations $\{(y(s), t(s)), s \in [0,1]\}$ of $\Gamma_x$ and $\{(y_n(s), t_n(s)), s \in [0,1]\}$  of $\Gamma_{x_n}$ such that
\begin{align}\label{n25.6}
\lim_{n\to \infty} \sup_{s\in[0,1]} | (y_n(s), t_n(s)) - (y(s), t(s))| = 0.
\end{align}

(ii) If $x: [0,\infty) \to \R^n$ then we
say that {\it $x_n(t)$ converge to $x(t)$ in $M_1$ topology}
if they converge to $x$ on $[0,T]$ in $M_1$ topology for every $0 < T < \infty$. This is equivalent to the following statement.
There exist parametric representations $\{(y(s), t(s)), s \in [0,\infty)\}$ of $\Gamma_x$ and $\{(y_n(s), t_n(s)), s \in [0,\infty)\}$  of $\Gamma_{x_n}$ such that for every $T\in(0,\infty)$,
\begin{align}\label{n25.7}
\lim_{n\to \infty} \sup_{s\in[0,T]} | (y_n(s), t_n(s)) - (y(s), t(s))| = 0.
\end{align}

(iii)
Consider $\theta \in \calT$.
We will say that {\it $x: [0,\infty) \to \ol D_*$ belongs to $\calA_\theta$} if it is
c\`adl\`ag and satisfies the following conditions. For all $t\geq 0$,
$x_{t-} \ne x_t$ if and only if $x_{t-} \in \interior N^+_\theta \cup  \interior N^-_\theta$. Moreover, if $x_{t-} \in \interior N^+_\theta $ then $x_t = \beta^+(x_{t-})$. If $x_{t-} \in \interior N^-_\theta $ then $x_t = \beta^-(x_{t-})$.
Let $\calA_\calT=\bigcup_{\theta\in\calT}\calA_\theta$.

(iv) Assume that $\theta \in \calT$ and $x\in \calA_\theta$. If $x_{t-}= e^{i\alpha}\in \interior N^+_\theta$ and
$x_t = \beta^+(x_{t-}) = e^{i\alpha^+}$, then we let
 $[x_{t-}, x_t]_\theta
= \{e^{it}: t\in[\alpha, \alpha^+]\} $
be the arc on $\prt D_*$ between $x_{t-}$ and $x_t$. Thus
$\theta(e^{is})=\pi/2$ for a.e. $e^{is}\in [x_{t-}, x_t]_\theta $.
Similarly, if $x_{t-}= e^{i\alpha}\in \interior N^-_\theta$ and
$x_t = \beta^-(x_{t-}) = e^{i\alpha^-}$, then we let
$[x_{t-}, x_t]_\theta = \{e^{it}: t\in[\alpha^-, \alpha]\} $.

\smallskip

We define
the graph $\Gamma^\theta_x$ as the set of all pairs $(a, t)$ such that $a=x_t$ if $x$ is continuous at $t$ and
 $a\in [x_{t-}, x_t]_\theta$
if $x_{t-} \ne x_t$. A pair of functions $\{(y(s), t(s)), s \in [0,\infty)\}$ is a parametric representation of $\Gamma^\theta_x$ if $y$ is continuous, $t$ is continuous and non-decreasing, and $(v,u) \in \Gamma^\theta_x$ if and only if $(v,u) = (y(s), t(s))$ for some $ s \in [0,\infty)$.
Suppose that $x_n \in \calA_{\theta_n}$ for some $\theta_n \in \calT$, $n\geq 1$, and $x\in \calA_\theta$ for some $\theta \in \calT$.
We say that $x_n$ converge to $x$ in $M^\calT_1$ topology if there exist parametric representations $\{(y(s), t(s)), s \in [0,\infty)\}$ of $\Gamma^\theta_x$ and $\{(y_n(s), t_n(s)), s \in [0,\infty)\}$  of $\Gamma^{\theta_n}_{x_n}$ such that for every $T\in(0,\infty)$,
\begin{align}\label{n25.71}
\lim_{n\to \infty} \sup_{s\in[0,T]} | (y_n(s), t_n(s)) - (y(s), t(s))| = 0.
\end{align}

\end{definition}

Some \cadlag functions $x$ (for example, continuous functions) belong to more than one
family $\calA_\theta$. We leave it to the reader to check that the definitions in (iv)
are not affected by the choice of $\calA_\theta$.

We will extend the definition of $t\to \arg X_t$ to (some) processes that are not continuous. Although it is impossible to define a continuous version of $t\to \arg X_t$ for a process $X$ that is discontinuous, we will define a functional $\{X_t, t\geq0\} \to \{\barg X_t, t\geq 0\}$ in a way that reflects the structure of jumps in a natural way, leading to heuristically appealing results. The functional $\barg$ will be defined relative to $\theta$ but the dependence will be suppressed in the notation. Consider a
function $x\in\calA_\theta$ such that $x_t \ne 0$ for all $t\geq 0$.
Consider any parametric representation
$\{(y(s), t(s)), s \in [0,\infty)\}$ of $\Gamma^\theta_x$ and let $s \to \arg y(s)$ be the continuous version of $\arg y$ with $\arg y(0) \in [0, 2\pi)$. We let $\barg x_u = \arg y(s)$ where $s=\sup\{r: t(r)=u\}$. It is elementary to check that this definition of
$\barg x_u$ does not depend on the choice of
parametric representation
$\{(y(s), t(s)), s \in [0,\infty)\}$ of $\Gamma^\theta_x$.

Recall the definitions \eqref{n11.1}-\eqref{n11.2} and notation introduced in the paragraph preceding them. We define $\barg^*$ in an analogous way.
For $z\in D_*$, let
\begin{align*}
\barg^* X_t &= \barg X_t - \sum_{s: \, \exc_s \in \calE^L_t}
\left(\barg X_{s + \zeta(\exc_s)-} - \barg X_s \right),\\
\barg^* (X_t-z) &= \barg (X_t-z) - \sum_{s: \, \exc_s \in \calE^L_t}
\left(\barg (X_{s + \zeta(\exc_s)-}-z) - \barg (X_s-z) \right).
\end{align*}

\begin{theorem}\label{j15.5}
Consider $\theta\in \calT$. There exists a sequence of $C^2$ functions $\theta_k: \prt D_* \to (-\pi/2, \pi/2)$
which converges to $\theta$ in weak-* topology as elements of the dual space of $L^1(\prt D_*)$, that is,
$$
\lim_{k\to \infty} \int_{\prt D_*} f(x) \theta_k (x) |dx|  =  \int_{\prt D_*} f(x) \theta (x) |dx|
\quad \hbox{for every } f \in L^1 (\prt D_*).
$$
 Fix such a sequence $\{\theta_k\}$ and let $X^k$ be defined by the following SDE analogous to \eqref{j13.1},
\begin{align}\label{n19.1}
 X^k_t = z_k + B_t + \int_0^t  \bv_{\theta_k} (X^k_s) dL^k_s  \qquad \hbox{for } t\geq 0.
\end{align}
Assume that $z_k \to z_0 \in D_*$ as $k\to \infty$, $z_0 \ne 0$, and recall \eqref{n26.10}.
\begin{enumerate}[\rm (i)]
\item  {\rm (\cite[Thm. 1.1]{BurMar})}
$X^k$'s converge weakly in $M^\calT_1$ topology to a conservative Markov process $X$ on $\ol D_*$ such that $X_0 = z_0$, a.s.
Moreover, there is a \cadlag version of $X$ and for this version, $X\in\calA_\theta$, a.s.
The process $\{X_t; t\in [0, \sigma_{\partial D_*})\}$, where $\sigma_{\prt D_*}:=\inf\{t>0: X_t \in \partial D_*\}$, is Brownian motion
killed upon leaving $D_*$.

 \item
 $X^k$'s converge to $X$  in the sense of finite dimensional distributions.

\item
The Markov process $X$ has a stationary measure whose density $h$ is given by \eqref{ju21.21}.

\item
The functional $\{x_s, s\in[0,\infty)\} \to \{\barg x_s, s\in[0,\infty)\}$ is a continuous mapping from the set
$\calA_\calT$ equipped with $M^\calT_1$ topology to the set of \cadlag functions equipped with the $M_1$ topology. For every $t\geq 0$, the distributions of
 $\barg X^k_t$
converge to the distribution of $\barg X_t$.
\item
Let $\mu_0$ be as in \eqref{ju21.21}.
Then for every $z\in \ol D_*$,
the distributions of $\frac 1t \barg X_t - \mu_0$ under $\P_z$ converge to the Cauchy distribution when $t\to \infty$.
\item
For every $y\in \ol D_*$, $\P_y$-a.s.,
\begin{align}\label{d11.1}
\lim_{t\to \infty} \frac{1}{t}  \barg^* X_t  &= \mu_0.
\end{align}
Moreover, for any $y,z\in D_*$,
\begin{align}
\lim_{t\to \infty} \frac{1}{t}  \barg^* (X_t-z)  &= \mu(z),
\qquad  \P_y \hbox{-a.s.,}
\label{d11.2}
\end{align}
where  $\mu(z)$ is the harmonic function defined by   \eqref{aug28.2}.

\item
Assume that $\theta \in\calT \leftrightarrow (h,\mu_0)\in \calH$.
Then for every $x\in \prt D_*$, $ x \in \Gamma^\theta_X$ with probability 1
 if and only if
\begin{align}\label{sept6.1}
\int_0^1  e^{-\wt \theta(rx)}\cos \theta(rx) \frac{dr}{1-r}
= \int_0^1 \frac{h(rx)/(\pi \cos \theta(0))}{h(rx)^2+(\wt h(rx)-\mu_0/\pi)^2}
\frac{dr}{1-r} < \infty.
\end{align}
\item Suppose that $\theta, \bar\theta_k\in \calT$ and
    $\bar\theta_k$ converge to $\theta$ in weak-* topology.
    Let $\bar X^k$'s have their distributions determined by
    $\bar\theta_k$'s in the same way as $X$'s distribution
    is determined by $\theta$. Assume that $\bar X^k_0 =
    z_k$, $X_0= z_0$ and $z_k \to z_0$ as $k\to\infty$.
    Then $\bar X^k$ converge weakly to $X$ in $M^\calT_1$
    topology.
\end{enumerate}
\end{theorem}

\medskip

We will call the process $X$ obtained in Theorem \ref{j15.5}
ORBM with reflection angle $\theta$.

\begin{remark}\label{j15.10}
(i) Note that the distribution of $X$ in Theorem \ref{j15.5}
(i) does not depend on the approximating sequence $\theta_k$
because if we have two sequences $\{\theta_k\}$ and $\{\bar
\theta_k\}$ converging to $\theta$ then we can apply the
theorem to the sequence $\theta_1, \bar \theta_1, \theta_2,
\bar \theta_2, \dots$

(ii) Suppose that $z_0 \in D_*$, $\mu_0\in \R$, and
$h $ is positive and harmonic in $D_*$ with $h(0)=1/\pi$.
By Theorem \ref{aug27.0}, we can find
$\theta\in \calT\leftrightarrow (h,\mu_0)\in \calH$.
Let $X$ be the process corresponding to $z_0 $ and $\theta$ as in Theorem \ref{j15.5}.
Then $X$ has a stationary distribution with the density $h$ and $\mu_0$ is the rate of rotation of $X$ in the sense of Theorem \ref{j15.5} (v)-(vi).

(iii)
Theorem \ref{j15.5} establishes  existence of ORBM for all angles $\theta$ of oblique reflection. ORBMs  can be uniquely parametrized either by  $\theta \in \calT$ or by pairs $(h,\mu_0) \in \calH$. We will write $X \leftrightarrow \theta$ or $X \leftrightarrow (h,\mu_0)$.

(iv) If $\theta=\pi/2$ a.e. on an open arc $I\subset \prt D_*$ then as
in the proof of Proposition \ref{ju21.18}, $\theta+i\widetilde\theta$
extends to be analytic across $I$, and hence so does
$G=e^{i(\theta+i\widetilde\theta)}$. In this case, for $x\in I$,
\begin{align}
\lim_{r\to 1} \frac{e^{-\widetilde\theta(rx)}\cos\theta(rx)}{r-1}
=\Re \lim_{r\to 1} \frac{G(rx)-G(x)}{rx-x}x
=\Re G'(x)x.
\end{align}
Thus the integral in \eqref{sept6.1} is finite for each $x\in I$. A
similar statement holds if $\theta=-\pi/2$ a.e. on $I$.

Note that the process $X$ itself will not hit a fixed point $x\in I$.
The reason is that $X$ has only a countable number of excursions
from the boundary of $\prt D_*$ and the distribution of the location
of the endpoint of an excursion has a density. Hence, with probability 1,
no excursion will end at $x$. If an excursion ends at a point in $I$,
the process  $X$ will jump at that time to an end of the
interval where $\theta=\pi/2$ a.e. Thus, $X$ itself will avoid $x$
forever but the same argument shows that $x\in \Gamma^\theta_X$ with probability 1
because $\Gamma^\theta_X$ contains the arcs between the endpoints of excursions
hitting points inside $I$ and the points to which $X$ jumps at those times.

(v) Let $\nu $ be the positive measure on $\prt D_*$  defined by
$h(z) = \int_{\prt D_*} K_x(z) \nu(dx)$, where $K_x(z)$ is the Poisson
kernel for $z\in D_*$. Fix $x\in \partial D_*$ and write
\begin{align*}
h(rx)=c\frac{1+r}{1-r} + \int_{\prt D_*} \frac{1-r^2}{|y-rx|^2}
d\sigma(y)
\end{align*}
where $\sigma$ is a positive measure with $\sigma(\{x\})=0$ and
$c=\nu(\{x\})$.
Then
\begin{align}\label{sept6.2}
\lim_{r\to 1} (1-r)h(rx) = 2c
\end{align}
as can be seen by splitting the integral into $\int_I+\int_{\prt
D_*\setminus I}$ where $x\in I$ and $\sigma(I)< \eps$.
If $c=\nu(\{x\}) > 0$, then
\begin{align*}
\int_0^1 \frac{h(rx)}{h(rx)^2+(\wt h(rx)-\mu_0/\pi)^2}
\frac{dr}{1-r} \le \int_0^1 \frac{1}{(1-r)h(rx)} dr < \infty.
\end{align*}
and so  $x\in \Gamma^\theta_X$ with probability 1 by
Theorem \ref{j15.5} (vii), where $X \lra (h, \mu_0)$.

(vi)  The condition $\nu(\{x\}) > 0$ is stronger than the integrability
condition \eqref{sept6.1}. For
example,
if $h(z)=\frac{1}{\pi} \Re (1-z)^{-p} $, with $0 < p < 1$,
then
$\int_0^1 \frac{1}{(1-r)h(r)} dr < \infty$ so that
\eqref{sept6.1} holds at $x=1$. However, by \eqref{sept6.2},  the corresponding
positive measure $\nu$ satisfies $\nu(\{1\}) = 0.$

(vii) Suppose $\mu_0=0$.
If $h(x)=h(\ol x)$ for all $x\in \prt D_*$, where $\ol x$ denotes the complex conjugate of $x$, then $\wt h(r)=0$ for
$-1 < r < 1$. In this case, the integral in \eqref{sept6.1} is finite
for $x=1$ if and only if
\begin{align}\label{sept6.4}
\int_0^1\frac{1}{(1-r)h(r)} dr < \infty.
\end{align}
\end{remark}

\medskip
Condition \eqref{sept6.1} can be restated. Set
$f= \Re (1/(h+i \wt h -i \mu_0/\pi))$. Then $f$ is harmonic and
positive, so there is a positive measure $d\sigma$ such that
\begin{align*}
f(z)=\int_{\prt D_*} \frac{1-|z|^2}{|w-z|^2} d\sigma(w).
\end{align*}

\medskip
\begin{proposition}\label{sept11.1}
Condition \eqref{sept6.1} holds for $x\in \partial D_*$ if and only if
\begin{align}\label{sept11.2}
\int_{\prt D_*} \frac{1}{|w-x|}  d\sigma(w) < \infty.
\end{align}
\end{proposition}

For example, suppose $E$ is a closed subset of $\prt D_*$ of positive
length. Let $f(y)=\dist(y,E)^p$ for $y \in \partial D_*$,
where $p \in(0,1)$ is fixed. Then it is not
hard to verify that $f\in
C^p(\prt D_*)$, that is, $f$ is H\"older-continuous with exponent $p$ on $\partial D_*$.
Let the harmonic extension of $f$ to $D_*$ be also denoted by $f$.
Thus the function $f+i \wt
f$ is analytic on $D_*$, extends to be continuous on the closed disk $\ol D_*$, and hence
the zero set $Z=\{y\in\prt D_*: f(y)=\wt f(y)=0\}\subset E$ has zero length
(see \cite[page 51]{Hoff}). Set $h+i \wt h = 1/(f+i\wt f)$. Then $h$
is positive and harmonic on $D_*$. Since $f\in
C^p(\prt D_*)$,   $\widetilde f \in C^p(\prt D_*)$
by Theorem II.3.2 in \cite{GarMar}. Thus $h=f/(f^2+\widetilde{f}^2)$
is continuous up to $\prt D_*\setminus Z$, and so $h$ tends to $0$ as
$z\to E\setminus Z$.
The function $h$ tends to a positive number at each point of
$\prt D_*\setminus E$.  The positive measure $ \sigma (dy)= f(y) |dy|$ on $\partial D_*$
satisfies \eqref{sept11.2} for each $x\in E$,
since $f(y)\le |x-y|^p$ for every $x\in E$.
Let $\theta \in \calT \leftrightarrow (h,0)\in \calH$ and $X\lra (h, 0)$
be the corresponding ORBM.
By Theorem \ref{j15.5} (vii) and Proposition \ref{sept11.1},
for every $x\in E$, $x \in \Gamma^\theta_X$ with probability $1$. Note also that the integral in
\eqref{sept11.2} is
infinite for each point $x\in \prt D_*\setminus E$, since $f$ is
positive and continuous there. So for every $x \in \prt
D_*\setminus E$, this ORBM does not hit $x$ with probability 1.
The function $\theta$ is continuous on $\prt D_*\setminus
Z$, and $|\theta| < \pi/2$ off $E$.  We can take $E$ to have no
interior in $\prt D_*$, so   $|\theta| <\pi/2$ on a dense open
set.

\medskip
Recall that if $f:D_*\to D_*$ is a conformal map of $D_*$ onto itself, then there exist $\theta_0\in [0, 2\pi)$ and $w_0\in  D_*$ such
that $f(z)=e^{i\theta_0} \frac{z-w_0}{1-\overline {w_0}z}$. So in particular $f$ extends continuously to $\ol D_*$
as a smooth homeomorphism. The following result establishes conformal invariance of ORBM on the unit disk.

\medskip

 \begin{theorem}\label{T:3.8}
Suppose   $\theta\in \calT$ and $X$ is an ORBM on $D_*$ with reflection angle $\theta$.
Suppose  $f: D_* \to D_*$ is a conformal map of $D_*$ onto $D_*$.
Define  for  $ t\in[0,\infty)$,
\begin{align}\label{e:3.17}
c(t)  = \int_0^t |f'(X_s)|^2 ds  \quad \hbox{ and } \quad Y_t  = f(X_{c^{-1}(t)}) .
\end{align}
Then $Y$ is an ORBM on $D_*$ with reflection angle $\theta \circ f^{-1}\in \calT$.
Equivalently, if $(h, \mu_0)  \in \calH  \leftrightarrow \theta$, then $Y$ is the
ORBM on $D_*$ parametrized by $(\bar h,  \bar \mu_0)  \in \calH$,  where
$\bar h (z)= h (f^{-1}(z))/(\pi h (f^{-1}(0)))$
and $\bar \mu_0= \mu(f^{-1}(z))/h (f^{-1}(0))$.
Here $\mu (w)$ is the harmonic function defined by  \eqref{aug28.2}.
\end{theorem}

\medskip

\subsection{Excursion Reflected Brownian Motions}
\label{subs-main3}

We now address the question that was left unanswered in Section \ref{subs-main2},
namely whether there  exists a process on $D_*$ associated with a purely 
tangential angle of reflection, e.g.,
$\theta \equiv \pi/2$.  In Theorem \ref{j18.1} we will show that such a  
process does indeed exist and can be obtained as a suitable limit of
ORBMs in $D_*$ corresponding to angles of reflection $\theta \in
\calT$.   We refer to this process as excursion reflected Brownian
motion (ERBM).

We will first define ERBM more generally, in  a bounded simply
connected domain $D$
with variable excursion intensity $\nu(dx)$, where $\nu$ is a measure
on $\prt D$. Our construction resembles a process introduced in
\cite{FT,CFY,CFdarn} and called ``Brownian motion extended by
darning'' (BMD), and defined simultaneously in \cite{Lawler} under the
name of ERBM.  
We will use some concepts from excursion theory reviewed in Section \ref{sec:exc}.

\begin{definition}\label{jan22.1}
Suppose that $\nu(dx)$ is a finite positive measure on $\prt D$. Let $H^x$ be the standard Brownian excursion law in $D$ for excursions starting at $x\in \prt D$. If $D=D_*$ then we normalize the $\sigma$-finite measures $H^x$, $x\in \prt D_*$, so that all of them can be obtained from $H^1$ by rotation around 0. Let $\bDelta$ be a cemetery state and $\calC= \calC_D$ denote the family of all functions
$\omega:[0,\infty)\to \ol D\cup\{\bDelta\}$ such that
$\omega(0)\in \prt D$, $\omega$ is continuous up to its lifetime $\zeta<\infty$, and $\omega(t) = \bDelta$ for $t\geq \zeta$.
Let $\lambda$ denote the Lebesgue measure on $\R_+=[0,\infty)$
and let $\mathcal P$ be the Poisson point process on $\R_+\times \calC$ with characteristic measure $\lambda\times \int _{\prt D} H^x \nu(dx) $. With
probability 1, there are no two points with the same first
coordinate so the elements of $\mathcal P$ may be
unambiguously denoted by $(t,\exc_t)$. Let
\begin{equation*}
  \zeta_t=\inf\{s>0:\exc_t(s)=\bDelta\}.
\end{equation*}
Let $\sigma_v = \sum_{s\leq v}\zeta_s$ and $\sigma_{v-} = \sum_{u< v}\zeta_u$ for $v\geq 0$.

Let $D^\prt:=D\cup \{ \partial \}$ be a one-point compactification
of $D$ obtained by identifying the usual boundary $\prt D$ with a single point $\prt$.

If $D = D_*$ then
the lifetimes of excursions of the process $\calP$ have the same structure as those of the symmetric reflected Brownian motion (with the normal reflection), so $\sigma_v < \infty$ for all $v<\infty$ and $\lim_{v\to\infty} \sigma_v = \infty$, a.s. For all domains $D$ for which the last two statements are true,
with probability 1, for every $t\geq 0$, the formula $r = \inf \left\{ v\geq 0: \sigma_v \geq t\right\}$ defines a unique $r\geq 0$.
For $t\geq 0$ let
\begin{equation*}
  X_t=
  \begin{cases}
    \exc_r(t - \sigma_{r-}),&\text{if $\sigma_{r-} < \sigma_r$ and
    $t \in [\sigma_{r-} , \sigma_r)$},\\
    \prt,&\text{otherwise}.
  \end{cases}
\end{equation*}
 With probability one, $X$ is a conservative process taking values in $ D^\prt$.
We will call the process $X$ (or its distribution) excursion reflected Brownian motion (ERBM) in $D$ with excursion intensity $\nu$.
In general, $X$ is not a Hunt process on $\ol D$ as it does not have the quasi-left  continuity property at the first hitting time of $\partial D$,
which is a predictable stopping time. However,  $X$ is a conservative continuous Hunt process on $D^\prt$.
\end{definition}

\begin{remark} \label{j3.1}
(i)
If $H^x$ is a standard Brownian excursion law in $D$ and $c>0$ is a constant then $cH^x$ is also a standard Brownian excursion law in $D$. We talked about ``the'' standard excursion laws above because all standard excursion laws in a simply connected domain corresponding to a given boundary point are constant multiples of each other.

(ii) For any strictly positive function $a(x)$ on the boundary of $D$, ERBM corresponding to $(a(x)\nu(dx), (1/a(x))H^x)_{x\in \prt D}$ has the same distribution as ERBM determined by $(\nu(dx), H^x)_{x\in \prt D}$.
Hence, one has to specify both $\nu$ and the normalization of the excursion laws $H^x$ to identify ERBM uniquely.

(iii) It may be surprising at the first sight but it is easy to see that for any constant $c>0$, $(\nu(dx), H^x)_{x\in \prt D}$ and $(c\nu(dx), H^x)_{x\in \prt D}$ define the same ERBM. So we may assume that $\nu$ is a probability measure.

(iv) Combining the last two remarks, it is easy to check that if ERBM $X$ can be represented by $(\nu(dx),  H^x)_{x\in \prt D}$ and also by $( \nu_1(dx),  H_1^x)_{x\in \prt D}$,  then
$$
( \nu_1(dx), H_1^x)_{x\in \prt D}\equiv (c a(x)\nu(dx), (1/a(x)) H^x)_{x\in \prt D}
$$
 for some positive function $a(x)$ and some positive constant $c$.

(v) When $D$ is the unit ball $D_*$,  the ERBM in $D_*$ with excursion intensity $\nu$ being the uniform measure on $\prt D_*$
has the same distribution as the BMD  studied in \cite{FT,CFY,CFdarn}; see \cite[Remark 7.6.4]{CF1} where this identification is proved
when $D$ is the exterior of the unit ball. When $D$ is the exterior of the unit ball,
the process  also has the same distribution as the ERBM introduced in
\cite{Lawler}; see \cite[Example 6.3]{CF2}.

(vi)
To make things simple, we will assume in theorems on ERBM that $\prt
D$ is a Jordan curve (in other words, $D$ is a simply connected Jordan
domain). This is equivalent to saying that if $f: D_* \to D$ is a
one-to-one and onto analytic mapping then $f$ can be extended to be continuous
and one-to-one on $\ol D_*$.
We believe that all our results hold for arbitrary
bounded simply connected domains because ``exotic'' points on the boundary are negligible from the point of view of excursion theory.

(vii) The reader who wishes to learn more about potential theoretic properties of domains and their relationship to geometric properties may consult \cite{Ohtsuka} for a discussion of ``prime ends.'' The Martin boundary is presented in \cite{Doob84}; in particular, the identification of the Martin boundary and prime ends  is mentioned in \cite[1 XII 3]{Doob84}.
The Martin topology and boundary in simply connected planar domains are conformally invariant, see \cite[Thm. 9.6]{Pom}.

(viii)
If $D$ is a Jordan domain and $x\in \prt D$, then the Martin kernel $K_x(\,\cdot\,)$ is the unique, up to a multiplicative constant, positive harmonic function in $D$ that vanishes everywhere on the boundary except at $x$.
The density of the expected occupation measure for $H^x$ is a constant multiple of the Martin kernel $K_x(\,\cdot\,)$ by \cite[Prop. 3.4]{BurBook}.
\end{remark}

\begin{proposition}\label{P:3.10}
Suppose  $D\subset \C$ is a bounded simply connected Jordan domain.

\begin{enumerate}[\rm (i)]
\item
Let $ X$ be an ERBM constructed from $(\nu, H^x)_{x\in D}$, where $\nu$ is a probability measure on $\partial D$.
Then $X$ has a  unique stationary distribution whose density is proportional to
$h(y) = \int _{\prt D} K_x(y) \nu(dx) $.
\item
For every positive harmonic function $h$ in $D$ with $\|h\|_{L^1(D)}=1$ there exists
an ERBM $X$ with the stationary density $h$.
\end{enumerate}
\end{proposition}

We say that a real-valued function $f$ defined on a subset $S$ of $\R^n$ is Lipschitz with
constant $\lambda < \infty$ if $|f(x)-f(y)|\le \lambda|x-y|$ for all
$x,y \in S$. It follows from the definitions that a Lipschitz function
is Dini continuous.

\begin{theorem}\label{j18.1}
{\rm (i)}
Consider a sequence of $C^2$ functions $\theta_k: \prt D_* \to (-\pi/2, \pi/2)$ and let $X^k$ be defined by
\begin{align}\label{n24.1}
 X^k_t = x_k + B_t + \int_0^t  \bv_{\theta_k} (X^k_s) dL^k_s,
 \qquad \hbox{for } t\geq 0.
\end{align}
Let $(h_k,\mu_{0,k}) \leftrightarrow \theta_k$ as in Lemma \ref{ju20.1}.  We make the following assumptions:
\begin{enumerate}[\rm (a)]
\item
$\theta_k$ converge to $\pi/2$ almost everywhere.
\item
For some $c_1 > -\pi/2$ and all $x$ and $k$,
$\theta_k(x) \geq c_1$.
\item
There exist $\lambda < \infty$ and $c_2>0$ such that
$h_k$ restricted to $\prt D_*$ is Lipschitz with constant $\lambda$ for every $k$,
and $h_k(x)> c_2$ for every $x$ and $k$.
\item There is a finite measure $\nu(dx)$ on $\prt D_*$ such that
$h_k(x)dx \to \nu(dx)$ weakly as measures on $\prt D_*$, when $k\to \infty$.
\item
 $\lim_{k\to\infty}\dist(x_k, \prt D_*)= 0$.
\end{enumerate}
Then the processes
$X^k$ converge in the sense of finite dimensional distributions to
ERBM $X$
corresponding to $(\nu(dx), H^x)_{x\in \prt D_*}$, where
all $H^x$ are obtained from $H^1$ by rotation around 0.

\medskip  {\rm (ii)}
Conversely, suppose that $h$ is harmonic in $D_*$, Lipschitz on $\ol D_*$ and positive on
$\ol D_*$.
Then there exists a sequence of $C^2$ functions $\theta_k: \prt D_*
\to (-\pi/2, \pi/2)$ satisfying  conditions (a)-(e)
with  $\nu(dx)= h (x) dx$  on $\prt D_*$.
ORBMs $X^k$ corresponding to $\theta_k$'s converge in the sense of finite dimensional distributions to an ERBM $X$ with the stationary density $h$.
\end{theorem}

\begin{remark}\label{jan15.1}
\begin{enumerate}[\rm (i)]
\item
The roles of $\pi/2$ and $-\pi/2$ in Theorem \ref{j18.1} can be
reversed by replacing $\theta_k(x)$ with $-\theta_k(\ol x)$. See
Remark \ref{ju20.5}(iv).

\item It is easy to see from Theorems \ref{T:3.8} and \ref{j18.1}
that if $f: D_*\to D_*$ is a conformal map and $X$ is an ERBM on $D_*$
corresponding to $(\nu(dx), H^x)_{x\in \prt D_*}$, then $f(X)$ is a time-change
of ERBM on $D_*$ corresponding to $(\nu (dx)\circ f^{-1}, H^x)_{x\in \prt D_*}$.

\item Suppose that there exists $\lambda < \infty$  such that
$h_k$ restricted to $\prt D_*$ is Lipschitz with constant $\lambda$ for every $k$.
Then it is elementary to show that there exists $c_2>0$ such that  $h_k(x)> c_2$ for every $x$ and $k$
if and only if there exists $\lambda_1 < \infty$  such that
$1/h_k$ restricted to $\prt D_*$ is Lipschitz with constant $\lambda_1$ for every $k$.
\end{enumerate}
\end{remark}

\begin{example}\label{m2.1}
Theorem \ref{j18.1} has many assumptions so it deserves a simple example to
illustrate it. Suppose $h(x)$ and $1/h(x)$ are positive Lipschitz
continuous functions on $\prt D_*$ with $\|h\|_{L^1(D_*)}=1$. Let $h(z)$ be the harmonic
extension of $h$ to $D_*$.
Suppose also that $\mu_{0,k}\to \infty$ as $k\to \infty$.
Then $(h,\mu_{0,k}) \leftrightarrow \theta_k \in \calT$ as in Theorem
\ref{aug27.0}.  If $h_k=h$ for all $k$
then  $(h_k,\mu_{0,k})$ and $\theta_k$ satisfy assumptions (a)-(e) of Theorem \ref{j18.1}.
\end{example}

\subsection{ORBMs in Simply Connected Domains}
\label{subs-main4}

We will use conformal mappings to construct  ORBMs  in arbitrary simply connected domains.
In the following, we will usually use $X$ to denote ORBM in the disk $D_*$ and $Y$ to denote ORBM in other domains.

\begin{theorem}\label{j15.7}
Suppose that $f$ is a one-to-one analytic function
mapping $D_*$ onto a simply connected domain $D\subset \C$. Suppose
that $\theta\in \calT$, $\theta \leftrightarrow (h,\mu)$, let
$\bar h = h \circ f^{-1}$ and assume that $\bar h$ is in
$L^1(D)$.  Let $X \lra \theta$ be ORBM in $D_*$ and define
\begin{align}\label{j18.4}
c(t) &= \int_0^t |f'(X_s)|^2 ds,  \qquad \text{  for  } t\geq 0, \\
\zeta &= \inf\{t\geq 0: c(t) = \infty\},\label{feb5.1}\\
Y_t &= f(X_{c^{-1}(t)}), \qquad \text{  for  } t\in[0,\zeta).\label{feb5.2}
\end{align}
We will call $Y$ an ORBM in $D$. The following hold.
\begin{enumerate}[\rm (i)]
\item
With probability 1, $\zeta = \infty$.

\item
The process $Y$ is an extension of killed Brownian motion in $D$ in the sense that for
every $t\geq 0$ and $\tau_t = \inf\{s\geq t: Y_s \in \prt D\}$, the process $\{Y_s, s\in [t, \tau_t)\}$ is
Brownian motion killed upon exiting $D$.

\item
The process $Y$ has a stationary distribution with   density   $\wh h  = \bar h/ \|\bar h\|_{L^1 (D)}$.

\item  Recall that $\mu$ is the function given by \eqref{aug28.2}.
For $z\in D$, let $\barg^* (Y_t-z) = \barg^* (X_{c^{-1}(t)}-f^{-1}(z))$ for all $t$. Then, for every $z\in D$, a.s.
\begin{align}\label{f6.2}
\lim_{t\to \infty} \frac{\barg^* (Y_t-z) }{t} &= \frac{\mu(f^{-1}(z))}{ \|\bar h\|_{L^1(D)}}.
\end{align}
\item
Suppose that $\mu_0\in \R$ and $\wh h$ is a positive harmonic function in $D$ with
$\|\wh h\|_{L^1(D)}=1$.
Then there exists an ORBM $Y$ in $D$ with the following properties.
\hfil\break {\rm (a)}  The stationary distribution of $Y$ is $\wh h(x) dx$.
\hfil\break {\rm (b)} Set $g = f^{-1}$ and define
\begin{align}\label{jan18.1}
b(t) &:= \int_0^t |(g'(Y_s)|^2 ds, \qquad t\geq 0,\\
X_t &:= g(Y_{b^{-1}(t)}), \qquad t\geq 0,\label{feb17.1} \\
\barg^* Y_t &:= \barg^* X_{b(t)}, \qquad t\geq 0.\label{feb17.2}
\end{align}
Since $\wh h \circ f$ is a positive harmonic function on $D_*$, ~$\|\wh
h \circ f\|_1 = \pi \wh{h} \circ f(0)< \infty$. Set $h_1 =\wh h \circ f /\|\wh h\circ f\|_1$ and let
$\mu\in \calR\leftrightarrow (h_1,\mu_0)\in \calH$.
Then $X$ is the ORBM in $D_*$ parametrized by $(h_1, \mu_0)$ and
\eqref{f6.2} holds with $\bar h=h_1\circ f^{-1}=\wh h/\|\wh h\circ f\|_1$.
\item  (Consistence)
If $D$ has a smooth boundary and $\theta$ is $C^2$
 then the distribution of $Y$ is the same as that of the process identified in Theorem \ref{j25.1} (ii) relative to $\theta\circ f^{-1}$.
\end{enumerate}
\end{theorem}

\begin{remark}\label{jan17.3}
(i)
The quantity $\barg (Y_t-z)$ has a natural interpretation when $Y$ is continuous,
namely, $\barg (Y_t-z) - \barg (Y_0-z)$ is the number of windings of $Y$ around $z$ over the time interval $[0,t]$. The quantity $\barg^* (Y_t-z)$ is obtained from $\barg (Y_t-z)$ by
 discarding
(the windings of) all excursions of $Y$ which make more than a full loop around $z$ (from endpoint to endpoint of the excursion, not within the excursion).
Our definition of $\calE^L_s$ was chosen to make this simple geometric interpretation of $\barg^* (Y_t-z)$ possible.

Unfortunately, when $Y$ is not continuous, $\barg^* (Y_t-z)$ does not have a simple intuitive interpretation because the definition of $\barg$ in $D_*$ depends on $\theta$.

(ii)
The process $Y$ constructed in
Theorem \ref{j15.7}
will be called ORBM in $D$. The family of ORBMs in $D$ can be
parametrized either in terms of pairs $(\theta,f)$ or triplets $(\wh
h, \mu_0,f)$, so we will write $Y \lra (\theta,f)$ or $Y \lra (\wh h,
\mu_0,f)$.
The function $f$ provides a way to parametrize $\prt D$, in a sense.

(iii) If $\mu\in \calR \leftrightarrow (h,\mu_0)\in \calH$ then we say
that $\mu\circ f^{-1}(z)$
is the rotation rate about $z\in D$ for the process $Y$
given by \eqref{feb5.2}. If $\mu_1$ is a harmonic function defined on
$D$, let $\wt \mu_1$ be the
harmonic conjugate of $\mu_1$ vanishing at $f(0)$. Then $\wt{\mu_1 \circ
f}$ is a harmonic function on $D_*$ vanishing at $0$ and
$\wt{\mu_1\circ f}= \wt{\mu_1}\circ f$. By Theorem \ref{aug27.0} and Theorem
\ref{j15.7}, $\mu_1$ is a rotation rate for an ORBM if
and only if $\wt \mu_1 (z) > -1 $ for all $z\in D$.

(iv) Suppose that $f$ is a conformal mapping from a bounded simply connected planar domain $D_1$ to another bounded simply connected planar domain $D_2$. Let
${\calK}(D_1, D_2,f)$ be the family of positive integrable harmonic functions $h$ in $D_1$ such that
$ h\circ f^{-1}\in L^1( D_2)$.
By Theorems  \ref{T:3.8} and \ref{j15.7}, $f$ establishes a correspondence between a subfamily  of  ORBMs on $D_1$ that have the density of
stationary distribution in ${\calK}(D_1, D_2, f)$ and a subfamily of ORBMs on $D_2$
that have the density of stationary distribution in ${\calK}(D_2, D_1, f^{-1})$.
The subfamilies are non-empty because they always contain normally reflected Brownian motions.
Theorem \ref{j17.10} below gives some sufficient conditions on the integrability of positive harmonic functions in domains.
The correspondence between ORBMs on different planar domains need not extend to all ORBMs on either side because the assumption $\bar h \in L^1(D)$ of Theorem \ref{j15.7} does not hold for some $h$ and $f$; see Example \ref{E:4.2} below.

(v) There exist processes in $D$ that are extensions of Brownian motion in $D$, which have a stationary density  and a ``limiting rate of rotation'' $\mu_0$ and which are not ORBM's. An example of such a process is the conformal image of reflected Brownian motion in $D_*$ with diffusion on the boundary (see a Ph.D.~thesis \cite{card} devoted to this class of processes).
\end{remark}

\begin{theorem}\label{j18.3}
Suppose that $D\subset \C$ is a simply connected bounded Jordan domain and $f$ is a conformal mapping from  $D_*$ onto  $D$,
which, by  Carath\'eodory's theorem,  necessarily extends to a homeomorphism from $\ol D_*$ onto $\ol D$.
Consider a sequence of $C^2$ functions $\theta_k: \prt D_* \to (-\pi/2, \pi/2)$ and processes $X^k$ which
satisfy \eqref{n24.1} and assumptions (a)-(e) of Theorem \ref{j18.1}. Let $(h_k,\mu_k) \lra \theta_k$ and
let
$c_k(t),\zeta_k$ and $Y^k$ be defined relative to $\theta_k,f$ and $X^k$ as in Theorem \ref{j15.7}.

Let $\nu $, $h$ and $X$ be defined as in Theorem \ref{j18.1}.
Let
$c(t),\zeta$ and $Y$ be defined relative to $\theta,f$ and $X$ as in Theorem \ref{j15.7}.
In (i)-(iv) below,   $\bar h := h \circ f^{-1}$ is assumed to be in $L^1(D)$.

\begin{enumerate}[\rm (i)]
\item Almost surely,  $\zeta_k=\infty $ for every $k\geq 1$  and $\zeta=\infty$.
\item The process $Y$ is an ERBM in $D$ corresponding to
    $(\bar\nu(dx),\bar H^x)_{x\in \prt D}$ with excursion
    intensity $\bar\nu$ defined by $\bar \nu(A) =
    \nu(f^{-1}(A)) $ for $A\subset \prt D$, and excursion
    laws $\bar H^x$ normalized so that the density of the
    expected occupation time for $\bar H^x$ is the Martin
    kernel $K_x(\,\cdot \,)$  in $D$ normalized by
    $K_x(f(0)) = 1$.
\item
Processes $Y^k$ converge to $Y$ in the sense of convergence of finite dimensional distributions.
\item The process $Y$ has a stationary distribution with
    the density $\wh h = \bar h / \|\bar h\|_{L^1(D)}$.
\item
For every positive harmonic function $\wh h$ in $D$ with $\|\wh
h\|_{L^1(D)}=1$
 such that $ \wh h \circ f$ is Lipschitz on $\ol D_*$ and strictly positive on $\prt D_*$,
  there is a sequence of $C^2$ functions $\theta_k: \prt D_*
\to (-\pi/2, \pi/2)$
such that $Y^k $ and $Y$ can be constructed as in the initial part of the theorem and
the stationary measure for ERBM $Y$ has density $\wh h$.
\end{enumerate}
\end{theorem}

The next two theorems show that ORBM in an arbitrary domain (possibly with a fractal boundary) can be approximated by ORBMs in smooth domains where the oblique angle of reflection has a natural interpretation. This provides a justification of the name ``obliquely reflected Brownian motion'' for processes in domains with rough boundaries.

\begin{theorem}\label{j17.3}
Suppose that $D\subset \C$ is a simply connected
Jordan domain, $y_0\in D$ and $f$ is a conformal mapping from $D_*$
onto  $D$ which, necessarily, has a
one-to-one continuous extension to $\ol D_*$.
Let $D_k$ be simply connected domains with smooth boundaries such that $y_0 \in D_k \subset D_{k+1}\subset D$ for all $k$ and $\bigcup_k D_k = D$.
Let $f_k: D_* \to D_k$ be conformal mappings such that $f_k^{-1}(y_0) = f^{-1}(y_0)$ and $f_k \to f$ as $k\to \infty$.

Suppose that $\mu_0\in \R$, $\bar h \in
L^1(D)$ is positive and harmonic with $\|\bar h\|_{L^1 (D)}=1$,
and $\bar h \circ f$ is  strictly positive on $\prt D_*$.
  Let
$Y$ be the process constructed as in Theorem \ref{j15.7} (v), relative
to $D,f,\mu_0$ and $\bar h$,
with $Y_0=y_0$.
Let $\bar h_k = \bar h / \|\bar
h\|_{L^1(D_k)}$. Let $Y^k$ be defined in the same way that $Y$ was
defined, relative to $D_k,f_k,\mu_0$ and $\bar h_k$,
with $Y^k_0=y_0$.
Then $Y^k$ converge weakly to $Y$
in $M_1^\calT$ topology.
\end{theorem}

\medskip

The following   concrete example shows how one can approximate a general ORBM in $D$
by ORBMs in an increasing sequence of smooth domains with smooth reflection angles.
Suppose that  $Y \lra (\theta,f)  \lra (\bar h, \mu_0, f)$.
Take any strictly increasing sequence of positive numbers $r_k$ that increases to 1.
Let $D_k= f(B(0, r_k))$ and $f_k(z) = f(z/r_k)$. It is easy to see that
$D_k$ is a smooth subdomain of $D$ and
$f_k$ is a conformal mapping from $B(0, r_k)$ to $D$.
Clearly $h_k (z):= \bar h (f(r_k z))$ is a positive harmonic function on $D_*$ that is smooth
on $\ol D_*$.
By Theorem \ref{aug27.0}, $  \theta_k \lra ( h_k/ h_k (0), \mu_0)$  is smooth on $\partial D_*$.
Thus $\bar \theta_k  (w) =  \theta_k ( f^{-1}(w)/r_k) \in (-\pi/2, \pi/2)$ defines a smooth function on $\partial D_k$.
Let $Y^k$ be the ORBM on $D_k$ with reflection angle $\bar \theta_k$ constructed in Theorem \ref{j25.1}(ii) .
Theorem \ref{j17.3} asserts that $Y^k$ converge weakly to ORBM $Y $ on $D$
in $M_1^\calT$ topology.

\medskip

\begin{theorem}\label{j18.8}
Suppose that $D\subset \C$ is a simply connected
Jordan domain, $y_0\in D$ and $f: D_*
\to D$ is a conformal mapping  which, necessarily, has a
one-to-one continuous extension to $\ol D_*$.
Let $D_k$ be simply connected
domains with smooth boundaries such that $y_0 \in D_k \subset D_{k+1}\subset D$ for all $k$ and $\bigcup_k D_k = D$.
Let $f_k: D_* \to D_k$ be one-to-one analytic functions such that $f_k^{-1}(y_0) = f^{-1}(y_0)$ and $f_k \to f$ as $k\to \infty$.

Suppose that $\theta: \prt D \to (-\pi/2, \pi/2)$ is a
 continuous function.
 Let $\theta_*: \prt D_* \to (-\pi/2, \pi/2)$ be defined by $\theta_* = \theta \circ f$.
Let $Y$ be ORBM in $D$, such that $Y \lra (\theta_*, f)$ and $Y_0=y_0$.

For every $k$,
let $g_k:\prt D_k \to \prt D$ be a measurable function such that for every $x\in \prt D_k$, $g_k(x)=y\in\prt D$ and $|x-y| = \dist(x, \prt D)$. Let $\theta_k(x)= \theta(g_k(x))$ for $x\in \prt D_k$.
Let
$Y^k$ be the ORBM in $D_k$ such that $Y^k \lra (\theta_k, f_k)$ and $Y^k_0=y_0$.
Then $Y^k$'s converge weakly in $M_1$ topology to $Y$.
\end{theorem}

The assumption that $\bar h \in L^1(D)$ applied in Theorem
\ref{j15.7} is sufficient but not necessary.
We will sketch an
argument illustrating this claim in Example \ref{feb15.1}
below.
In other words, the construction given in Theorem
\ref{j15.7} generates a process $Y_t$ for all $t\geq 0$ for some
domains $D$ and functions  $\bar h$ such that $\| \bar h \|_{L^1(D)} =
\infty$. Of course, in such a case no constant multiple of $\bar h(x) dx$ can be the stationary (probability) distribution for $Y$, although it can be an invariant measure.

In view of the assumption of integrability of $\bar h$ made
in Theorems \ref{j15.7} and \ref{j17.3}, it would be useful to
have an effective tool to check whether a given harmonic
function is in $L^1(D)$. We do not have such a test and we
doubt that a universal test of this kind exists. We do have some sufficient conditions for integrability of
positive harmonic functions.
First, recall
Theorem \ref{j27.4}. It contains a criterion for a harmonic function $h$
in $D_*$ corresponding to an angle of oblique reflection
$\theta$ to be bounded. A ``push'' $h\circ f^{-1}$ of such
function to a bounded simply connected domain is also bounded,
and hence integrable.
Second,
Theorem \ref{j17.10} below presents some examples of
domains where all positive harmonic functions are integrable.

Recall
that a function $\psi: \R \to \R$ is
Lipschitz, with constant $\lambda < \infty$, if $|\psi(x) - \psi(y)| \leq \lambda |x-y|$ for all
$x,y \in \R$.
A domain $D\subset \R^2$ is said to be Lipschitz, with
constant $\lambda$, if there exists $\delta >0$ such that, for
every $x \in \prt D$, there exists an orthonormal
basis $(e_1, e_2)$
and a Lipschitz function $\psi: \R \to \R$, with constant $\lambda$, such that
\begin{align*}
\{y \in \ball(x,\delta) \cap D\} =
\{y\in\ball(x,\delta): \psi(\langle y,e_1\rangle)<\langle y,e_2\rangle\}.
\end{align*}

We recall the definition of a John domain following \cite{Aikawa}.
Let $\delta_D(x) = \dist(x, \prt D)$ and $x_0 \in D$. We say that $D$ is a John domain with John constant
$c_J > 0$ if each $x \in D$ can be joined to $x_0$ by a rectifiable curve $\gamma$ such that
$\delta_D(y) \geq c_J \ell(\gamma(x, y))$ for all $y \in \gamma$, where
$\gamma(x, y)$ is the subarc of $\gamma$ from $x$ to $y$ and $\ell(\gamma(x, y))$ is the length of $\gamma(x, y)$.
The first two parts of the following theorem follow from Theorems 1 and 2 of \cite{Aikawa}.

\begin{theorem}\label{j17.10}
\begin{enumerate}[\rm (i)]
\item  {\rm  (\cite[Thm.~1]{Aikawa})}  If $D\subset \R^2$ is a bounded John domain with John constant $c_J \geq 7/8$ then all positive harmonic functions in $D$ are in $L^1(D)$.

\item {\rm (\cite[Thm.~2]{Aikawa})}  If $D\subset \R^2$ is a bounded Lipschitz domain with constant $\lambda<1$ then all positive harmonic functions in $D$ are in $L^1(D)$.

\item
There exists a bounded Lipschitz domain $D$ with constant $\lambda=1$ and a positive harmonic function $h$ in $D$ which is not in $L^1(D)$.
\end{enumerate}
\end{theorem}

\section{Proofs}\label{sec:proofs}

\begin{proof}[\bf Proof of Theorem \ref{j25.1}]

(i) This part is a special case of \cite[Thm.~2.6]{HLS}.

(ii)
Let $X$ be the unique pathwise solution of \eqref{j13.1}.
Then by It\^o's formula,
$f(X_t) - \frac 12 \int_0^t \Delta f (X_s) ds$
is a submartingale under $\P_z$ for every $z\in \ol D$ and $f\in \calC$.
Thus, in view of (i), $(X, \P_z)$ is the unique solution to the submartingale problem
\eqref{f3.1}.

(iii)  This part is known, see, e.g., \cite{KKY}. For the reader's convenience, we give a sketch of the Dirichlet form approach to the construction of ORBM.  The argument
given below works in higher dimensions as well.
In $C^2$-smooth domains with $C^2$-smooth reflection angle,
  it is enough to construct ORBM  locally nearly the boundary and then patch
  the pieces together.
Thus by locally flattening the boundary,  we may and do assume
that $D=\HH$, the upper half space.
Let $\bv (x)= (v_1 (x), 1)$ for $x\in \partial \HH$ with $v_1(x):=\tan \theta (x)$.
Consider a non-symmetric bilinear form $(\calE, \calF)$ on $L^2(\HH, dz)$, where
\begin{eqnarray*}
\calF &=& \left\{ f\in L^2(\HH, dz): \nabla f \in L^2(\HH, dz) \right\}, \\
\calE (f, g)&=& \int_{\HH} \nabla f(z) \cdot \nabla g (z) dz -\int_{\prt \HH} v_1 (x)
\frac{\partial f(x, 0)}{\partial x} g(x, 0) dx \quad \hbox{for } f, g\in \calF.
\end{eqnarray*}
 Let $\calE^0(f, g)=\int_{\HH} \nabla f(z) \cdot \nabla g (z) dz$, and for $\alpha >0$,
 $$
 \calE^0_\alpha (f, g) := \calE^0(f, g)) + \alpha (f, g)_{L^2(\HH; dz)} \quad \hbox{and}
 \quad   \calE_\alpha (f, g) := \calE (f, g)) + \alpha (f, g)_{L^2(\HH; dz)}.
 $$
Observe that for $f \in C_c^2 (\bar \HH)$, by the  integration by parts formula,
$$  \left| \int_{\prt \HH} v_1 (x)
\frac{\partial f(x, 0)}{\partial x} f(x, 0) dx \right|  =  \frac12 \left| \int_{\prt \HH}   v_1' (x)
f(0, x)^2 dx \right| \leq \frac12 \|   v_1' \|_\infty \| f(x, u)\|_{L^2(\partial \HH, dx)}^2.
$$
By the boundary trace theorem, for every $\eps >0$ there is $C_\eps >0$ such that
$$
\| f(x, u)\|_{L^2(\partial \HH, dx)}^2 \leq \eps \calE^0 (f, f)+ C_\eps \| f \|_{L^2(\HH; dz)}^2
\quad \hbox{for } f \in \calF.
$$
It follows from the above two displays that there are constants $\alpha>0$ and $C_0\geq 1$ such that
$$ C_0^{-1} \calE^0_1(f, f)\leq \calE_\alpha (f, f) \leq C_0 \calE^0_1(f, f)
$$
for every $f\in C_c^2(\ol \HH)$ and hence for every $f\in \calF$.
On the other hand, for $f, g \in C_c^2 (\bar \HH)$,
\begin{eqnarray}
&&  - \int_{\prt \HH} v_1 (x)  \frac{\partial f(x, 0)}{\partial x} g(x, 0) dx  \nonumber \\
&=&  - \int_{\prt \HH} v_1 (x)  \int_0^\infty \frac{\partial }{\partial y} \left( \frac{\partial f(x, y)}{\partial x} g(x, y) \right) dy dx
\nonumber  \\
&=&  - \int_{ \HH} v_1 (x)   \frac{\partial f(x, y)}{\partial x} \frac{\partial g(x, y) }{\partial y}   dy dx
- \int_{ \HH} v_1 (x) g(x, y)   \frac{\partial^2 f(x, y)}{\partial x \partial y} dy dx  \nonumber \\
&=&     \int_{ \HH} v_1 (x)  \left(  \frac{\partial f(x, y)}{\partial x} \frac{\partial g(x, y) }{\partial y} -
\frac{\partial f(x, y)}{\partial y} \frac{\partial g(x, y) }{\partial x}  \right)  dy dx  \nonumber \\
&&  -  \int_{ \HH}   v_1' (x) g(x, y) \frac{\partial g(x, y)}{\partial y} dy dx.
\label{e:4.1a}
\end{eqnarray}
Thus, with $C_1=2\|v\|_\infty,+\|v'\|_\infty $,
$$ \left| \int_{\prt \HH} v_1 (x)  \frac{\partial f(x, 0)}{\partial x} g(x, 0) dx \right|
\leq C_1 \calE^0_1 (f, f)^{1/2} \calE^0_1(g, g)^{1/2} \quad \hbox{for } f, g \in C_c^2(\bar \HH).
$$
Hence, the bilinear form $(\calE, \calF)$ satisfies the sector condition: there is a constant $C_2\geq 1$ such that
$$
| \calE(f, g) |\leq C_2 \calE_\alpha (f, f)^{1/2}  \calE_\alpha (g, g)^{1/2}
\quad \hbox{for } f, g \in \calF.
$$
Moreover, by increasing the value of $\alpha$ if needed, we have from \eqref{e:4.1a} that for every $f\in C_c^2(\bar \HH)$,
$$ \calE   (f, f-(0\vee f )\wedge 1 ) \geq 0 \qquad \hbox{and} \qquad \calE_\alpha  (f-(0\vee f )\wedge 1, f) \geq 0 .
$$
Thus $(\calE, \calF)$ is a regular non-symmetric Dirichlet form on $L^2(\bar \HH; dz)$.
Let $X$ be the Hunt process on $\bar \HH$ associated with $(\calE, \calF)$.
Then one can use  stochastic analysis for  non-symmetric Dirichlet forms to show that
$X$ satisfies the SDE \eqref{j13.1} for quasi-every starting point $x\in \bar \HH$
(see \cite{KKY}).
Since $X$ behaves like Brownian motion inside $\HH$, we can refine the result
to allow $X$ to start from every point $x\in \HH$
and conclude that \eqref{j13.1} holds for such $X$.
\end{proof}

\medskip

\begin{proof}[\bf Proof of Theorem \ref{f3.2}]

(i) This part of our theorem is a special case of \cite[Thm.~2.18]{HLS}.

\smallskip

(ii) Almost sure continuity of $X$ follows from \eqref{j13.1}.

Recall that we are assuming that $\theta: \prt D_* \to (-\pi/2, \pi/2)$ and $\theta \in C^2$. It follows from \cite[Cor. II.3.3]{GarMar}
that $h$ is $C^{2-\eps}$  on $\ol D_*$ for every $\eps>0$.

Let $Q$ denote the probability measure on $D_*$ with density $h(z)$.
 We will show that
\begin{equation}\label{e:4.1}
 \E_Q \left[\int_0^1 g(X_s) dL_s \right]= \int_{\prt D_*} g(x) (h(x)/2) dx
\end{equation}
for every continuous function $g$ on $\prt D_*$.
Fix any continuous function $g$ on $\prt D_*$. Its harmonic extension to $\ol D_*$ (also denoted $g$) is continuous on $\ol D_*$.
Then for $\eps\in(0,1)$,
\begin{equation}\label{e:4.2}
\E_Q \left[\int_0^1 \frac1\eps \bone _{\{1-\eps < |X_s|<1\}} g(X_s) ds\right] = \int_{D_*} \frac1\eps \bone _{\{1-\eps < |z|<1\}} g(z) h(z)dz.
\end{equation}
By continuity and boundedness of $g$ and $h$, the limit of the right hand side, as $\eps \to 0$, is equal to
$\int_{\prt D_*} g(x) h(x) dx$. It is standard to show that $\int_0^1 \frac1\eps \bone _{\{1-\eps < |X_s|<1\}} g(X_s) ds$ converges to
$2\int_0^1 g(X_s) dL_s$ in distribution as $\eps \to 0$. We claim that the family
\begin{align}\label{d8.2}
\left\{\int_0^1 \frac1\eps \bone _{\{1-\eps < |X_s|<1\}}  g(X_s) ds, \eps \in (0,1/2)\right\}
\end{align}
is uniformly integrable.
Since $g$ is bounded, it   suffices to prove uniform integrability of
the family
$\left\{\int_0^1 \frac1\eps \bone _{\{1-\eps < |X_s|<1\}}  ds, \eps
  \in (0,1/2)\right\}$. If we denote by $\calL^a_t$ the local time of
the  two-dimensional Bessel process on $[0,1]$ reflected at 1, then the distribution of
$\int_0^1 \frac1\eps \bone _{\{1-\eps < |X_s|<1\}}  ds$ is the same as $\frac1\eps \int_{1-\eps}^1 \calL^a_1 da$. The last random variable
is stochastically majorized by $\sup\{\calL^a_1: a\in [1/2, 1]\}$ for every $\eps \in (0,1/2)$. A version of the Trotter and Ray-Knight theorems shows that $\calL^a_1$ is a diffusion in $a$, so $\sup\{\calL^a_1: a\in [1/2, 1]\}$ is an almost surely finite random variable. Therefore,
the family in \eqref{d8.2} is uniformly integrable.
Taking $\eps\to 0$ in \eqref{e:4.2} yields \eqref{e:4.1}.
It follows  that
the Revuz measure of $L$ is $\frac12h(x)dx$ on $\prt D_*$, relative to the invariant measure $h(z)dz$ on $D_*$.

We will now provide a representation of $X$ using a map which is
locally conformal.
Let $D_- =\{z\in \C: \Re z <0\}$ be the left half-plane and  $f(z) = \exp(z)$  the exponential function that 
maps $D_-$  onto  $D_*\setminus\{0\}$.
For $x\in \prt D_-$ such that $f(x) = z\in \prt D_*$, define 
$\wh\bv(x) = i \tan \theta (z) - 1$. Note that $\wh \bv (x)$ is a periodic $C^2$-smooth function on $\partial D_-$ with period 
$2\pi i$. 
Suppose that $\wh x_0 \in \ol D_-$ and $\wh B$ is a two-dimensional Brownian motion.
It is known (see \cite[Theorem 4.3]{LS}) that  there is a pathwise unique solution $(\wh X, \wh L)$ to the following  Skorokhod SDE,
\begin{align}\label{n29.1}
\wh X_t = \wh x_0 + \wh B_t + \int_0^t \wh\bv(\wh X_s) d\wh L_s,
\end{align}
where $\wh X$ is a continuous process that takes values in $\overline D_-$ and $\wh L$
is a continuous non-decreasing real-valued process with $\wh L_0=0$
that increases only when $\wh X_t \in
\partial D_-$. 
The process $\wh X$ is an ORBM in $D_-$ with the oblique angle of reflection 
$\theta\circ f$.
The It\^o formula yields
\begin{align}\label{n29.2}
f(\wh X_t) &= f(\wh X_0) + \int_0^t f'(\wh X_s) d\wh B_s
+ \int_0^t f'(\wh X_s)\wh\bv(\wh X_s) d\wh L_s \\
&= f(\wh X_s) + \int_0^t f'(\wh X_s) d\wh B_s
+ \int_0^t \bv_\theta(f(\wh X_s)) d\wh L_s, \nonumber
\end{align}
where $f'$ is interpreted as the Jacobian of $f$.
Let
\begin{align}\label{n29.3}
c(t)= \int_0^t |f'(\wh X_s )|^2 ds.
\end{align}
It is not hard to show that $c(t) < \infty$ for every $t>0$, a.s. It follows that 
$$ c^{-1}(t):=\inf\{s>0: c(s)>t\}
$$
is well defined for every $t>0$ and 
the process $X_t := f(\wh X_{c^{-1}(t)})$
satisfies \eqref{j13.1}
with  Brownian motion $B_t := \int_0^{c^{-1}(t)} f'(\wh X_s) d\wh B_s$ and $L := \wh L$.
So $X$  is an ORBM in $D_*$ with the oblique angle of reflection $\theta$.
The exponential function $f(z)= \exp (z): D_- \to D_*$ is neither
one-to-one nor onto $D_*$, but it is locally conformal and maps
$\partial D_-$
onto $\partial D_*$ so we will 
refer to the fact that $X_t$ is an ORBM  as conformal invariance of
ORBM. 

Let $\sigma_t = \inf\{s\geq 0: L_s > t\} = \wh \sigma_t = \inf\{s\geq 0: \wh L_s > t\}$, $ A_t = \arg X_{\sigma_{t}}$ and
$ \wh A_t = \Im \wh X_{\wh\sigma_{t}}$
for $t\geq 0$. Then $\wh A$ and $A$ are indistinguishable processes.

It follows from the uniqueness of the deterministic Skorohod problem
that  the process $\bar X_t := \wh X_t - i\int_0^t \tan \theta(\wh
X_s) d \wh L_s$ is a normally reflected Brownian motion in the left half-plane $D_-$. Hence,
if we let $ C_t = \Im \bar X(\wh\sigma_{t})$ for $t\geq 0$,
then $C_t$ is a Cauchy process with the initial value $C_0=\Im \bar X_{\bar S}=\arg X_S$, where
$\bar S := \inf\{t> 0: \bar X_t \in \prt D_-\}$ and
$S := \inf\{t> 0: X_t \in \prt D_*\}$.
Clearly, $C_0$ depends only on the initial starting point of $X$ and is independent of the reflection angle $\theta$.
We have
\begin{align}\label{n13.1}
A_t = \wh A_t = C_t +
\int_0^{\wh \sigma_t} \tan(  \theta(\wh X_s)) d \wh L_s
= C_t +
\int_0^{\sigma_t} \tan(  \theta( X_s)) d  L_s.
\end{align}
For $u\geq 0$, define
\begin{equation}\label{e:4.8}
 T_u = \inf\{t >u : X_t \in \prt D_*\},
\end{equation}
 with the convention
$\inf \emptyset := \infty$.
We obtain from \eqref{n13.1},
\begin{align*}
\arg X_t &= A_{L_t} + \arg X_t - \arg X_{T_t}
 \\
&= C_{L_t} +
\int_0^{t} \tan(  \theta( X_s)) d  L_s
+ \arg X_t - \arg X_{T_t}   \\
&= C_t + (C_{L_t} - C_t) +
\int_0^{t} \tan(  \theta( X_s)) d  L_s
+ (\arg X_t - \arg X_{T_t})   .
\end{align*}
Hence,
\begin{eqnarray}\label{n12.9}
\frac1t  \arg X_t - \mu_0
  &= & \frac1t C_t + \frac1t (C_{L_t} - C_t)  +
\left(\frac1t \int_0^{t} \tan(  \theta( X_s)) d L_s -\mu_0 \right)  \nonumber \\
& &  + \frac1t (\arg X_t - \arg X_{T_t})   .
\end{eqnarray}

By \eqref{e:4.1}, $\E_Q [L_1] =\int_{\partial D_*} (h(x)/2) dx = 1$.
It follows from these remarks, \eqref{ju21.22},
\eqref{e:4.1} with $g(x)=\tan \theta (x)$, and
the limit-quotient theorem for additive functionals (see, e.g., \cite[Thm. X 3.12]{RevYor}) that for every $z\in \ol D_*$, $\P_z$-a.s.,
\begin{align}\label{n12.6}
 \lim_{t\to \infty} \frac1t
\int_0^{t} \tan(  \theta( X_s)) d  L_s
&= \E_Q \left[\int_0^1 \tan \theta (X_s) dL_s\right]
 = \int_{\prt D_*} (1/2) \tan \theta(x) h(x) dx
=\mu_0,    \\
& \lim_{t\to \infty} \frac1t
L_t = 1.    \label{n12.7}
\end{align}

Fix an arbitrarily small $\eps >0$ and any $z \in \ol D_*$ and let
\begin{align}\label{feb19.10}
p_1(t) = \P_{z}(|\arg X_t - \arg X_{T_t}| > \eps t).
\end{align}
We will argue that $p_1(t)$ is small for large $t$.
Let
 $T'_u = \sup\{t \in [0, u]: X_t \in \prt D_*\}$ with the convention $\sup \emptyset = 0$.
By the Markov property applied at time $t$ and the symmetry of Brownian motion,
\begin{align*}
\P_{z}\left(\arg X_{T'_t} - \arg X_t >0\right)
=\P_{z}\left(\arg X_{T'_t} - \arg X_t <0\right)
=1/2.
\end{align*}
This and the Markov property applied at time $t$ imply that
\begin{align}\label{n12.1}
\P_{z}\left(|\arg X_{T'_t} - \arg X_{T_t}| > \eps t\right)
\geq p_1(t)/2.
\end{align}
For a fixed $u>0$, the Cauchy process $C$ is continuous at time $u$, a.s.
Let $\delta >0$ be so small that
\begin{align*}
\P\left(\sup_{1-\delta \leq u,v \leq 1+\delta}
|C_u - C_v| \geq \eps/2\right) < \eps.
\end{align*}
Then, by scaling, for any $t >0$,
\begin{align}\label{n12.3}
\P\left(\sup_{(1-\delta)t \leq u,v \leq (1+\delta)t}
|C_u - C_v| \geq \eps t/2\right) < \eps.
\end{align}
By \eqref{n12.7}, we can find $t_1$ so large that for
$t\geq t_1$,
\begin{align}\label{n12.2}
\P_{z}(L_t \in((1-\delta)t, (1+\delta)t) ) \geq 1-\eps.
\end{align}
The jumps of $A$ have the same size as those of $C$ and occur at the same time because the last integral in \eqref{n13.1} is a continuous function of $t$.
If the events in \eqref{n12.1} and \eqref{n12.2} occur then
$C$ has a jump of size greater than $\eps t$ at a time $s = L_t \in ((1-\delta)t, (1+\delta)t)$.
The probability of this event is greater than $p_1(t)/2 -\eps$, by \eqref{n12.1} and \eqref{n12.2}. However, by \eqref{n12.3}, this probability is less than $\eps$. Hence, $p_1(t)/2 < 2\eps$ and, therefore, $p_1(t) < 4 \eps$ for sufficiently large $t$.
This and \eqref{feb19.10} imply that for sufficiently large $t$,
\begin{align}\label{n12.4}
\P_{z}\left(\frac1t |\arg X_t - \arg X_{T_t}| > \eps \right)
< 4\eps.
\end{align}

Another consequence of \eqref{n12.3} and \eqref{n12.2} is that $|C_{L_t} - C_t| \leq \eps t$ with probability greater than $1-2\eps$ for large $t$.
Thus, for sufficiently large $t$,
\begin{align}\label{n12.5}
\P_{z}\left(\frac1t |C_{L_t} - C_t| > \eps \right)
< 2\eps.
\end{align}
It follows from \eqref{n12.6} that
for sufficiently large $t$,
\begin{align}\label{n12.8}
\P_{z}\left(\left|\frac1t \int_0^{t} \tan(  \theta( X_s)) d L_s -\mu_0\right| > \eps \right)
< \eps.
\end{align}
 Note that $(C_t-C_0)/t$ has the Cauchy distribution.
Since $\eps>0$ is arbitrarily small, the last observation, \eqref{n12.9}, \eqref{n12.4}, \eqref{n12.5} and \eqref{n12.8}  imply that the distributions of
$\frac1t \arg X_t - \mu_0$ converge to the Cauchy distribution as $t\to \infty$.

\smallskip

(iii)
Consider a modification of the process $C$ which is left continuous with right limits.
For $t\geq 0$, let
$$
\Lambda_t  = \sum_{s\leq t} (C_{t+} - C_t) \bone_{\{ |C_{t+} - C_t| > 2 \pi\}}, \qquad
C^*_t   = C_t -C_0- \Lambda_t=C_t-\arg X_S -\Lambda_t.
$$
The process $C^*$ is a Cauchy process with jumps larger than $2\pi$ removed and starts from $C^*_0=0$.
It is elementary to see that $C^*_t $ is a zero mean martingale and a L\'evy process. Hence, the law of large numbers holds for $C^*$, that is, a.s.,
\begin{align}\label{n13.5}
\lim_{t\to \infty} C^*_t/t = 0.
\end{align}
Note that the jumps removed from $C$ correspond to increments of $\arg X$ in the sum on the right hand side of
 \eqref{n11.1}.
Thus \begin{align}\label{d11.10}
\arg^* X_{\sigma(t)} = C^*_t +
\int_0^{\sigma(t)} \tan(  \theta( X_s)) d L_s + \arg X_S ,
\end{align}
and
\begin{align}\label{d11.8}
\frac1t \arg^* X_{\sigma(t)} = \frac1t C^*_t +
\frac1t \int_0^{\sigma(t)} \tan(  \theta( X_s)) d L_s
+\frac1t  \arg X_S .
\end{align}
It follows from \eqref{n12.7} that, a.s., \begin{align}\label{feb19.15}
\lim_{t\to\infty} \sigma(t)/t = 1.
\end{align}
This, \eqref{n12.6},  \eqref{n13.5} and \eqref{d11.8} imply that for every $z\in \ol D_*$, $\P_z$-a.s.,
\begin{align}\label{feb19.20}
\lim_{t\to\infty}
\frac1t \arg^* X_{\sigma(t)} = \mu_0.
\end{align}

We claim  that
\begin{align}\label{feb19.14}
\lim_{t\to\infty} T_t/t = 1, \quad \hbox{a.s.}
\end{align}
First note that since $\int_0^\infty 1_{\{X_s\in D_*\}} dL_s =0$, we have by
\eqref{n12.7} that $\lim_{t\to \infty} T_t=\infty$.
For every $\eps >0$, $L_t-\eps <L_{T_t} \leq L_t$  so
$$ \frac{L_t-\eps }{t} <\frac{L_{T_t}}{T_t} \frac{T_t}{t} \leq \frac{L_t}{t}.
$$
This together with \eqref{n12.7} establishes the claim \eqref{feb19.14}.
Combining \eqref{feb19.15}, \eqref{feb19.20} and \eqref{feb19.14} yields
\begin{align}\label{n13.6}
\lim_{t\to\infty}
\frac1t \arg^* X_{T_t} = \mu_0.
\end{align}

Next we will argue that \eqref{n13.6} implies that $\lim_{t\to\infty} \frac1t
\arg^* X_{t} = \mu_0$ by using excursion theory. Recall that $H^x$ denotes the excursion law for Brownian motion in $D_*$. We will estimate the $H^x$-measure of the family $F_a$ of excursions with the property that $|\arg \exc(0) - \arg \exc(\zeta-)| \leq 2 \pi$ and $\sup_{t\in[0,\zeta(\exc))}|\arg \exc(0) - \arg \exc(t)| \geq a$, for $a \geq 4\pi$.
Note that this quantity does not depend on $x$.
Let $\wh H^x$ be the excursion law for Brownian motion in $D_- =\{z\in \C: {\rm Re} z<0\}$  starting from $x\in \prt D_-$.
Excursion laws are conformally invariant, up to a multiplicative constant (see
\cite[Prop. 10.1]{BurBook}). 
The exponential function $f(z) = \exp(z)$ maps $D_-$ onto $D_* \setminus \{0\}$
and is locally conformal, up to the boundary.  
Hence, for some constant $c_4$, $H^x(F_a) = c_4 \wh H^y(\wh F_a)$, where $\wh F_a$ is the family of excursions with the property that $|\Im \exc(0) - \Im \exc(\zeta-)| \leq 2 \pi$ and $\sup_{t\in[0,\zeta(\exc))}|\Im \exc(0) - \Im \exc(t)| \geq a$.
If we normalize all excursion laws as in \eqref{eq:M5.2} then it is easy to check that $c_4 = 1$ (although our argument does not depend on the value of this constant).
Thus, the equality $H^x(F_a) =  \wh H^y(\wh F_a)$ holds for all $x\in \prt D_*$ and $y\in \prt D_-$.
By \cite[Thm. 5.1(v)]{BurBook}, for some $c_5<\infty$,
\begin{align}\label{n13.8}
\wh H^x\left(\sup_{t\in[0,\zeta(\exc))}|\Im \exc(0) - \Im \exc(t)| \geq a\right) \leq c_5 /a.
\end{align}
It is easy to see that if Brownian motion starts in $D_-$ from a point $z$ with $|\Im z| > a$ with $a\geq 4\pi$ then the chance that it will exit $D_-$ through the line segment on the imaginary axis between $-2\pi i$ and $2\pi i$ is bounded above by $c_6/a$. This, \eqref{n13.8} and the strong Markov property of $\wh H^x$ applied at the time $\inf \{ t\in[0,\zeta(\exc)):  |\Im \exc(0) - \Im \exc(t)| \geq a\}$
imply that
\begin{align}\label{n13.9}
H^x(F_a) = \wh H^x\left(\wh F_a\right) \leq c_5c_6 /a^2 = c_7 /a^2.
\end{align}
Fix some $\alpha >0$.
By the exit system formula \eqref{exitsyst},
the probability that there exists an excursion $\exc_t$ of $X$ such that $L_t>s$ and $\exc_t$ belongs to $F_{\alpha L_t}$ is equal to
\begin{align*}
\int_{s}^\infty  H^{X_{\sigma(u)}}(F_{\alpha u}) du
\leq \int_{s}^\infty c_7 / (\alpha u)^2 du
= c_8 /(\alpha^2 s).
\end{align*}
This quantity goes to 0 as $s\to \infty$, so for every fixed $\alpha >0$, with probability 1, there is $s_\alpha=s_\alpha (\omega) < \infty$ such that there are no
excursions $\exc_t\in F_{\alpha L_t}$ with $L_t>s_\alpha$.

Fix an arbitrarily small $\alpha>0$ and
suppose that $t_1$ is so large that $\frac1t \arg^* X_{T_t} \leq \mu_0+\alpha$ and $L_t/t \leq 2$
for all $t>t_1$.
 If $\frac1u \arg^* X_{u} \geq \mu_0+5\alpha$
for some $u> t_1$ then $ |\arg^* X_{u} - \arg^* X_{T_u}| \geq 4\alpha u \geq 2 \alpha L_u$. This means that an excursion starting at $T_u$ belongs to $F_{2\alpha L_u}= F_{2\alpha L_{T_u}}$. Since there are no such excursions beyond some $s_{2 \alpha}$, it follows that
$\limsup_{t\to \infty}\frac1t \arg^* X_{T_t} \leq \mu_0+5\alpha$, a.s. This holds for all rational $\alpha>0$ simultaneously, a.s., so
$\limsup_{t\to \infty}\frac1t \arg^* X_{T_t} \leq \mu_0$, a.s.
The matching lower bound for $\liminf$ can be proved analogously. We conclude that for every $z\in D_*$, $\P_z$-a.s.,
$\lim_{t\to\infty} \frac1t \arg^* X_{T_t} = \mu_0$.

The proof of \eqref{j16.1} will be combined with the proof of Theorem \ref{j15.7} (iv) given below.

(iv)
Since $h$ is $C^2$ on $\ol D_*$, it follows from \eqref{ju20.3} that $\theta(z)$ is $C^2$
on $\ol D_*$, and hence $\theta(x)$ is $C^2$ on $\prt D_*$. Moreover $H=h+i\wt h$
is $C^{2-\eps}$, by Corollary II.3.3 in \cite{GarMar}. By assumption, $h$ is
positive and continuous on $\prt D_*$. Thus $H(\ol D_*)$ is a compact
subset of $\{\Re z >0\}$
and so by \eqref{ju20.3}, $\sup_x |\theta(x)| < \pi/2$. We can now apply parts (i) and (iii) of the theorem to see that part (iv) holds.
\end{proof}

\medskip

\begin{proof}[\bf Proof of Theorem \ref{j15.5}]
Fix a Borel measurable function $\theta: \prt D_* \to [-\pi/2, \pi/2]$. First we need to prove that there exists a sequence of $C^2$ functions $\theta_k: \prt D_* \to (-\pi/2, \pi/2)$
which converges to $\theta$ in weak-* topology.  For this,
we extend $\theta$ harmonically to $\ol D_*$ and then we
let $\theta_k(e^{it}) = \theta(e^{it} (1-1/k))$.  Then $\theta_k$'s converge
to $\theta$ in weak-* topology. See \cite[page 33]{Hoff}.

\medskip
(i) This was essentially proved in \cite[Thm. 1.1]{BurMar}. That theorem was concerned with ORBM in a half-plane while the present result is set in a disc. Theorem \ref{j15.5}(i) can be proved just like \cite[Thm. 1.1]{BurMar} by repeating the arguments given in \cite{BurMar} with some minor adjustments. We omit the proof to save space. The Markov property of $X$ follows from that of $X^k$ and the convergence of finite dimensional distributions.
Since for each $k$, the subprocess of $X^k$ before hitting $\prt D_*$ is Brownian motion in $D_*$ before hitting $\prt D_*$,
the same claim applies to the subprocess of $X$ before hitting $\prt D_*$.

The transition probabilities are the same for each process $|X^k|$ so the process $|X|$ has the same transition probabilities. It follows that $X$ is conservative.

\medskip
(ii) This claim was shown in the proof of \cite[Thm. 1.1]{BurMar} although
it was not a part of the statement of that theorem. See Step 4 on page 214
of \cite{BurMar}.

\medskip
(iii)  Suppose that $(h_k, \mu_k) \lra \theta_k$ and $X^k$ solves the SDE
\eqref{n19.1} except that the initial distribution for $X^k$ is the
stationary distribution $h_k(z)dz$. According to Remark \ref{ju20.5}, the measures $h_k(z)dz$ converge to $h(z)dz$. It is easy to see that part (i) of this theorem implies that $X^k$'s converge weakly to a process $X$ satisfying the SDE \eqref{j13.1}, with the initial distribution $h(z)dz$.
For every $t\geq 0$ and $k\geq 1$, the distribution of $X^k_t$ is $h_k(z) dz$. Hence, for every $t\geq 0$, the distribution of $X_t$ is $h(z) dz$. This shows that $h$ is a stationary distribution for $X$ satisfying \eqref{j13.1}.

We next show
uniqueness of the stationary distribution.
As observed in  \eqref{e:2.31},
for every reflection angle field $\theta$, the radial part $|X|$ of $X$ is a two-dimensional Bessel process confined to $[0,1]$ by reflection at $1$.
This easily implies that for any initial distribution of $X$, the distribution of $X_1$ has a strictly positive density inside $\ball(0,1/2)$.
If there were more than one invariant measure, at
least two of them (say, $Q_1$ and $Q_2$) would be mutually
singular by Birkhoff's ergodic theorem \cite{Sinai}.
We have just shown that the Lebesgue measure restricted to
$\ball(0,1/2)$ (let us call it $Q_3$) is absolutely continuous with respect to
the distribution of $X_1$, so that in particular, $Q_3 \ll Q_1$ and $Q_3 \ll Q_2$. Since $Q_1 \perp Q_2$ by assumption,
there exists a set $A \subset \ball(0,1/2)$ such that $Q_1(A) = 0$ and $Q_2(\ball(0,1/2) \setminus A) =
0$. Therefore, one must have $Q_3(A) = Q_3(\ball(0,1/2) \setminus A) = 0$ which
contradicts the fact that $Q_3(\ball(0,1/2))\ne 0$.

\medskip
(iv) The first claim
follows easily from the definitions.
The second claim follows from the first claim  and part (i) of the theorem.

\medskip
(v) Since $\theta_k$ are smooth, \eqref{n12.9} holds for $X^k$'s, that is,
\begin{eqnarray}\label{n22.1}
\frac1t \arg X^k_t - \mu_k
  &=& \frac1t C^k_t + \frac1t (C^k_{L^k_t} - C^k_t)  +
\left(\frac1t \int_0^{t} \tan(  \theta_k( X^k_s)) d L^k_s -\mu_k \right) \nonumber \\
&&   + \frac1t (\arg X^k_t - \arg X^k_{T^k_t}) ,
 \end{eqnarray}
where the symbols with the superscript or subscript $k$ denote objects analogous to those in \eqref{n12.9}.
Since $X^k$'s converge to $X$ weakly, we can assume that all these processes are
constructed on a single probability space and $X^k_t \to X_t$, a.s., for every fixed $t$, as $k\to \infty$. In view of \eqref{n22.1}, we can write
\begin{align}\label{feb22.1}
\frac1t &\arg X_t - \mu_0
=
\left(\frac1t \arg X_t -
\frac1t \arg X^k_t \right) - (\mu_0 -\mu_k)
   + \frac1t  (C^k_t -C^k_0)
 + \frac1t (C^k_{L^k_t} - C^k_t)  \nonumber \\
& \quad +
\left(\frac1t \int_0^{t} \tan(  \theta_k( X^k_s)) d L^k_s -\mu_k \right)
+ \frac1t (\arg X^k_t - \arg X^k_{T^k_t})
+\frac 1t \arg X^k_{S^k} ,
\end{align}
where $S^k=\inf\{t>0: X^k_t\in \partial D_*\}$.
The distribution of $\frac1t (C^k_t-C^k_0)$ is Cauchy for every $k$ and $t$ so it
suffices to show that all other terms on the right hand side of \eqref{feb22.1}
are small for large $t$ and $k$.

Fix an arbitrarily small $\eps>0$.
Note that \eqref{n12.4} and \eqref{n12.5} do not depend on $\theta$ so we can apply them for all $\theta_k$.
Hence, we can find $t_1$ so large that for $t\geq t_1$,
\begin{align*}
\P\left ( \left|\frac1t \left(C^k_{L^k_t} - C^k_t\right)
+ \frac1t \left(\arg X^k_t - \arg X^k_{T^k_t}\right)
\right| \geq \eps \right) < \eps.
\end{align*}

We will assume without loss of generality that $X^k_0 = z \ne 0$, a.s., for all $k$.
(The case $z=0$ can be dealt with by applying the Markov property at time $t=1$.)
 Then $\arg X^k_{S^k}$ has the same distribution for each $k\geq 1$ and
so by taking $t_1$ larger if needed,
\begin{align*}
\P\left ( \left|\frac 1t \arg X^k_{S^k}
\right| \geq \eps \right) < \eps, \qquad \hbox{for all } k\geq 1 \hbox{ and } t\geq t_1.
\end{align*}
Recall that $X^k_t \to X_t$, a.s. By Remark \ref{ju20.5} (vi), $\mu_k \to
\mu_0$. Thus,
for a fixed $t$, we can make $k$ so large that
\begin{align*}
\P\left ( \left|\frac1t \arg X_t -
\frac1t \arg X^k_t \right| +|\mu_0 -\mu_k|
 \geq \eps \right) < \eps.
\end{align*}

Hence, it will suffice to prove that for a fixed $\eps>0$, some $t_1$ and $k_1$, all  $t\geq t_1$, $k\geq k_1$ and  $z_k\in \ol D_*$,
\begin{align}\label{n27.1}
\P_{z_k}\left(\left|\frac1t \int_0^{t} \tan(  \theta_k( X^k_s)) d L^k_s -\mu_k\right| > \eps \right)
< \eps.
\end{align}

If we let $Q_k(dx) =h_k(x)dx$ then by \eqref{n12.6},
\begin{align}\label{n27.2}
\E_{Q_k}\left[ \frac1t \int_0^{t} \tan(  \theta_k( X^k_s)) d L^k_s \right] =\mu_k.
\end{align}
Hence, to finish the proof of part (iv) of the theorem, it will suffice to show that
\begin{align}\label{d8.1}
\Var\left(
\frac1t \int_0^{t} \tan(  \theta_k( X^k_s)) d L^k_s\right)
\leq c_1 /t.
\end{align}
We will split the rest of the proof of this part of the theorem into steps.

\textit{Step 1}. We will recall some results from \cite[Lemmas 2.2-2.3]{BurMar} but we will change the notation.

We will say that $D\subset \C$ is a monotone domain if $D$ is open, connected and for every $z\in D$ and $b>0$ we have $z+ib \in D$.

Let $\HH =\{z\in \C: \Im z >0\}$ be the upper half-plane. Suppose that
$\theta: \prt \HH  \to [-\pi/2, \pi/2]$
is a Borel measurable function
and suppose $\theta$ is not equal almost everywhere either to $\pi/2$ or to $-\pi/2$. Then there exists a univalent
analytic mapping $g$ of $\HH $ onto a monotone domain $D=g(\HH)$ such that
for almost all $x\in \prt \HH $,
~$g(x)$ and $g'(x)$ exist,
$g'(x)\ne 0$ and $\arg g'(x) = \theta (x)$.
We choose $g$ so that
$\lim_{|z|\to\infty} |g(z)| =\infty$.
We construct $g$ as follows.
Let
$\theta: \HH  \to \R$
be the bounded harmonic extension of our original function
$\theta: \prt \HH  \to [-\pi/2, \pi/2]$
and let $\wt \theta$ be the harmonic conjugate of $\theta$
such that $\wt\theta(i)=0$.
Define
$g: \HH  \to\C$ by setting $g(i)=i$ and
$$g'(z) = \exp (i(\theta(z) + i \wt\theta(z))) .$$
Then $g$ is one-to-one on $\HH$ because $\Re g'(z) > 0$. (See
\cite{BurMar}).
By abuse of notation, we will use the same symbol $\theta$ to denote
real functions on both $\prt D_*$ and $\prt \HH $. Specifically, for
$z\in \prt \HH $, we let $\theta(z) = \theta(\exp(iz))$, where
$\theta(\exp(iz))$ refers to the function $\theta\in \calT$ introduced in the assumptions of Theorem \ref{j15.5}.
Hence, in this proof, $\theta: \prt \HH  \to \R$ is a periodic function
with period $2\pi$. It follows that $g$ is also periodic with period
$2\pi$, up to an additive constant. That is,
$g(z+2\pi)=g(z)+d$ for all $z\in \HH $,  where $d=g(i+2\pi)-g(i)$ . The
constant $d$ is non-zero since $\Re g'> 0$.

Suppose that
$\theta_k:\prt D_* \to (-\pi/2, \pi/2)$
are $C^2 $-functions which converge weak-*to $\theta$
as $k\to\infty$. Let $g_k$ and $D_k:=g_k (\HH)$ correspond to
$\theta_k$ in the same way as $g$ and $D=g(\HH)$ correspond to $\theta$.
Note that
$g_k(z+2\pi)=g_k(z)+d_k$ for some constant $d_k$.
Moreover if $\eps >0$, then $g_k(z+i \eps)$ converges to $g(z+i\eps)$ uniformly
in $z\in \R$ and $d_k \to d$.
Indeed, by weak-* convergence of $\theta_k\in \calT$,
we conclude uniform convergence of $\theta_k(z)+i \wt \theta(z)$
on the compact set
$\{z: |z|=e^{-\eps}\}$, and hence $g_k'$ converges uniformly on
$I=\{z: 0 \le \Re z \le 2\pi, \Im z=\eps\}$. Integration then shows
that $d_k\to d$
and hence $g_k$ converges uniformly to $g$ on $\R + i\eps$.
Let $f(z) = \exp (i g^{-1}(z ))$, for $z\in D$ and $f_k(z) = \exp (i
g_k^{-1}(z ))$ for $z\in D_k$. Then $f$ and $ f_k$ are locally conformal maps of $D$
and $ D_k$ onto $D_*\setminus \{0\}$  which are periodic with periods $d$ and $ d_k$, respectively.

The monotone domains $D_k$
converge to $D$ in the following sense.
\begin{enumerate}[\rm (a)]
\item
If $B$ is open and such that $B\cap \prt D\ne\emptyset$, there is a
$k_0=k_0(B)$ such that $B\cap \prt D_k \ne \emptyset $ for all
$k\geq k_0$.

\item  If $B$ is connected and open, with $B\cap D\ne\emptyset$ and
$B\subset D_k$ for infinitely many $k$, then $B\subset D$.

\item  If $K$ is compact and $K\subset D$ then $K\subset D_k$ for
all $k\geq k_0 = k_0(K)$.
\end{enumerate}

\medskip

We invoke conformal invariance of  ORBM  as in
\eqref{n29.1}-\eqref{n29.3}.
For $x\in \prt D_k$ such that $f_k(x) = z\in \prt D_*$, let
$\wh\bv_k(x) = i \sec \theta_k (z)$. In other words, $\wh\bv_k$ is the conformal (inverse) image of
the vector of reflection $\bv_{\theta_k}$.
Suppose that $\wh B$ is a two-dimensional Brownian motion and consider the Skorokhod SDE
\begin{align}\label{n29.4}
\wh X^k_t = \wh x_k + \wh B_t + \int_0^t \wh\bv_k(\wh X^k_s) d\wh L^k_s,
\end{align}
where $\wh L^k$ is the local time of $\wh X^k$ on $\prt D_k$.
The process $\wh X^k$ is reflected Brownian motion in $D_k$ with the oblique angle of reflection $\theta_k$.
If
$c_k(t)= \int_0^t |f'_k(\wh X^k_s)|^2 ds$
then the process $X^k_t = f_k(\wh X^k_{c_k(t)})$
is reflected Brownian motion in $D_*$ with the oblique angle of reflection $\theta_k$.

Let $K_k = f_k^{-1} \left(\prt \ball(0,1/2)\right)$. Note that
$K_k $ is the image under the map $g_k$  of  the horizontal line $\{z:\Im z=\ln 2\}$, and
so $K_k$ is an analytic curve.
Let $a_k = \Re d_k=\Re(g_k(2\pi) - g_k(0))$ and for $z\in \prt D_k$, let $R_k(z) = \{x\in \prt D_k: |\Re x - \Re z| \geq a_k\}$.
Let $\wh T^k(A) = \inf\{t\geq 0: \wh X^k_t \in A\}$.
We will show that for every $\theta$ there exists $p_1>0$ such that for every approximating sequence $\{\theta_k\}$ there exists $k_1$ such that for any $k\geq k_1$ and $z_k \in \prt D_k$,
\begin{align}\label{n29.5}
\P_{z_k}( \wh T^k(K_k) < \wh T^k(R_k) ) \geq p_1.
\end{align}

Let $[x,z]$ denote the line segment between $x,z\in\C$.
For every $\theta$ there exist $a,b\in(0,\infty)$ such that for every approximating sequence $\{\theta_k\}$ there exists $k_1$ such that for any $k\geq k_1$ and $z\in \prt D_k$ we have $a_k \geq a$ and $K_k \cap [z, z+ ib] \ne \emptyset$.

With probability greater than $p_2>0$, Brownian motion starting from 0 will hit the line $\{z: \Im z = 2 b\}$ before hitting the lines $\{z: |\Re z| = a/2\}$, and then it will cross the imaginary axis before hitting any of the lines
$\{z: |\Re z| = a\}$ or $\{z: \Im z = b\}$. Since $\int_0^t \wh\bv_k(\wh X^k_s) d\wh L^k_s$ is a purely imaginary number with non-negative imaginary part, this implies that with probability greater than $p_2$, the process $\wh X^k$ starting from $z_k\in \prt D_k$
will hit the line $\{z: \Im z - \Im z_k = 2 b\}$ before hitting the lines $\{z: |\Re z - \Re z_k| = a/2\}$, and then it will cross the line $\{z: \Re z = \Re z_k\}$ before hitting any of the lines $\{z: |\Re z -\Re z_k| = a\}$ or $\{z: \Im z - \Im z_k = b\}$. If the trajectory of
$\wh X^k$ follows a path described above then, in view of the definitions of $a$ and $b$, it will cross $K_k$ before hitting $R_k$. We conclude that \eqref{n29.5} holds with  $p_1=p_2>0$.

Let
\begin{align*}
T^k(A) & = \inf\{t\geq 0:  X^k_t \in A\},\\
\Tkb & = T^k(\ball(0,1/2)),\\
T^k_* & = \inf\{t\geq 0: X^k_t \in \prt D_*, |\arg X^k_t - \arg X^k_0| \geq 2\pi\}.
\end{align*} 
By the conformal invariance of ORBM,
\eqref{n29.5} implies that
\begin{align}\label{n29.6}
\P_{z_k}(  \Tkb <  T^k_* ) \geq p_1,
\qquad \text{  for all  }k \text {  and  } z_k \in \prt D_*.
\end{align}

\emph{Step 2}. We will estimate the variance of
$ \int_0^{1} \tan(  \theta_k( X^k_s)) d L^k_s$.

Let $S_1^k = T^k(\prt D_* \cup \prt \ball(0,1/2))$. The probability that Brownian motion will make a loop in the annulus $ D_* \setminus  \ball(0,1/2)$ (that is, $\arg X^k$ will increase or decrease by $2\pi$) before hitting the boundary of the annulus is less than $p_3<1$. This implies that, for any $z\in \ol D_*$,
\begin{align}\label{d1.1}
\P_z\left(|\arg X^k_{S_1^k} - \arg X^k_0| \leq 2 \pi
\right) \geq 1-p_3 .
\end{align}
This and an easy inductive argument based on the strong Markov property applied at the times when consecutive loops are completed shows that there exists $n$ so large that for any $z\in \ol D_*$,
\begin{align}\label{d1.7}
\P_z\left(|\arg X^k_{S_1^k} - \arg X^k_0| \geq n 2 \pi
\right) \leq p_1/4 ,
\end{align}
where $p_1$ is as in \eqref{n29.6}.
Fix such an $n$ and let
\begin{align*}
S^k_2 &= \inf\{t\geq 0:
|\arg X^k_t - \arg X^k_0| \geq (n+1)2 \pi\},\\
S^k_{3} &= \inf\{t\geq S_2^k:
|\arg X^k_t - \arg X^k_{S^2_k}| \geq n 2 \pi\},\\
S^k_{4} &= \inf\{t\geq 0:
|\arg X^k_t - \arg X^k_0| \geq (2n+1) 2 \pi\},\\
S^k_{5,j} &= \inf\{t\geq 0:
|\arg X^k_t - \arg X^k_0| \geq j(2n+2) 2 \pi\}.
\end{align*}
By \eqref{n29.6}, we have for $z\in \prt D_*$,
\begin{align*}
\P_{z}(  \Tkb \leq  T^k_* \land S_2^k )
+
\P_{z}( S_2^k \leq \Tkb \leq  T^k_* )
\geq p_1.
\end{align*}
It follows that either
\begin{align}\label{d1.10}
\P_{z}(  \Tkb \leq  T^k_* \land S_2^k )
\geq p_1/2,
\end{align}
or
\begin{align}\label{d1.11}
\P_{z}( S_2^k \leq \Tkb \leq  T^k_* )
\geq p_1/2.
\end{align}
Suppose that the last estimate holds.
By \eqref{d1.7} and the strong Markov property applied at $S^k_2$,
\begin{align*}
\P_{z}( S_2^k \leq S^k_3 \leq \Tkb \leq  T^k_* )
\leq p_1/4,
\end{align*}
so, in view of \eqref{d1.11},
\begin{align*}
\P_{z}( S_2^k \leq \Tkb \leq S^k_3 \land  T^k_* )
\geq p_1/4.
\end{align*}
It follows from this and \eqref{d1.10} that
\begin{align*}
\P_{z}(  \Tkb \leq S^k_3  )
\geq p_1/4,
\end{align*}
and, therefore, for $z\in \prt D_*$,
\begin{align*}
\P_{z}(  \Tkb \leq S^k_4  )
\geq p_1/4.
\end{align*}
We combine this with \eqref{d1.1} using the strong Markov property at $S^k_1$ to see that for $z\in \ol D_*$,
\begin{align*}
\P_{z}(  \Tkb \leq S^k_{5,1}  )
\geq (1-p_3) p_1/4 =: p_4 >0.
\end{align*}
Applying the strong Markov property repeatedly at $S^k_{5,j} $'s, we see that for $z\in \ol D_*$ and $j\geq 1$,
\begin{align*}
\P_z\left(\Tkb \geq S^k_{5,j}
\right) \leq (1-p_4)^j .
\end{align*}
In other words,
\begin{align}\label{d1.13}
\P_z\left(
|\arg X^k_{\Tkb} - \arg X^k_0|
 \geq j(2n+2) 2\pi
\right) \leq (1-p_4)^j .
\end{align}

Let $X^0$ be the ORBM corresponding to $\theta\equiv 0$.
It is easy to see that
\begin{align*}
\arg X^k_t - \arg X^k_0
-  \int_0^{t} \tan(  \theta_k( X^k_s)) d L^k_s
\end{align*}
has the same distribution as
$\arg X^0_t - \arg X^0_0$. The estimate \eqref{d1.13} applies to $X^0$; to prove that, one can apply the same argument as the one for $X^k$'s or a direct elementary proof. Since
\begin{align*}
 & \int_0^{t} \tan(  \theta_k( X^k_s)) d L^k_s\\
&= \big( \arg X^k_t - \arg X^k_0\big)
- \left( \arg X^k_t - \arg X^k_0 -  \int_0^{t} \tan(  \theta_k( X^k_s)) d L^k_s\right),
\end{align*}
and \eqref{d1.13} applies to both quantities within parentheses, we obtain for $z\in \ol D_*$ and $j\geq 1$,
\begin{align*}
&\P_z\left(
\left| \int_0^{\Tkb} \tan(  \theta_k( X^k_s)) d L^k_s \right|\geq 2j(2n+2) 2\pi
\right) \\
&\leq
\P_z\left(
\left| \arg X^k_{\Tkb} - \arg X^k_0 \right|\geq j(2n+2) 2\pi
\right) \\
&\qquad + \P_z\left(
\left| \arg X^k_{\Tkb} - \arg X^k_0
-  \int_0^{\Tkb} \tan(  \theta_k( X^k_s)) d L^k_s \right|\geq j(2n+2) 2\pi
\right) \\
&\leq 2(1-p_4)^{j} .
\end{align*}
This implies that for some $c_2<\infty$ and all $z\in \ol D_*$ and all $k$,
\begin{align}\label{d1.30}
\E_z
\left[ \left| \int_0^{\Tkb} \tan(  \theta_k( X^k_s)) d L^k_s \right|^3\right]\leq c_2 .
\end{align}
Let $V_0=U_1 = 0$, and for $m\geq 1$,
\begin{align*}
V_m &= \inf\{t\geq U_m: X^k_t \in \ball(0,1/2)\},\\
U_{m+1} &=  \inf\{t\geq V_{m}: X^k_t \notin \ball(0,3/4)\}.
\end{align*}
Since $\P(U_{m+1} - V_{m} > 1\mid \calF_{V_m}) > p_5>0$, we have
\begin{align}\label{d1.20}
\P(U_m \leq 1 ) \leq c_3 (1-p_5)^m.
\end{align}
Note that the local time $L^k$ does not increase on intervals $[V_m, U_{m+1}]$. Hence
\begin{align}\label{d1.21}
 \int_0^{1} \tan(  \theta_k( X^k_s)) d L^k_s
=
\sum_{m=1}^\infty
 \int_{U_m\land 1}^{V_m\land 1} \tan(  \theta_k( X^k_s)) d L^k_s ,
\end{align}
and, therefore,
\begin{align*}
&\left| \int_0^{1} \tan(  \theta_k( X^k_s)) d L^k_s \right|^3
=
\left|\sum_{m=1}^\infty
 \int_{U_m\land 1}^{V_m\land 1} \tan(  \theta_k( X^k_s)) d L^k_s \right|^3\\
&\leq
3\sum_{m=1}^\infty \sum _{i\leq m} \sum _{j\leq m}
\left|\bone_{\{U_m < 1\}}
 \int_{U_m\land 1}^{V_m\land 1} \tan(  \theta_k( X^k_s)) d L^k_s \right| \cdot
\left| \bone_{\{U_i < 1\}}
 \int_{U_i\land 1}^{V_i\land 1} \tan(  \theta_k( X^k_s)) d L^k_s \right|\\
&\qquad\times \left| \bone_{\{U_j < 1\}}
 \int_{U_j\land 1}^{V_j\land 1} \tan(  \theta_k( X^k_s)) d L^k_s \right|\\
&\leq
3 \sum_{m=1}^\infty \sum _{i\leq m} \sum _{j\leq m}
\Bigg[\bone_{\{U_m < 1\}}
\left| \int_{U_m\land 1}^{V_m\land 1} \tan(  \theta_k( X^k_s)) d L^k_s \right|^3 +
\bone_{\{U_i < 1\}} \left|
 \int_{U_i\land 1}^{V_i\land 1} \tan(  \theta_k( X^k_s)) d L^k_s \right|^3 \\
&\qquad +
\bone_{\{U_j < 1\}} \left|
 \int_{U_j\land 1}^{V_j\land 1} \tan(  \theta_k( X^k_s)) d L^k_s \right|^3\Bigg].
\end{align*}
This, \eqref{d1.30} and \eqref{d1.20} imply that
for some $c_4<\infty$, all $z\in \ol D_*$ and all $k$,
\begin{align}\label{d1.22}
\E_z\left[\left| \int_0^{1} \tan(  \theta_k( X^k_s)) d L^k_s \right|^3\right]
\leq
3 \sum_{m=1}^\infty \sum _{i\leq m} \sum _{j\leq m}
3 c_3(1-p_5)^m c_2 < c_4.
\end{align}

\medskip
\emph{Step 3}.
For a fixed $z\in \ol D_*$ and all $k$, the processes $\{|X^k_t|, t\geq 0\}$ have the same distribution, that of 2-dimensional Bessel process on $[0,1]$, reflected at 1.
Hence, $\P_z\left( |X^k_{1/2}| < 1/4\right) > p_6$, where $p_6$ does not depend on $z\in \ol D_*$ and $k$. This and the Markov property at time $1/2$ can be used to show that the density of the distribution of $X^k_1$ under $\P_z$ is greater than $c_5>0$ on $\ball(0,1/2)$, where $c_5$
does not depend on $z\in \ol D_*$ and $k$.

Let $\P^k_x$ denote the
distribution of the process $X^k$ starting from $x$.
Consider $z\in \ol D_*$. We will construct a process $X^k$ with distribution $\P^k_z$ in a special way.
First we will construct i.i.d. random vectors $A_1, A_2, A_3, \dots$ The distribution of each $A_j$ is
partly continuous, with
density $c_5$ in $\ball(0,1/2)$.
With probability $1- c_5 \pi /4$,
$A_j$ takes value $\bDelta$ (the cemetery state).
Let $q^k_1$ be the density of $X^k_1$  under the
distributions $\P^k_z$. Let $ B_1$ be a random
vector with density $ q^k_1(x) - c_5\bone_{\ball(0,1/2)}(x)$
on $D_*$. With
probability $c_5\pi/4$, $B_1$ takes
value $\bDelta$. We construct $ B_1$ so that it is equal to $\bDelta$ if and only if
$A_1 \ne \bDelta$.
Moreover, we make the conditional distribution of $B_1$ given $\{B_1 \ne \bDelta\}$ independent of $A_j$'s.

In the following construction, the expression ``Markov
bridge'' will refer to the Markov bridge corresponding to
$\P^k$. If $A_1 \in \ball(0,1/2)$ then we let $\{ X^k_t,
0\leq t\leq 1\}$ be the Markov bridge between the points in
time-space $(0,z)$ and $(1,A_1)$. If $A_1
=\bDelta$ then we let $\{ X^k_t, 0\leq t\leq 1\}$ be the
Markov bridge between the points  $(0,z)$ and $(1, B_1)$,
otherwise independent of $A_j$'s and $B_1$.

We continue by induction. Suppose that $\{ X^k_t, 0\leq
t\leq n\}$ has been
defined. Let
$ q^k_{n+1}( X^k_n, x)$ be the density of $X^k_1$ under the distribution
$\P^k_{ X^k_n}$. Let $B_{n+1}$
be a random vector with density $ q^k_{n+1}( X^k_n, x) -
c_5\bone_{\ball(0,1/2)}(x)$ on $D_*$. With probability
$c_5\pi/4$, this random vector takes value $\bDelta$.
We construct $ B_{n+1}$ so that
it is equal to $\bDelta$ if and only if $A_{n+1} \ne
\bDelta$.
Moreover, we make the conditional distribution of $B_{n+1}$ given $\{B_{n+1} \ne \bDelta\}$ independent of $A_j$'s and $\{ X^k_t, 0\leq
t\leq n\}$, except that it has the density $ q^k_{n+1}( X^k_n, x) -
c_5\bone_{\ball(0,1/2)}(x)$ on $D_*$.

If $A_{n+1} \in \ball(0,1/2)$ then we let $\{X^k_t, n\leq t\leq n+1\}$ be the Markov bridge between the
points in time-space $(n, X^k_n)$ and $(n+1,A_{n+1})$,
otherwise independent of
$A_j$'s and $\{ X^k_t, 0\leq
t\leq n\}$. If $A_{n+1}
=\bDelta$ then we let $\{ X^k_t, n\leq t\leq n+1\}$ be the
Markov bridge between $(n, X^k_n)$ and $(n+1, B_{n+1})$, otherwise independent of
$A_j$'s and $\{ X^k_t, 0\leq
t\leq n\}$.
It is easy to check that this inductive construction yields a process $\{ X^k_t, t\geq 0\}$ with distribution $\P^k_z$.

Let $\Gamma^k_n = \int_n^{n+1} \tan(\theta_k( X^k_s)) d
 L^k_s$. Let $\bfA_n = \bigcup_{j=1}^n \{A_n \ne
\bDelta\}$ and note that $\P(\bfA_n^c) = (1- c_5 \pi /4)^n =:
c_6^n$, where $c_6 < 1$. If $\bfA_n$ holds then the trajectory of
$\{ X^k_t, n\leq t\leq n+1\}$ does not depend on $
X^k_1$. Hence, $\Cov(\Gamma^k_1, \Gamma^k_n \bone_{\bfA_n}  )
=0$. We have,
\begin{align*}
\Cov(\Gamma^k_1, \Gamma^k_n)
&= \Cov(\Gamma^k_1, \Gamma^k_n \bone_{\bfA_n} + \Gamma^k_n \bone_{\bfA^c_n} )
=
\Cov(\Gamma^k_1, \Gamma^k_n \bone_{\bfA_n}  )
+
\Cov(\Gamma^k_1,  \Gamma^k_n \bone_{\bfA^c_n} )\\
&= \Cov(\Gamma^k_1,  \Gamma^k_n \bone_{\bfA^c_n} )
=
\E(\Gamma^k_1 \Gamma^k_n \bone_{\bfA^c_n} ) - \E \Gamma^k_1 \E (\Gamma^k_n \bone_{\bfA^c_n} ),
\end{align*}
so, in view of \eqref{d1.22}, for some $c_{10}<1$,
\begin{align*}
|\Cov(\Gamma^k_1, \Gamma^k_n)|
&\leq
(\E|\Gamma^k_1|^3)^{1/3} (\E|\Gamma^k_n|^3)^{1/3}
(\E\bone_{\bfA^c_n}^3)^{1/3}
+ \E \Gamma^k_1 (\E (\Gamma^k_n)^2)^{1/2}
(\E \bone_{\bfA^c_n}^2)^{1/2}\\
&\leq c_7 c_4^{n/3} + c_8 c_4^{n/2}
\leq c_9 c_{10}^n.
\end{align*}
It is easy to see that the estimate applies also to $n=1$ (possibly with new values of the constants).
This implies that
\begin{align*}
&\Var\left( \int_0^{n} \tan(  \theta_k( X^k_s)) d L^k_s\right)
 \\
&= \sum _{i=1}^n\sum _{j=1}^n
\Cov \left( \int_{i-1}^{i} \tan(  \theta_k( X^k_s)) d L^k_s,
\int_{j-1}^{j} \tan(  \theta_k( X^k_s)) d L^k_s\right)\\
&\leq \sum _{i=1}^n\sum _{j=1}^n c_9 c_{10}^{|i-j|}
\leq c_{11} n.
\end{align*}
It is elementary to check that the estimate also applies with non-integer upper limit, that is, for any $t>1$,
\begin{align*}
&\Var\left( \int_0^{t} \tan(  \theta_k( X^k_s)) d L^k_s\right)
\leq c_{11} t.
\end{align*}
This completes the proof of \eqref{d8.1} and hence the proof of part (v) of the theorem.

\medskip

(vi) The claim follows from the ergodic theorem if we show that under the stationary distribution $h(x)dx$,
\begin{align}
\E_h \left[\barg^* X_1 \right]=\mu_0 . \label{d11.3}
\end{align}
Recall that
\begin{align}\label{d11.4}
 \lim_{k\to \infty} \mu_{k} = \mu_0 .
\end{align}
Theorem \ref{f3.2} (iii) implies that
\begin{align}
\E_{h_k} \left[\arg^* X^k_1 \right] =\mu_{k} . \label{d11.5}
\end{align}
It follows easily from definitions of $\arg^*$ and $\barg^*$, and Theorem \ref{j15.5}(iv) that $\arg^* X^k_1 \to \barg^* X _1 $ in distribution. Hence, in view of \eqref{d11.4}-\eqref{d11.5}, the proof of \eqref{d11.3} will be complete if we prove that the family $\{\arg^* X^k_1\}_{k\geq 1}$ is uniformly integrable.

The following formula can be derived in the same way as
 \eqref{n12.9} has
been derived,
\begin{align}\label{d11.9}
\arg^* X^k_{1} =  C^*_{L^k_1} +
\int_0^{1} \tan(\theta_k( X^k_s)) d L^k_s
 + \left(\arg X^k_1 - \arg X^k_{T^k_1}\right) .
\end{align}
Here $C^*$ is a Cauchy process with jumps larger than $2\pi$ removed.

 Recall that
 $S^k=\inf\{t>0: X^k_t\in \partial D_*\}$ and $T^k_1=\inf\{t>1: X^k_t\in \partial D_*\}$. So by the Markov property of
 $X^k$, under the stationary measure $h_k(x)dx$, $\{X^k_s-X^k_1; 1\leq s\leq T^k_1\}$ has the same distribution
 as that of $\{X^k_s - X^k_0; 0 \leq s\leq S^k\}$.
It follows from the paragraph following \eqref{e:2.31} that
under the stationary measure  $h_k(x)dx$,
$Y^k=|X^k|$ is a stationary 2-dimensional Bessel process in $(0, 1]$ reflected at 1.
Let   $\sigma_k (a, b]=\int_{\{a<|x|\leq b\}} h_k(x) dx$. Then $\sigma_k (dr)$
 is the stationary probability distribution of $Y^k$ so it is independent of $k$.
This and the rotational invariance of Brownian motion imply that the distribution of
$\arg X^k_0 - \arg X^k_{S^k}$ does not depend on $k$.
By an earlier remark, the distribution of $\arg X^k_1 - \arg X^k_{T^k_1}$ is the same so it
does not depend on $k$ either. Hence, the family
$\left\{\arg X^k_1 - \arg X^k_{T^k_1}\right\}$, $k\geq 1$, is uniformly integrable.
The distribution of $L^k_1$ does not depend on $k$ so the same applies to $C^*_{L^k_1}$. Random variables
$\int_0^{1} \tan(\theta_k( X^k_s)) d L^k_s $
are uniformly integrable by \eqref{d1.22}.
All these remarks taken together with \eqref{d11.9} show that the family $\{\arg^* X^k_1\}_{k\geq 1}$ is uniformly integrable. This completes the proof of part (vi) of the theorem.

\medskip

(vii)  An explicit integral test was given in \cite{BM}:
The ORBM in $D_+=\{z: \Im z > 0\}$ with angle of
reflection $\theta$ hits $0$ with positive probability if and only if
\begin{align}\label{sept6.3}
\int_0^1 \frac{1}{y} \Re \exp \left(i\left(\theta(iy)+i \wt
\theta(iy)\right)\right) dy < \infty,
\end{align}
where $\theta(z)$ is the bounded harmonic extension of $\theta$ to
$D_+$ and $\wt \theta$ is the harmonic conjugate of $\theta$ vanishing
at $z=i$. In \cite{BM} there was the added assumption that $\theta \in
C^{1+\eps}$, for some $\eps>0$, except possibly at $0$.
As noted in \cite{BurMar}, the same result holds if we only assume
$\theta$ is measurable and $|\theta| \le \pi/2$.
One way to transfer this result to $\theta\in \calT$ is to set
$\theta_1(t)=\theta(e^{it})$, for $t\in \R$, and
$\theta_1(z)=\theta(e^{iz})$ for $z \in D_+$ as before.
Then
\begin{align*}
\int_0^1 \frac{1}{y} \Re \exp \left(i\left(\theta_1(iy)+i \wt
\theta_1(iy)\right)\right) dy
=\int_0^1\frac{1}{y} \Re
\exp\left( i(\theta+i\wt\theta)(e^{-y})\right)  dy.
\end{align*}
Setting $r=e^{-y}$, we have $y=\ln 1/r \sim 1-r$ on $[e^{-1},1]$ and so ORBM
hits $1$ with positive probability in $D_*$ if and only if
the left-hand side of \eqref{sept6.1} is finite for $x=1$.
By \eqref{ju20.4},
\begin{align*}
1/(h+i \wt h -i\mu_0/\pi)= \pi \cos \theta(0)e^{i(\theta+i \wt
\theta)}
\end{align*}
and by taking real parts, the two integrals in \eqref{sept6.1} are
equal.

Suppose that for some $z_0 \in D_*$ and $x\in \partial D_*$, $\P_{z_0} (x\in \Gamma^\theta_X)>0$. 
A simple coupling argument shows that for some $r>0$ and $p>0$,
$\P_{z} (x\in \Gamma^\theta_{X[0, 1]})\geq p$
for all $z\in \ball(z_0,r) \subset D_*$.  
Since for every $k\geq 1$, $X_t$ returns to $\ball(z_0,r)$ for some $t\geq k$ with probability one,  
we have by this ``renewal property" that 
$\P_z (x\in \Gamma^\theta_X)=1$ for all $z\in \ball (z_0, r)$.

\medskip
(viii) Let $\rho$ denote the Prokhorov distance between
probability measures (\cite[App. III]{Bill}). For any
stochastic processes $V $ and $Z$, we will write $\rho(V,Z)$ to
denote the distance between their distributions relative to
$M_1$ distance between trajectories. For every $k$, one can find a sequence
$(\theta^n_k)_{n\geq 1}$ of $C^2$ functions with values in $(-\pi/2,\pi/2)$ which converges to $\bar\theta_k$ as
$n\to \infty$ in weak-* topology.
Recall that $\bar X_k$ are defined relative to $\bar \theta_k$ in the same way that $X$ is defined relative to $\theta$. Processes $X^k$ are defined by \eqref{n19.1} relative to $\theta_k$.
By part
(i) of the theorem, one can find a sequence $\theta^{n_k}_k: \prt D_* \to (-\pi/2, \pi/2)$ with the
following properties. Let $X^{k,n_k}$ be the solution to
\eqref{n19.1} relative to $\theta^{n_k}_k$. Then $\rho (X^{k,n_k} , \bar
X^k) < 1/k$. Moreover, we can choose $n_k$'s so large that the sequence $(\theta^{n_k}_k)_{k\geq 1}$ converges to $\theta$ in
weak-* topology.
Since the sequence $(\theta_1, \theta_1^{n_1}, \theta_2, \theta_2^{n_2}, \theta_3, \theta_3^{n_3}, \dots)$
converges to $\theta$, the sequence of processes $X^1, X^{1, n_1}, X^2, X^{2, n_2}, X^3, X^{3, n_3}, \dots$
converges in distribution to a process $X'$, by part (i) of the theorem. We must have $X = X'$ in distribution, because
$(\theta_k)_{k\geq 1}$ is a subsequence of
$(\theta_1, \theta_1^{n_1}, \theta_2, \theta_2^{n_2}, \theta_3, \theta_3^{n_3}, \dots)$.
We see that $\rho (X^{k,n_k} , X)
\to 0$. Since $\rho (X^{k,n_k} , \bar X^k) < 1/k$, we
obtain $\rho (\bar X^k , X) \to 0$ as $k\to\infty$.
\end{proof}

\medskip

\begin{proof}[\bf Proof of Proposition \ref{sept11.1}]
The integral in \eqref{sept6.1} is equal to
\begin{align*}
\int_0^1\frac{1}{1-r} \int_{\prt D_*} \frac{1-r^2}{|z-rx|^2}
d\sigma(z) dr
=\int_{\prt D_*} \int_0^1\frac{1+r}{|1-rx\ol z|^2}dr d\sigma(z).
\end{align*}
Let $w=x \bar z$. Then $|w|=1$ and
\begin{align*}
\frac1{|1-rw|^2}=\frac{1}{(1-rw)(1-r\ol w)}=
\frac{1}{\ol w -w} \left( \frac{-w}{1-rw}-\frac{\ol w}{1-r\ol w} \right) .
\end{align*}
So
\begin{align*}
\int_0^1\frac{1}{|1-rw|^2} dr = \frac1{\ol w-w} \ln \frac{1-w}{1-\ol w}
=  \frac{ \arg{(1-w)}}{|1-w| \sin  \arg{(1-w)} }\sim \frac{1}{|1-w|}.
\end{align*}
Thus the integral in \eqref{sept6.1} is finite if and only if \eqref{sept11.2} holds.
\end{proof}

\medskip

\begin{proof}[\bf Proof of Proposition \ref{P:3.10}]
An application of the Riemann mapping theorem shows that it suffices to prove the proposition for $D=D_*$.

(i)
The expected occupation measure for an excursion law $H^x$ is a constant multiple of $K_x(\,\cdot\,)$ by \eqref{jan17.1}.
According to the definition, the ERBM is a ``mixture'' of excursion laws. This easily implies that
the stationary distribution for $X$ has the density that is proportional to
$  \int _{\prt D_*}   K_x(y) \nu(dx) $.

\smallskip
(iii)
The function $h$ has a representation
$h(y) = \int _{\prt D_*} K_x(y) \nu(dx) $.
If one constructs an ERBM corresponding to $\nu$ then
the stationary measure of this process is $h$ by part (i) of the proposition.
\end{proof}

\medskip

\begin{proof}[\bf Proof of Theorem \ref{j18.1}]
(i)
 Since $\lim_{k\to \infty} {\rm dist}(x_k, \partial D_*) = 0$,
every subsequence of $x_k$ contains a further subsequence that converges to some point in $ \partial D_*$.
We will assume that the whole sequence $x_k$ converges to a point $x_\infty \in\partial D_*$. We will show that the limit distribution of $X^k$ does not depend on $x_\infty$. Hence, the result holds for every sequence satisfying $\lim_{k\to \infty} {\rm dist}(x_k, \partial D_*) = 0$.

  As was noted in the paragraph following \eqref{e:2.31},
  for any $r_0\in[0,1]$, $t\geq 0$ and $\theta_1, \theta_2 \in \calT$, if $X^k$ is an ORBM in $D_*$ with the angle of reflection $\theta_k$ and $|X^k_0|=r_0$ for $k=1,2$, then the distributions of $|X^1_t|$ and $|X^2_t|$ are identical. Suppose that $X$ is an ORBM. Then $\P(|X_t| \in [1-\eps, 1]) \leq c\eps$ for some $c$ and all $\eps \geq 0$.
Fix an arbitrary $\eps\in (0,1)$.
Let
$\calE^*_\eps=\{\exc^1, \exc^2,\dots\}$ be the set of
all excursions of $X$ from $\prt D_*$ which enter the ball $\ball(0, 1-\eps)$, ordered according to their starting times. Let $S_n= S_n(\eps) = \inf\{t\geq 0: \exc^n_t \in \ball(0, 1-\eps)\}$. It follows from the rotation invariance of Brownian motion that the distribution of $\{\exp(-i\arg \exc^n_{S_n})  \exc^n_t, t\geq S_n\}$ (the excursion rotated about 0 so that $\exc^n_{S_n}$ is mapped to $1-\eps\in \R$) does not depend on $n$, $\theta$ or the value taken by $S_n$.

Since the process
$\{  \exc^n_t, t\geq S_n\}$
is Brownian motion killed upon hitting of $\prt D_*$, its trajectory has modulus of continuity $c(\omega)\sqrt{2 r |\log r|} $, where $c(\omega)$ is finite for almost all $\omega$ (see \cite[Thm. 2.9.25]{KS}).
If we time-reverse $\exc^n$ and rotate it so that it starts from 0, then it will have the distribution $H^0$ conditioned by $\{\exists t>0: \exc_t \in \ball(0, 1-\eps)\}$.
Hence, the claim about the modulus continuity can be extended as follows.
The modulus of continuity of $\{ \exc^n_t, t\in (0,\zeta)\}$ is $c_1(\omega)\sqrt{2 r |\log r|} $, where $c_1(\omega)$ is finite for almost all $\omega$.
This easily implies that for any sequence of random variables $V_k$ which converges to 0 in distribution, processes
$\{\exp(-i V_k )  \exc^n_t, t\geq 0\}$ converge to $\{  \exc^n_t, t\geq 0\}$
in distribution in the Skorokhod topology as $k\to \infty$. Note that no assumptions on the joint distribution of $V_k$ and $\{ \exc^n_t, t\geq 0\}$
are needed.

Recall that $h_k(0) = 1/\pi$ for any $(h_k,\mu_{0,k})\in \calH$. Hence $\int_{\prt D_*} h_k(x) dx = 2$ and, therefore, $\nu(\prt D_*) =2$. It follows that $\nu/2$ is a probability distribution on $\prt D_*$.

Let $\calE^k_\eps$ be defined relative to $X^k$ in the same way as $\calE^*_\eps$ has been defined relative to a generic $X$. We will suppress both $\eps$ and $k$ in the notation for excursions, i.e., we will write $\calE^k_\eps=\{\exc^1, \exc^2,\dots\}$.
In view of the opening remarks of this proof, it is routine to show that in order to prove  part (i) of the theorem, it is sufficient to show that for any fixed $\eps \in (0,1)$ and $n$, the joint distribution of $(\exc^1_0, \exc^2_0, \dots, \exc^n_0)$ converges to that of a sequence of $n$ i.i.d. random variables with distribution $\nu/2$, as $k\to \infty$.

Let $\sigma^k_t = \inf\{s \geq 0: L^k_s > t\}$ and $A^k_t = \arg X^k_{\sigma^k_t}$, with the convention that $\arg X^k_{\sigma^k_t} \in [0,2\pi)$.
By abuse of notation, we define $\theta_k$ for real $x$ by $\theta_k(x) = \theta_k(e^{i x})$.
 Let $B$ be Brownian motion in $\C$ starting at the origin and $S^k=\inf\{t>0: x_k+B_t\in \partial D_*\}$.
Let $\wh A_0^k=\arg (x_k+B_{S^k})$. Since $x_k\to x_\infty\in \partial D_*$,
$a_0:= \lim_{k\to \infty} \wh A_0^k=\arg x_\infty$ a.s.
Let $C_t$ be a Cauchy process with $C_0=0$ that is independent of $B$,
  and let $\wh A^k_t$ be the solution to the SDE
\begin{align}\label{j4.10}
\wh A^k_t = \wh A^k_0 +  C_t + \int_0^t \tan \theta_k(\wh A^k_s) ds.
\end{align}
Clearly,  $\wh A^k_0$ has the same distribution as $ A^k_0 $.
Let $\bar A^k_t\in [0,2\pi)$ be the unique number such that
$\bar A^k_t = \wh A^k_t + j 2\pi$ for some integer $j$. Then, by the conformal invariance of ORBM's presented in \eqref{n29.1}-\eqref{n29.3},
the distribution of $\{\bar A^k_t, t\geq 0\}$ is the same as
that of $\{ A^k_t, t\geq 0\}$.

To incorporate our assumptions on $h_k$ and $1/h_k$,
we first note that by \eqref{aug27.7} and \eqref{aug27.3}
\begin{align}\label{sept14.1}
\tan\theta_k(z)=\frac{\mu_k(z)}{\pi h_k(z)}=\frac{\mu_{0,k}}{\pi h_k(z)} -
\frac{\wt h_k (z)}{h_k(z)},
\end{align}
for $z\in D_*$.
If $f$ is Lipschitz with constant $\lambda$, then its modulus of continuity
satisfies $\omega_f(\delta)\le \lambda \delta$. By \cite[Thm. III.1.3]{Gar}
the modulus of continuity of $\wt f$ satisfies
\begin{align*}
\omega_{\wt f} (\delta) \le C\lambda \delta ( 1+\log{\pi/\delta}),
\end{align*}
where $C$ is a constant not depending on $f$ or $\delta$. So by
assumption (c),
$\wt
h_k$ are Dini continuous on $\ol D_*$,
with constants depending only on $\lambda$, not $k$. We also conclude that each $\theta_k$ and $\wt
\theta_k$ are Dini continuous on $D_*$, and therefore on $\ol D_*$, by  \eqref{ju20.3}.
In particular,
\eqref{sept14.1} holds for $x\in \prt D_*$.

By a change of variables,
\begin{align}\label{s19.1}
\wh A_{t/\mu_{0,k}}^k  &= \wh A_0^k  + C_{t/\mu_{0,k}} +
\int_0^t \tan \theta_k(\wh A_{r/\mu_{0,k}}^k)\frac{dr}{\mu_{0,k}}\\
&= \wh A_0^k  + C_{t/\mu_{0,k}} -\frac{1}{\mu_{0,k}}
\int_0^t \frac{\wt h_k (A_{r/\mu_{0,k}}^k)}{h_k (A_{r/\mu_{0,k}}^k)} dr
+\frac{1}{\pi}\int_0^t \frac{1}{h_k(A_{r/\mu_{0,k}}^k)} dr.\nonumber
\end{align}

By assumption (d), $h_k(z)=\int K_z(x) h_k(x) |dx|$ converges
to $\int K_z \nu(dx) :\equiv h(z)$, where $K_z$ is the Poisson kernel for
$z\in D_*$. Since each $h_k$ is Lipschitz with
constant $\lambda$ on $\prt D_*$ and therefore on $\ol D_*$,
we have that $|h(z)-h(w)| \le \lambda |z-w|$ for $z,w \in D_*$. Thus $h$
extends to be Lipschitz with constant $\lambda$ on $\ol D_*$ and so
$\nu(dx)=h(x)|dx|$.

Recall from Remark \ref{jan15.1} (iii) that the assumption (c) implies that all functions $1/h_k$ are Lipschitz with the same constant. Without loss of generality, we will assume that the Lipshitz constant for $1/h_k$ is $\lambda$. It follows that $1/h$ is Lipshitz with constant $\lambda$.

Recall that $a_0:= \lim_{k\to \infty} \wh A_0^k=\arg x_\infty$.
By abuse of notation, let $h(x) = h(e^{ix})$ for real $x$ and let $a_t$ be the solution to
\begin{align}\label{s19.2}
a_t = a_0 + \int_0^t \frac 1 {\pi h(a_s)} ds.
\end{align}
Let $t_1$ be such that $a_{t_1} = a_0 + 2\pi$.
Since
\begin{align*}
\frac{\prt}{\prt t} \nu([a_0, a_t])/2
=(1/2)\frac{\prt}{\prt t} \int_{a_0}^{a_t} h(b) db
= (1/2) \frac{h(a_t)}{ \pi h(a_t)} = \frac1{2\pi}
\end{align*}
and $\nu([a_0, a_{t_1}])/2 = \nu([a_0, a_0 + 2\pi])/2 =1$,
we must have $t_1 = 2\pi$.
Hence, for $0\leq s \leq t \leq 2\pi$,
\begin{align}\label{j5.5}
\nu([a_s, a_t])/2 = \frac{t-s}{2\pi}.
\end{align}

It follows from \eqref{s19.1}-\eqref{s19.2} that
\begin{align*}
\wh A_{t/\mu_{0,k}}^k -a_t  = F_t^k
+\frac{1}{\pi}\int_0^t \left(\frac{1}{h(A_{r/\mu_{0,k}}^k)}
-\frac{1}{h(a_r)}\right) dr,
\end{align*}
where
\begin{align}\label{sept15.1}
F_t^k= \wh A_0^k - a_0  + C_{t/\mu_{0,k}} -\frac{1}{\mu_{0,k}}
\int_0^t \frac{\wt h_k (A_{r/\mu_{0,k}}^k)}{h_k (A_{r/\mu_{0,k}}^k)} dr
+\frac{1}{\pi}\int_0^t
\left(\frac{1}{h_k(A_{r/\mu_{0,k}}^k)}-\frac{1}{h(A_{r/\mu_{0,k}}^k)}\right)dr.
\end{align}
Since $1/h$ is Lipschitz with constant $\lambda$,
\begin{align*}
|\wh A_{t/\mu_{0,k}}^k -a_t| \le \sup_{0\le s \le 2\pi} |F_s^k|
+\frac{\lambda}{\pi}\int_0^t \left|\wh A_{r/\mu_{0,k}}^k -a_r\right|dr,
\end{align*}
for $0\le t \le 2\pi$.
By Gr\"onwall's inequality (see \cite{Bell}),
\begin{align}\label{sept14.3}
\left|\wh A_{t/\mu_{0,k}} -a_t\right| \le \left(\sup_{0\le s \le 2\pi}
|F_s^k| \right)  ~e^{\lambda t /\pi}.
\end{align}

We claim that
\begin{align}\label{sept15.5}
\lim_{k\to \infty}\sup_{0\le s \le 2\pi} |F^k_s|=0,
\end{align}
in probability.
By the definition of $a_0$, $\lim_k \wh A_0^k -a_0=0$.
By assumption (a), $\theta_k(0)=\int
\theta_k(x)|dx|/2\pi$ converges to $\pi/2$. But then $\mu_{0,k}=\tan\theta_k(0)$
converges to $+\infty$. Thus \break \hfill
$\sup_{0 \le t \le 2\pi} C_{t/\mu_{0,k}} =0$,
a.s.  Since $\wt h_k$ and $1/h_k$ are Dini continuous on $D_*$ with
constant depending only on $\lambda$,
and $h_k(0)=1/\pi$ and $\wt h_k(0)=0$, we have that $\wt h_k/h_k$ is
bounded  on $\prt D_*$ by a constant independent of $k$. Thus the first integral
in \eqref{sept15.1} also tends to $0$.

 If
$\beta_n(f)$ denotes the $n^{th}$ Cesaro mean of $f$ on $\prt D_*$ then
for continuous $f$, $\beta_n(f)$ converges uniformly on $\prt D_*$ to $f$,
with the difference
$\|\beta_n(f)-f\|_{\infty}$ depending only on the modulus of
continuity of $f$ and $n$.  See
\cite[page 18]{Hoff}.
Since $1/h_k$ and $1/h$ are Lipschitz with constant $\lambda$, given
$\eps>0$ we can choose $n$ so that
\begin{align}\label{sept17.1}
\|1/h_k -\beta_n(1/h_k)\|_\infty < \eps ~\text{ and } ~
\|1/h -\beta_n(1/h)\|_\infty < \eps.
\end{align}
By assumption (d) $h_k$ converges to $h$, uniformly on compact subsets
of $D_*$, and since $1/h_k$ is uniformly bounded, $1/h_k$ converges to
$1/h$ uniformly on compact subsets of $D_*$. Since $1/h_k$ are
uniformly bounded, this also implies $1/h_k$ converges to $1/h$ weak-* and
therefore for $k$ sufficiently large, and $n$ fixed,
\begin{align}\label{sept17.2}
\| \beta_n(1/h_k) -\beta_n(1/h)\|_{\infty} < \eps.
\end{align}
By \eqref{sept17.1}, \eqref{sept17.2}, and the triangle inequality,
$1/h_k$ converges uniformly to $1/h$.
We conclude that
the second integral in \eqref{sept15.1} tends to $0$ as well, proving
the claim.

We will need a generalization of the above results \eqref{sept14.3} and
\eqref{sept15.5}.
Let $D_u =\{z\in \C: \Im z >0\}$ be the upper half-plane.
Let $H^x$ be the excursion law for Brownian motion in $D_*$, for excursions starting from $x\in \prt D_*$ and
let $\wh H^x$ be the excursion law for Brownian motion in $D_u$, for excursions starting from $x\in \prt D_u$.
The measure $\wh H^0(\exc(\zeta-)\in dx)$ is the distribution of the end point of the excursion under $\wh H^0$. It is also the L\'evy measure for the Cauchy process. Let
\begin{align*}
\mu_\eps(dx) = \wh H^0\left(\sup_{t\in[0,\zeta)} \Im \exc_t < |\log(1-\eps)|, \exc(\zeta-)\in dx\right).
\end{align*}
The measure $\mu_\eps$ is the L\'evy measure for a pure jump process, say $C^\eps_t$, similar to the Cauchy process, except that it has fewer big jumps.
We can choose a right continuous version of
$C^\eps$, and so $ \sup_{0\leq s \leq t} |C^\eps_s| \to 0$, a.s., as $t\to 0$.
We let $\wh A^{k,\eps}_t$ be the solution to the equation analogous to \eqref{j4.10},
\begin{align}\label{j5.2}
\wh A^{k,\eps}_t = \wh A^{k,\eps}_0 +  C^\eps_t + \int_0^t \tan \theta_k(\wh A^{k,\eps}_s) ds.
\end{align}
An argument analogous to that showing \eqref{sept14.3} and
\eqref{sept15.5} proves that for every fixed $\eps>0$,
\begin{align}\label{j5.4}
\sup_{0\leq s \leq 2\pi} \left|\wh A^{k,\eps}_{s/\mu_{0,k}} - a_{ s}\right|\to 0,
\end{align}
in probability, as $k\to \infty$.

Recall the definition of $\calE^k_\eps=\{\exc^1, \exc^2,\dots\}$ from the beginning of the proof. We claim that for any fixed $\eps \in (0,1)$ and $n$, the joint distribution of $(\exc^1_0, \exc^2_0, \dots, \exc^n_0)$ converges to that of a sequence of $n$ i.i.d. random variables with distribution $\nu/2$, as $k\to \infty$.

We will present a special construction of $(\exc^1_0, \exc^2_0, \dots, \exc^n_0)$.
The heuristic meaning of the construction is the following. Excursions that reach $\ball(0, 1-\eps)$ occur as a Poisson process with constant intensity on the local time scale. If we have already observed $\exc^1, \exc^2, \dots, \exc^m$, the next excursion will occur after an exponential waiting time on the local time scale, where the local time has the same distribution as the process $\wh A^{k,\eps}_t $. This process, suitably rescaled, behaves like the function $a_t$ according to \eqref{j5.4}. By \eqref{j5.5}, a point on the boundary chosen in a uniform manner on the $a_t$ scale has the distribution $\nu/2$. We will also need a fact that, on small time intervals, exponential density is almost constant. The process
$\wh A^{k,\eps}_t $ represents rapid rotation along the unit circle and the exponential clock will chose a point on the circle according to the distribution very close to $\nu/2$, because the almost constant exponential density (on small intervals) is transformed into the density of $\nu/2$ by the function $a_t$.

Suppose that excursions $\exc^1, \exc^2, \dots, \exc^m$ have been already generated, for some $m\geq 0$. If $m\geq 1$, let $T_m$ be the time when $\exc^m$ ended. If $m=0$ then we take $T_0 $ to be the first hitting time of $\prt D_*$ by $X^k$.
Unless stated otherwise, every new random object introduced below will be assumed to be independent from all random objects constructed so far.

By conformal invariance of excursion laws,
\begin{align*}
H^x \left( \exists t\in[0,\zeta): \exc_t \in \ball(0,1-\eps)\right)
=
\wh H^0\left(\exists t\in[0,\zeta): \Im \exc_t \geq |\log(1-\eps)|\right),
\end{align*}
and the last quantity is equal to $1/|\log(1-\eps)|$ (see \cite{BurBook} for the justification of both claims).

Consider an exponential random variable $\alpha$ with density $f_\alpha(t)$ and expected value $|\log(1-\eps)|$, independent of objects constructed so far.
For every $\delta>0$ there exists $c_3>0$
so small that for any interval $[t, t+ c_3]$ and any $s_1,s_2\in[t, t+ c_3]$, we have $f_\alpha(s_1) /f_\alpha(s_2) \in (1-\delta, 1+\delta)$.
We generate an integer-valued random variable $N$, such that $\P(N= j) =
\P(\alpha\in[j 2\pi/\mu_{0,k},(j+1)2\pi/\mu_{0,k}])$ for $j\geq 0$.
We consider a solution to \eqref{j5.2} with $\wh A^{k,\eps}_0 = \arg
X^k_{T_m} + N 2\pi/\mu_{0,k}$. We generate a random variable $\alpha'$ with the same distribution
as $\alpha$ conditioned to be in $[N, N+1)$.
Note that we can take $\delta>0$ so small and then let $k$ be so large
that, in view of \eqref{j5.5} and \eqref{j5.4}, the distribution of $\exp(i\wh A^{k,\eps}_{\alpha'-N})$ is arbitrarily close to $\nu/2$.

We generate an excursion $\bar \exc^{m+1}$ with the
(probability) distribution $H^0(\,\cdot\, \mid \exists
t\in[0,\zeta): \exc_t \in \ball(0,1-\eps))$. We let $\wh
\exc^{m+1}_t =  \exp(i\wh A^{k,\eps}_{\alpha'-N})\,\bar
\exc^{m+1}_t$.

In view of the preceding remarks, the distribution of $\wh \exc^{m+1}_0$ is arbitrarily close to
$\nu/2$, conditional on the trajectories of $\exc^1, \dots, \exc^m$, if $k$ is arbitrarily large. According to our construction, the joint distribution of  $(\exc^1, \dots, \exc^m, \wh \exc^{m+1})$ is the same as that of $(\exc^1, \dots, \exc^m, \exc^{m+1})$. We conclude that
for any fixed $\eps \in (0,1)$ and $n$, the joint distribution of $(\exc^1_0, \exc^2_0, \dots, \exc^n_0)$ converges to that of a sequence of $n$ i.i.d. random variables with distribution $\nu/2$, as $k\to \infty$.
This completes the proof of part (i) of the theorem.

\smallskip

(ii)
We will generalize Example \ref{m2.1}.
Suppose that $h$ is positive on $\ol D_*$, harmonic in $D_*$ and
Lipschitz on $\ol D_*$. Then $1/h$ is Lipschitz on $\ol D_*$. Set
$h_k(z)=h((1-1/k)z)$ and suppose $\mu_{0,k} \to \infty$.
Then $(h_k,\mu_{0,k}) \leftrightarrow \theta_k
\in \calT$ as in Theorem \ref{aug27.0}, satisfy the assumptions of part (i)
and the conclusions of that part of the theorem with the given $h$.
\end{proof}

\medskip

\begin{proof}[\bf Proof of Theorem \ref{j15.7}]
(i) Suppose that $X_0$ has the stationary distribution with density $h$. Then
for every $t>0$,
\[
\begin{array}{rclcl}
\displaystyle \E \left[c(t)\right] &= & \displaystyle \E\left[\int_0^t |f'(X_s)|^2 ds\right]
 &= & \displaystyle \int_0^t \E \left[|f'(X_s)|^2 \right]ds
 = \int_0^t \int_{D_*} |f'(x)|^2 h(x) dx ds\\
& =  & \displaystyle  t \int_{D_*} |f'(x)|^2 h(x) dx
& = &\displaystyle  t \int_D  \bar h(x) dx <\infty.
\end{array}
\]
It follows that under the stationary distribution, $\zeta = \infty$, a.s. This implies that $\zeta = \infty$, $\P_x$-a.s., for almost all $x\in D_*$.

Consider an $x\in D_*$ and $r>0$ so small that $\ol{\ball(x,r)}\subset D_*$. The exit distributions from
$\ball(x,r)$ are mutually absolutely continuous for any two points $y,z \in \ball(x,r)$. Let $T$ be the exit time from $\ball(x,r)$. It is easy to see that $c(T) < \infty$, $\P_y$-a.s., for every $y \in \ball(x,r)$. Since
$\zeta = \infty$, $\P_y$-a.s., for at least one $y \in \ball(x,r)$, it follows that this claim holds for all $y \in \ball(x,r)$. The claim holds for all
balls such that $\ol{\ball(x,r)}\subset D_*$ so
$\zeta = \infty$, $\P_y$-a.s., for all $y\in D_*$.

\smallskip
(ii) This part follows easily from conformal invariance of Brownian motion killed
 upon leaving a domain.

\smallskip

(iii) This claim follows from the interpretation of the stationary distribution as the long time occupation measure, the definition of $\wh h$ and the ``clock'' $c(t)$.
We sketch the easy argument.
For an arbitrarily small $\eps>0$ and $x, y \in D_*$ we can find $r>0$ so small that
\begin{align*}
&\lim_{t\to \infty}
\frac{\int_0^t \bone_{\{Y_t \in \ball(f(x),r)\}} ds}
{\int_0^t \bone_{\{Y_t \in \ball(f(y),r))\}} ds}
\leq
\lim_{t\to \infty}
\frac{
\sup_{z\in f^{-1}(\ball(f(x),r))} |f'(z)|^2
\int_0^t \bone_{\{X_t \in f^{-1}(\ball(f(x),r))\}} ds}
{\inf_{z\in f^{-1}(\ball(f(y),r))} |f'(z)|^2
\int_0^t \bone_{\{X_t \in f^{-1}(\ball(f(y),r))\}} ds}\\
&\leq
\lim_{t\to \infty}
\frac{
\sup_{z\in f^{-1}(\ball(f(x),r))} |f'(z)|^2 (1+\eps)|f'(x)|^{-2}
\int_0^t \bone_{\{X_t \in \ball(x,r)\}} ds}
{\inf_{z\in f^{-1}(\ball(f(y),r))} |f'(z)|^2
(1-\eps)|f'(y)|^{-2}
\int_0^t \bone_{\{X_t \in \ball(y,r)\}} ds}\\
&\leq
\lim_{t\to \infty}
\frac{
\sup_{z\in f^{-1}(\ball(f(x),r))} |f'(z)|^2 (1+\eps)|f'(x)|^{-2}
\sup_{z \in \ball(x,r)} h(z)}
{\inf_{z\in f^{-1}(\ball(f(y),r))} |f'(z)|^2
(1-\eps)|f'(y)|^{-2}
\inf_{z \in \ball(y,r)} h(z)}.
\end{align*}
If we let $\eps,r\to 0$ then the right hand side converges to $h(x)/h(y)$. Hence, the limsup of the left hand side is at most  $h(x)/h(y)$. A similar argument shows that the liminf of the left hand side is at least $h(x)/h(y)$. This implies that the stationary density for $Y$ is proportional to $h \circ f ^{-1}$. Hence, it must be equal to $\wh h$.

\smallskip

(iv)
It follows from the definition of the ``clock'' $c(t)$ and the ergodic theorem that, a.s.,
\begin{align*}
\lim_{t\to \infty} \frac{c(t)}{t} =
\int_{D_*} |f'(x)|^2 h(x) dx
=\|\bar h\|_{L^1(D)}.
\end{align*}
We have already proved \eqref{j15.2}. That claim and the above formula imply for $z = f(0)$,
\begin{align}\label{m10.2}
\lim_{t\to \infty} \frac{\barg^* (Y_t -z) }{t}
&=
\lim_{t\to \infty} \frac{\barg^* X_{c^{-1}(t)} }{t}
=
\lim_{t\to \infty} \frac{\barg^* X_{t} }{c(t)}
=
\lim_{t\to \infty} \frac{\barg^* X_{t} }{t} \cdot
\frac t {c(t)} \\
&=
\lim_{t\to \infty} \frac{\barg^* X_{t} }{t}
\lim_{t\to \infty} \frac t {c(t)}
= \frac{\mu_0}{ \|\bar h\|_{L^1(D)}}
= \frac{\mu(0)}{ \|\bar h\|_{L^1(D)}}. \nonumber
\end{align}

Next we  prove \eqref{j16.1}.
Suppose that $f=\tau$ is a one-to-one analytic map of $D_*$ onto $D_*$
such that $\tau(0) = z$, as in Lemma \ref{ju21.1}. Then $\tau$ is a
M\"obius transformation. Let $\wh h = h \circ
\tau/  \| h\circ \tau \|_1$, $\wh \mu_0=\mu(z)/ \| h\circ \tau \|_1$, and $\wh \theta = \theta \circ
\tau $.
Then by Lemma \ref{ju21.1}, $\wh\theta \lra (\wh h, \wh \mu_0)$.
If $\bar h = \wh h \circ \tau^{-1}=h/\|h\circ \tau\|_1$ then
\begin{align*}
\|\bar h\|_1 = 1/ \|h\circ\tau\|_1 .
\end{align*}
By \eqref{m10.2}
\begin{align}\label{m10.1}
\lim_{t\to \infty} \frac{\arg^* (X_t-z)} t =
\frac{\wh\mu_0}{ \|\bar h\|_1} =
\wh\mu_0 \|h\circ\tau\|_1 = \mu(z).
\end{align}

Finally, we prove \eqref{f6.2} in full generality along the same lines as in \eqref{m10.2}. For any $z\in D$, by \eqref{j16.1},
\begin{align*}
&\lim_{t\to \infty} \frac{\barg^* (Y_t -z) }{t}
=
\lim_{t\to \infty} \frac{\barg^* (X_{c^{-1}(t)}-f^{-1}(z)) }{t}
=
\lim_{t\to \infty} \frac{\barg^* (X_{t} -f^{-1}(z))}{c(t)} \\
&=
\lim_{t\to \infty} \frac{\barg^* (X_{t} -f^{-1}(z))}{t} \cdot
\frac t {c(t)}
=
\lim_{t\to \infty} \frac{\barg^* (X_{t}-f^{-1}(z)) }{t}
\lim_{t\to \infty} \frac t {c(t)}
= \frac{\mu(f^{-1}(z))}{ \|\bar h\|_1}.
\end{align*}

(v) Let
$\theta$ correspond to $(h,\mu_0)$. Let $Y$ be constructed as in
\eqref{j18.4}-\eqref{feb5.2}. Then it is easy to see that $Y$
satisfies conditions (a) and (b) of part (v).

(vi)
This follows directly from the It\^o formula and Theorem \ref{j25.1}.
\end{proof}

\medskip

We now present an example showing that a conformal mapping may not always map an ORBM in one planar domain to another ORBM,
in the sense of Theorem \ref{j15.7}.

\begin{example}\label{E:4.2}

Let $S$ be a two-dimensional infinite wedge with corner at the origin $0$ and  angle $0<\alpha<2\pi$. Consider $\theta_1,\theta_2\in(-\pi/2,  \pi/2)$ and suppose that each $\theta_k$ represents the angle of reflection on one of the two sides of the wedge,
measured from the inward normal toward the origin $0$. In \cite{VW}, it was shown that there exists
a strong Markov process  that behaves like Brownian motion in the interior of the wedge and reflects instantaneously at the boundary with the oblique angle of reflection given by $\theta_k$.
This process,  called obliquely reflected Brownian motion in \cite{VW},
 is characterized as the unique solution to the corresponding submartingale problem away from the vertex.
It was shown \cite{VW} that the process enters 0 in a finite time and then stays there forever (i.e., it cannot be continued as a Markov process beyond that time)
if and only if  $\beta:=(\theta_1+\theta_2)/\alpha \geq 2$.
Let $D$ be an acute triangle obtained by truncation of the infinite wedge $S$.
Assume that $\theta_1$ and $\theta_2$ are such that  $\beta \geq 2$,
 set $\theta_3=0$ on the  edge opposite to $0$, and assume that
  the analogues of $\beta$ at the other two vertices are strictly less than 2. Let $f$ be a conformal mapping from the unit disk $D_*$ onto the Jordan domain $D$ and note that it extends to a homeomorphism
from $\overline D_*$ onto $\overline D$. Let $\theta (x)$ be the pre-image of the $\theta$-function on $\partial D$ by $f$.
Then $\theta$ is a piecewise constant function on $\partial D_*$ taking values in $(-\pi/2, \pi/2)$.
Thus by Theorem \ref{j15.5}, the ORBM $X$ in $D_*$ with reflection angle $\theta$ is a continuous, conservative Markov process having stationary distribution $h(x) dx$.
Consequently, $Z_t=f(X_t)$ is a continuous, conservative Markov process   on $\overline D$.
 The process $Z$ is an extension of killed Brownian motion in $D$ modulo a time change in the sense that for
every $t\geq 0$ and $\tau_t = \inf\{s\geq t: Z_s \in \prt D\}$, the process $\{Z_s, s\in [t, \tau_t)\}$ is a time change
of  Brownian motion killed upon exiting $D$. Let
$\wh \tau_t= \inf\{s\geq t: Z_s =0\}$ for $t\geq 0$.
Then the process $\{Z_s, s\in [t,  \wh \tau_t)\}$ is a time change
of  the obliquely reflected Brownian motion in $D$ killed upon hitting $0$.
More precisely,  let $x_0=f^{-1} (0)$, $\sigma_{x_0}=\inf\{t\geq 0: X_t=x_0\}$,
$c(t)=\int_0^t |f'(X_s)|^2 ds$ and $c^{-1}(t)=\inf\{s: c(s)>t\}$.
 Then $Y_t=f(X_{c^{-1}(t)})$, $t\in [0, \sigma_{x_0})$,
is  obliquely reflected Brownian motion in $D$ killed upon hitting $0$.
 The result in \cite{VW} and Theorem \ref{j15.7} imply that $c(\sigma_{x_0} ) = \int_0^{\sigma_{x_0}} |f'(X_s)|^2 ds <\infty $
 but $  \int_0^{\sigma_{x_0}+\eps} |f'(X_s)|^2 ds =\infty $ a.s. for every $\eps >0$,
 and that $h\circ f^{-1}\notin L^1(D)$.   \qed
\end{example}

\medskip

\begin{proof}[\bf Proof of Theorem \ref{j18.3}]
(i) The argument given in the proof of Theorem \ref{j15.7}(i)
which shows that $\zeta = \infty$, a.s., applies verbatim in
the present case because we have assumed that $\|\bar h\|_{L^1 (D)} <
\infty$.

Every harmonic function $h_k$ is bounded because $\theta_k$ is
continuous and takes values in $(-\pi/2, \pi/2)$. Hence, the
function $\bar h_k := h_k \circ f^{-1}$ is also bounded. Since
$D$ is bounded, it follows that $\|\bar h_k\|_{L^1(D)} < \infty$.
Once again, the argument given in the proof of Theorem \ref{j15.7}
(i) applies and shows that $\zeta_k = \infty$, a.s., for all
$k$.

(ii) Recall the representation of $X$ as the Poisson point
process on the space $\R_+ \times \calC_{D_*}$ (see Definition
\ref{jan22.1}). Excursion laws are conformally invariant in the
sense of the transformation in
\eqref{j18.4}-\eqref{feb5.2} by \cite[Prop.
10.1]{BurBook} so $Y$ can be represented as a Poisson point
process on $\R_+ \times \calC_{D}$. In other words, $Y$ is an ERBM
and it only remains to identify the corresponding $(\bar
\nu(dx), \bar H^x)_{x\in \prt D}$. We can arbitrarily set the
excursion intensity $\bar\nu$ to be $\bar \nu(A) =
\nu(f^{-1}(A)) $ for $A\subset \prt D$, in view of Remark
\ref{j3.1} (ii).

We will find the matching normalization for $\bar H^x$. Fix some $z\in D$ and
suppose that $r>0$ is very small. The Green function
$G_x(\,\cdot\,)$ in $D$ has the property that
\begin{align}\label{jan22.2}
\lim_{r\to 0} \frac{\inf_{y\in \prt \ball(z,r))} G_y(z)}
{\sup_{y\in \prt \ball(z,r))} G_y(z)} =
\lim_{r\to 0} \frac{\inf_{y\in \prt \ball(z,r))} G_y(z)}
{|\log r|} = 1.
\end{align}
Let $T_A$ denote the hitting time of $A$. Recall that
$G_x(\,\cdot\,)$ is the density of the expected occupation time
for Brownian motion in $D$ killed upon exiting from $D$. Also,
by Remark \ref{j3.1} (v), the density of the expected
occupation time for $\bar H^x$ is $\bar c_x K_x(\,\cdot\,)$. Hence,
for $x\in \prt D$, by the strong Markov property of $\bar H^x$,
\begin{align*}
\bar c_x K_x(z) = \int_{\prt \ball(z,r)} G_y(z) \bar H^x (X(T_{\prt \ball(z,r)}) \in dy).
\end{align*}
This and \eqref{jan22.2} imply that, as $r\to 0$,
\begin{align}\label{jan22.5}
|\log r| \bar H^x (T_{\prt \ball(z,r)}<\infty)
= \bar c_x K_x(z) + o(1) .
\end{align}
An analogous formula holds for excursion laws $H^x$ in $D_*$, with the corresponding constants $c_x$ equal to each other, by rotation invariance.
Let $N(dx, z, r, D,t)$ be the number of excursions of the ERBM in $D$ (here $D$ can be also $D_*$), which started from $dx\subset \prt D$ before time $t$ and hit $\prt \ball(z,r)$ before their lifetime. It is easy to see that
\begin{align}\label{jan22.4}
\lim_{r\downarrow 0, \eps\downarrow 0}
\lim_{t\to \infty}
\frac{N(dx, z, r, D,t)}{N(dx, z, r(1+\eps), D,t)} =1.
\end{align}
By the ergodic theorem,
\begin{align*}
\lim_{r\to 0}\lim_{t\to \infty}
\frac{N(dx, 0, r, D_*,t)}{N(dy, 0, r, D_*,t)}
\end{align*}
exists and is equal to $\nu(dx)/\nu(dy)$. The
fact that small balls are mapped by $f$ onto regions very close to balls, \eqref{jan22.4}, and the
definition of $Y$ as a transform of $X$ imply that for $Y$ we have
\begin{align*}
\lim_{r\to 0} \lim_{t\to \infty}
\frac{N(dx, f(0), r, D,t)}{N(dy, f(0), r, D,t)}
=\frac{\nu(f^{-1}(dx))}{\nu(f^{-1}(dy))}
= \frac{\bar\nu(dx)}{\bar\nu(dy)}.
\end{align*}
This in turn implies that all $\bar c_x$ in \eqref{jan22.5} must be equal to each other so, in view of Remark \ref{j3.1} (iii), we may take all of them to be equal to 1.

(iii) The processes $X^k$ converge to $X$ in the sense of finite dimensional distributions according to Theorem \ref{j18.1}.
A stronger assertion follows from the proof of that theorem. Fix some $\eps>0$ and let $\exc^{k,n}$ be the $n$-th excursion of the process $X^k$ which hits the ball $\ball(0, 1-\eps)$, and let $T^{k,n}_\eps$ be the hitting time of the ball. Then the joint distributions of $\{\exc^{k,n}_t, t\in [T^{k,n}_\eps, \zeta)\}$, $n\geq 1$, $\eps >0$, $\eps \in \Q$, converge as $k\to \infty$, in the Skorokhod topology. By the Skorokhod lemma, we can assume that
$\{\exc^{k,n}_t, t\in [T^{k,n}_\eps, \zeta)\}$, $n\geq 1$, $\eps >0$, $\eps \in \Q$, converge a.s., as $k\to \infty$, in the Skorokhod topology. Hence, $X^k_t \to X_t$ for almost all $t\geq 0$ simultaneously, a.s.

The function $f$ is Lipschitz continuous inside every disc $\ball(0,1- \rho)$, $\rho\in (0,1)$.
This implies that for every $\eps>0$ and $n$, the images of the excursions $f(\exc^{k,n}_t)$ converge as $k\to \infty$, a.s., in the Skorokhod topology over their lifetimes to the corresponding excursion of $Y$.
It will suffice to show that for every fixed $t>0$, the clocks $c_k(t)$ converge to $c(t)$ in probability (note that the clocks are monotone functions).

Let
\begin{align}
c(t) &= \int_0^t |f'(X_s)|^2 ds,  \qquad \text{  for  } t\geq 0, \nonumber \\
Y(t) &= f(X_{c^{-1}(t)}), \qquad \text{  for  } t\in[0,\infty), \label{m3.1}\\
c_k(t) &= \int_0^t |f'(X^k_s)|^2 ds,  \qquad \text{  for  } t\geq 0, \nonumber \\
Y^k(t) &= f\left(X^k_{c_k^{-1}(t)}\right), \qquad \text{  for  } t\in[0,\infty).\label{m3.2}
\end{align}
Then $Y$ and $Y^k$'s have distributions as specified in the statement of the theorem.

We will assume for a moment that $X^k_0$'s and $X_0$ have stationary distributions.
Let $D_\eps = D_* \setminus \ball(0, 1-\eps)$.
By assumption (i)
\begin{align}\label{m3.3}
\int_{D_*} |f'(x)|^2 h(x) dx = \int_D h\circ f^{-1} dx < \infty.
\end{align}
By assumption $D$ is bounded, so that $\int_{D_*} |f'|^2 dx= {\rm
Area}(D) < \infty$ and by
the proof of Theorem \ref{j18.1}, $h_k$ converges uniformly to $h$.
Thus
\begin{align}
\sup_k \int_{D_*} |f'(x)|^2 h_k(x) dx < \infty,
\label{m3.4}
\end{align}
and, moreover,
\begin{align}\label{m3.5}
\lim_{\eps \downarrow 0} &\int_{D_\eps} |f'(x)|^2 h(x) dx = 0,
\end{align}
and
\begin{align}
\lim_{\eps \downarrow 0} \sup_k &\int_{D_\eps} |f'(x)|^2 h_k(x) dx = 0.\label{m3.6}
\end{align}
For $\eps>0$ (suppressed in the notation), let
\begin{align*}
\bar c(t) &= \int_0^t |f'(X_s)|^2
\bone_{\{X_s \in D_\eps\}} ds, \qquad
\wh c(t)  = \int_0^t |f'(X_s)|^2
\bone_{\{X_s \in \ball(0, 1-\eps)\}} ds, \\
\bar c_k(t) &= \int_0^t |f'(X^k_s)|^2
\bone_{\{X^k_s \in D_\eps\}} ds, \qquad
\wh c_k(t)  = \int_0^t |f'(X^k_s)|^2
\bone_{\{X^k_s \in \ball(0, 1-\eps)\}} ds.
\end{align*}

Fix some $t\geq 0$ and arbitrarily small $p_1, \delta >0$. It follows from \eqref{m3.3}-\eqref{m3.6} that there exists $\eps_1>0$ such that for $\eps\in(0,\eps_1)$ and all $k$,
\begin{align*}
\E \left[\bar c(t)\right]
&= \E\left[\int_0^t |f'(X_s)|^2 \bone_{\{X_s \in D_\eps\}} ds\right]
= \int_0^t \E \left[|f'(X_s)|^2 \bone_{\{X_s \in D_\eps\}} \right] ds \\
& = \int_0^t \int_{D_\eps} |f'(x)|^2 h(x) dx ds
=  t \int_{D_\eps} |f'(x)|^2 h(x) dx ds < p_1 \delta,
\end{align*}
and
\begin{align*}
\E \bar c_k(t)
&= \E\int_0^t |f'(X^k_s)|^2 \bone_{\{X^k_s \in D_\eps\}} ds
= \int_0^t \E \left(|f'(X^k_s)|^2 \bone_{\{X^k_s \in D_\eps\}} \right) ds \\
& = \int_0^t \int_{D_\eps} |f'(x)|^2 h_k(x) dx ds
=  t \int_{D_\eps} |f'(x)|^2 h_k(x) dx ds < p_1 \delta.
\end{align*}
It follows that for $\eps\in(0,\eps_1)$ and all $k$,
\begin{align}\label{m3.7}
\P(\bar c(t) \geq \delta) \leq p_1
\qquad \text{  and  } \qquad
\P(\bar c_k(t) \geq \delta) \leq p_1.
\end{align}

For almost all $s>0$, $X^k_s \to X_s$, a.s.,
and $\P(X_s \in \prt \ball(0, 1-\eps)) = 0$. Hence,
for almost all $s>0$, a.s.,
\begin{align*}
\lim_{k\to \infty}
|f'(X^k_s)|^2
\bone_{\{X^k_s \in \ball(0, 1-\eps)\}}
= |f'(X_s)|^2
\bone_{\{X_s \in \ball(0, 1-\eps)\}},
\end{align*}
and, therefore, by the bounded convergence theorem, a.s.,
\begin{align*}
\lim_{k\to \infty} \wh c_k(t)
=  \lim_{k\to \infty}
\int_0^t |f'(X^k_s)|^2
\bone_{\{X^k_s \in \ball(0, 1-\eps)\}} ds
 = \int_0^t |f'(X_s)|^2
\bone_{\{X_s \in \ball(0, 1-\eps)\}} ds
= \wh c(t).
\end{align*}
This and \eqref{m3.7} imply that for every fixed $t>0$, a.s.,
\begin{align*}
\lim_{k\to \infty} c_k(t) =  c(t),
\end{align*}
because $\delta$ and $p_1$ can be chosen arbitrarily close to 0.

We can remove the assumption that the processes are in  the stationary distribution as in the proof of Theorem \ref{j15.7} (i).

(iv)
This can be proved just as part (iii)
of Theorem \ref{j15.7}.

(v)
Let $h^* = \wh h \circ f$. Then $h^*$ is a positive harmonic function in $D_*$ and so $\| h^*\|_1=\pi h^*(0)<\infty$.
 Let $h = h^*/\|h^*\|_1$. By assumption, $h$ is Lipschitz continuous on $\ol D_*$ and strictly positive on $\prt D_*$. Let
$h_k(z)= (1-2^{-k})^{1/2} h ((1-2^{-k} )z)$. Then $h_k$  is a sequence of positive harmonic functions in $D_*$ with $L^1$ norm equal to 1 and
  $C^2$ on $\ol D_*$, such that  $h_k\to h$ uniformly on compact subsets of $D_*$, and  both $h_k$ and $1/h_k$ are $\lambda$-Lipschitz on $\prt D_*$
  for some $\lambda >0$ when $k$ is sufficiently large.
Let $\mu_{0,k}=k$, and let $\theta_k$ correspond to $(h_k,\mu_{0,k})$. Let $Y^k$'s and $Y$ be constructed as in the statement of Theorem \ref{j18.3}. Then it is easy to see that the stationary distribution for ERBM $Y$ has density $\wh h$.
\end{proof}

\medskip

\begin{proof}[\bf Proof of Theorem \ref{j17.3}]
Let $D^k_* = f^{-1}(D_k)$. It is easy to see that $D^k_*$ converge to $D_*$ in the sense that for every $r<1$ there exists $k_0$ such that $ \ball(0,r) \subset D^k_*$ for $k\geq k_0$.
Set $x_0=f^{-1}(y_0)=f_k^{-1}(y_0)$,  $a_0=f(0)$ and $a_k=f_k(0)$.
Then $a_k \to a_0$.
Let $h_k= \bar h\circ f_k/\|\bar h \circ f_k\|_1=\bar h \circ f_k/(\pi
\bar h(a_k))$,
and let $\theta_k \lra ( h_k, \mu_0)$. Note that $h_k$ are smooth and bounded on $\ol
D_*$ and therefore $\theta_k$ are smooth on $\prt D_*$ and
take values in $(-\pi/2, \pi/2)$.
Let $h=\bar h\circ f /\|\bar h \circ f\|_1=\bar h \circ f /(\pi \bar h(a_0))$,
and let $\theta \lra ( h, \mu_0)$.
Then $h_k$ converges to $h$ uniformly on compact subsets of $D_*$ and
by \eqref{ju20.3}, ~$\theta_k(z)$ converges to $\theta(z)$ uniformly
on compact subsets of $D_*$. Since the closed unit ball in
$L^{\infty} (\prt D_*; |dx|)=L^1 (\prt D_*; |dx|)^*$ is compact  in the weak-* topology,
it follows that $\theta_k$ converges  to $\theta$ in the  in the
weak-* topology in  $L^{\infty}(\prt D_*; |dx|)$.
Let $X^k$ be the solution to \eqref{j13.1} corresponding to
$\theta_k$ and starting from $x_0 = f^{-1}(y_0)$ and let $X$ be constructed as in Theorem \ref{j15.5}, relative to $\theta$ and also
starting from $x_0 = f^{-1}(y_0)$. Let
\begin{align}
c(t) &= \int_0^t |f'(X_s)|^2 ds
 \quad \hbox{and} \quad
Y(t)  = f(X_{c^{-1}(t)})  \quad \text{for  } t\in[0,\infty), \label{feb11.2}\\
c_k(t) &= \int_0^t |f_k'(X^k_s)|^2 ds    \quad \hbox{and} \quad
Y^k(t) = f_k\left(X^k_{c_k^{-1}(t)}\right)  \quad \text{for  } t\in[0,\infty).\label{feb11.3}
\end{align}
Then $Y$ and $Y^k$'s have distributions as specified in the statement of the theorem.

We will assume for a moment that $X^k_0$'s and $X_0$ have stationary distributions.
According to Theorem \ref{j15.5} (i), the processes $\{X^k_s, 0\leq s\leq t\} $ converge weakly to $\{X_s, 0\leq s\leq t\} $ in $M^\calT_1$ topology.
By the Skorokhod theorem, we can assume that all these processes are defined on the same probability space and $\{X^k_s, 0\leq s\leq t\} $ converge almost surely to $\{X_s, 0\leq s\leq t\}$ in $M^\calT_1$ topology.

Let $D_\eps = D_* \setminus \ball(0, 1-\eps)$.
We have
\begin{align}\label{feb11.10}
&\int_{D_*} |f'(x)|^2 h(x) dx =\frac{1}{\pi \bar h(a_0)} \int_D \bar h dx < \infty, \\
\sup_k&\int_{D_*} |f'_k(x)|^2 h_k(x) dx = \frac{1}{\pi \bar h(a_k)} \sup_k \int_{D_k} \bar h dx < \infty,
\label{feb11.11}
\end{align}
and, moreover, as in \eqref{m3.5} and \eqref{m3.6}
\begin{align}\label{feb8.1}
\lim_{\eps \downarrow 0} &\int_{D_\eps} |f'(x)|^2 h(x) dx = 0, \\
\lim_{\eps \downarrow 0} \sup_k &\int_{D_\eps} |f'_k(x)|^2 h_k(x) dx = 0.\label{feb8.2}
\end{align}
For $\eps>0$ (suppressed in the notation), let
\begin{align*}
\bar c(t) &= \int_0^t |f'(X_s)|^2
\bone_{\{X_s \in D_\eps\}} ds, \qquad
\wh c(t)  = \int_0^t |f'(X_s)|^2
\bone_{\{X_s \in \ball(0, 1-\eps)\}} ds, \\
\bar c_k(t) &= \int_0^t |f_k'(X^k_s)|^2
\bone_{\{X^k_s \in D_\eps\}} ds, \qquad
\wh c_k(t)  = \int_0^t |f_k'(X^k_s)|^2
\bone_{\{X^k_s \in \ball(0, 1-\eps)\}} ds.
\end{align*}

Fix some $t\geq 0$ and arbitrarily small $p_1, \delta >0$. It follows from \eqref{feb8.1}-\eqref{feb8.2} that there exists $\eps_1>0$ such that for $\eps\in(0,\eps_1)$ and all $k$,
\begin{align*}
\E \left[\bar c(t)\right]
&= \E\left[\int_0^t |f'(X_s)|^2 \bone_{\{X_s \in D_\eps\}} ds\right]
= \int_0^t \E \left[|f'(X_s)|^2 \bone_{\{X_s \in D_\eps\}} \right] ds \\
& = \int_0^t \int_{D_\eps} |f'(x)|^2 h(x) dx ds
=  t \int_{D_\eps} |f'(x)|^2 h(x) dx ds < p_1 \delta,
\end{align*}
and
\begin{align*}
\E \left[\bar c_k(t)\right]
&= \E\left[\int_0^t |f'_k(X^k_s)|^2 \bone_{\{X^k_s \in D_\eps\}} ds\right]
= \int_0^t \E \left[|f'_k(X^k_s)|^2 \bone_{\{X^k_s \in D_\eps\}} \right] ds \\
& = \int_0^t \int_{D_\eps} |f'_k(x)|^2 h_k(x) dx ds
=  t \int_{D_\eps} |f'_k(x)|^2 h_k(x) dx ds < p_1 \delta.
\end{align*}
It follows that for $\eps\in(0,\eps_1)$ and all $k$,
\begin{align}\label{feb11.1}
\P(\bar c(t) \geq \delta) \leq p_1
\qquad \text{  and  } \qquad
\P(\bar c_k(t) \geq \delta) \leq p_1.
\end{align}

For any fixed $\eps>0$,
there is $k_0\geq 1$ such that
 \begin{align}\label{feb11.12}
&\sup_{x\in \ball(0, 1-\eps)}
\left(
|f'(x)|^2 h(x) \lor \sup_{k\geq k_0}  |f'_k(x)|^2 h_k(x)
\right) < \infty.
\end{align}
For every fixed $s>0$, $X^k_s \to X_s$, a.s.,
and $\P(X_s \in \prt \ball(0, 1-\eps)) = 0$. Hence,
for every fixed $s>0$, a.s.,
\begin{align*}
\lim_{k\to \infty}
|f_k'(X^k_s)|^2
\bone_{\{X^k_s \in \ball(0, 1-\eps)\}}
= |f'(X_s)|^2
\bone_{\{X_s \in \ball(0, 1-\eps)\}},
\end{align*}
and, therefore, by the bounded convergence theorem, a.s.,
\begin{align*}
\lim_{k\to \infty} \wh c_k(t)
=  \lim_{k\to \infty}
\int_0^t |f_k'(X^k_s)|^2
\bone_{\{X^k_s \in \ball(0, 1-\eps)\}} ds
 = \int_0^t |f'(X_s)|^2
\bone_{\{X_s \in \ball(0, 1-\eps)\}} ds
= \wh c(t).
\end{align*}
This and \eqref{feb11.1} imply that for every fixed $t>0$, a.s.,
\begin{align}\label{feb11.4}
\lim_{k\to \infty} c_k(t)
=  c(t),
\end{align}
because $\delta$ and $p_1$ can be chosen arbitrarily close to 0.

It follows easily from the definition \eqref{n25.71} of convergence in  $M_1^\calT$ topology
and continuity of $f$  on $\ol D_*$ that convergence of
 $X^k$  to $X$ in $M_1^\calT$ topology implies convergence of  $f(X^k)$  to $f(X)$ in $M_1^\calT$ topology. This is because the transformation $f$ affects only the first components of the pairs
$(y_n(s),t_n(s))$ and $(y(s),t(s))$ in \eqref{n25.71}.
When the clocks are changed, the second components are affected as well. Then we use \eqref{feb11.4}
to conclude that
$Y^k$ converge to $Y$ in $M_1^\calT$ topology.

We can remove the assumption that the processes are in  the stationary distribution as in the proof of Theorem \ref{j15.7} (i).
\end{proof}

\medskip

\begin{proof}[\bf Proof of Theorem \ref{T:3.8}]
Take  a   sequence of $C^2$ functions   $\theta_k: \prt D_* \to (-\pi/2, \pi/2)$
that converges to $\theta \in \calT$ in weak-* topology as elements of the dual space of $L^1(\prt D_*)$.
Let $X^k$ be ORBM on $D_*$  that satisfies \eqref{n19.1}.
By Theorem  \ref{j15.5}(i), $X^k$ converges weakly in $M^\calT_1$-topology to $X$, so does $f(X^k)$ to $f(X)$.
Define
$$ c_k(t)= \int_0^t |f'(X^k_s)|^2 ds \quad \hbox{and} \quad c (t)= \int_0^t |f'(X_s)|^2 ds.
$$
By an argument similar to that proving  \eqref{feb11.4}, we can show that
$\lim_{k\to \infty} c_k(t)=c(t)$ a.s. for every fixed $t>0$.
Consequently by the argument as in the second to the last paragraph in the proof of Theorem \ref{j17.3},
 $f\left(X^k_{c_k^{-1}(t)}\right)$ converges weakly in $M^\calT_1$-topology to $f(X_{c^{-1}(t)})$.
It is easy to see that  $f(X_{c^{-1}(t)})$ has stationary distribution with density $\bar h$.
Since $f$ is smooth on $\ol D_k$ and $\theta_k \circ f^{-1}$ converges to $\theta \circ f^{-1} \in \calT$
in weak-* topology as elements of the dual space of $L^1(\prt D_*)$, it follows from Theorem \ref{j15.5}
that $f(X_{c^{-1}(t)})$ is the ORBM on $D_*$ with reflection angle $\theta \circ f^{-1} $.
 \end{proof}

\medskip

\begin{proof}[\bf Proof of Theorem \ref{j18.8}]
This theorem can be proved just like
Theorem \ref{j17.3}. All we have to check
is whether the following claims hold:
\eqref{feb11.10}, \eqref{feb11.11}, \eqref{feb8.1},
\eqref{feb8.2}, and \eqref{feb11.12}. They are all easily
seen to hold in the present context.
\end{proof}

\medskip

\begin{example}\label{feb15.1}
We will sketch an example of a bounded domain $D$, an oblique angle of reflection $\theta$ and the corresponding ORBM with a stationary measure whose density $h$ is not in $L^1(D)$. The construction is a typical fractal-type argument; a construction similar in spirit can be found in Section 4 of \cite{BBlife}. We will not supply a formal proof because it would require a lot of space and the claim is rather specialized.

Let $D_0 = (0,1)^2$, and for $k\geq 1$ and small $r_k\in(0, 2^{-k-2})$ (to be specified later), let
\begin{align*}
D_k &= \ball (  2^{-k} -i 2^{-k}, 2^{-k -2}),\\
D_k' &= (2^{-k} - r_k, 2^{-k} + r_k) \times (- 2^{-k}  , 2^{-k}),\\
D &= D_0 \cup \bigcup_{k \geq 1} (D_k \cup D_k').
\end{align*}
The boundary $\prt D$ is smooth except for a countable number of points. We will specify the reflection angle relative to the inward normal vector $\n$ at each boundary point where $\n$ is well defined. For all points
 $x\in \prt D \cap (\prt D_0 \cup \prt D_k)$,
$k\geq 0$, we let $\theta (x) = 0$. In other words, the reflection is in the normal direction at the points on the boundary of the square $D_0$ and on the (arcs of the) circles $\prt D_k$.

It remains to define the angle of reflection for the part of $\prt D$ which lies on the sides of very thin channels $D_k'$. To make the example simple, we let the
angle of reflection be $\pi/2$ or $-\pi/2$, at $x\in \prt D \cap \prt D_k'$, $k\geq 1$, so that the reflected process is pushed down towards $D_k$. It would be more accurate to say that the process is teleported to $D_k$ if it hits  the side of a channel $\prt D \cap \prt D_k'$
because it has a jump that takes it to $\prt D_k$.

Heuristically speaking, the ratio of the average amounts of time spent by ORBM in $D_k$ and $D_0$ can be made arbitrarily large by making $r_k$ sufficiently small. The reason is that ORBM will jump to $D_k$ when it hits the boundary of $D_k'$. Going the other way is much harder---the process has to go though the very thin channel connecting $D_k$ and $D_0$ without hitting the sides of the channel. Let $a_k$ be the ratio of the average amounts of time spent by ORBM in $D_k$ and $D_0$. If we make all $a_k \geq 1$ then $\sum_{k\geq 1} a_k = \infty$ and it follows that there is no stationary probability distribution for ORBM. Every stationary measure has to have infinite mass.

It is clear that the ORBM described above is well defined as long as it does not hit $(0,0)$. An elementary argument can be used to show that the ORBM will not hit $(0,0)$ at a finite time, a.s., if we make the channels sufficiently thin (i.e., $r_k$'s sufficiently small).
\qed
\end{example}

\medskip

\begin{proof}[\bf Proof of Theorem \ref{j17.10}]
Parts (i) and (ii) are special cases of Theorems 1 and 2 of \cite{Aikawa}.

For part (iii),
let $D$ be the image of the unit disk by the map $F(z)=\sqrt{1-z}$
and let $h(w)=\Re((1+z)/(1-z))$ where $z=F^{-1}(w)$.
Then for the region $C$ in the disk given by $1-|z|^2>|1-z|$
(an approximate cone),
\begin{align*}
\int_D h(w) dw=\int_C \Re ((1+z)/(1-z)) |F'(z)|^2 dz \ge \int_C | 1-z|^{-2} dz/4,
\end{align*}
since $\Re ((1+z)/(1-z))= (1-|z|^2)/|1-z|^2$. This latter integral is infinite
by integrating in polar coordinates centered at $z=1$.
\end{proof}

\medskip

\noindent{\bf Acknowledgments}.
We are grateful to Gerald Folland, Jean-Francois Le Gall, Yoichi Oshima, Uwe Schmock and S.~R.~S.~Varadhan for very useful advice.

\bibliographystyle{amsplain}
\bibliography{oblique}

\end{document}